\documentclass[a4paper,10pt]{article}
\usepackage{amsmath, amsthm, amssymb, amsfonts, mathrsfs, amscd, color, graphicx}

\definecolor{highlight}{rgb}{1,0,.6}
\definecolor{extlinkz}{rgb}{0,0,1}
\definecolor{intlinkz}{rgb}{1,0,0}
\definecolor{citlinkz}{rgb}{0,.6,.1}

\usepackage[pdfpagelabels,plainpages=false,colorlinks=true,citecolor=citlinkz,linkcolor=intlinkz,urlcolor=extlinkz,urlbordercolor={0 0 0}]{hyperref}

\newtheorem{theo}{Theorem}
\newtheorem{defi}[theo]{Definition}
\newtheorem{lem}[theo]{Lemma}
\newtheorem{prop}[theo]{Proposition}
\newtheorem{cor}[theo]{Corollary}
\newtheorem{rem}[theo]{Remark}

\newcommand{\wh}[1]{\widehat{#1}}
\newcommand{\wt}[1]{\widetilde{#1}}
\newcommand{\ol}[1]{\overline{#1}}
\newcommand{\myatop}[2]{\genfrac{}{}{0pt}{}{#1}{#2}}

\newcommand{\R}{\mathbb{R}}
\newcommand{\Z}{\mathbb{Z}}
\newcommand{\T}{\mathbb{T}}
\newcommand{\C}{\mathbb{C}}
\newcommand{\N}{\mathbb{N}}

\newcommand{\Hil}{\mathcal{H}}
\newcommand{\Kil}{\mathcal{K}}

\newcommand{\meas}{\mathcal X}
\newcommand{\M}{\mathcal{M}}
\newcommand{\ALG}{\mathscr{M}}
\newcommand{\E}{\mathfrak{X}}

\newcommand{\ind}{\mathcal{I}}

\newcommand{\psis}[1]{\langle #1\rangle_\Gamma}

\newcommand{\vsp}{\textnormal{span}}
\newcommand{\osp}{\textnormal{outerspan}}
\newcommand{\csp}{\ol{\textnormal{span}}}
\newcommand{\Ran}{\textnormal{Ran}}
\newcommand{\Ker}{\textnormal{Ker}}

\newcommand{\rr}{\rho}
\newcommand{\vn}{\mathscr{R}}
\newcommand{\lr}{\lambda}
\newcommand{\vnL}{\mathscr{L}}

\newcommand{\id}{{\mathrm{e}}}
\newcommand{\orb}{{\mathcal{O}_{\Gamma}(\psi)}}
\newcommand{\g}{{g_{{}_\orb}}}

\newcommand{\Sy}{\textnormal{T}}
\newcommand{\mSy}{\wt\Sy}
\newcommand{\An}{\Sy^*}
\newcommand{\mAn}{\mSy{}^*}
\newcommand{\Grop}{\mathfrak{G}}
\newcommand{\mGrop}{\wt{\Grop}}
\newcommand{\Gram}{\mathcal{G}}
\newcommand{\Frame}{\mathfrak{F}}
\newcommand{\mFrame}{\wt{\mathfrak{F}}}
\newcommand{\Id}{\mathbb{I}}
\newcommand{\Proj}{\mathbb{P}}
\newcommand{\proj}{\mathcal{P}}
\newcommand{\qroj}{\mathcal{P}}

\newcommand{\iso}{\mathcal T}
\newcommand{\isom}[1]{\iso[#1](x)}

\newcommand{\ipr}[2]{\{#1,#2\}}
\newcommand{\norm}[1]{\left\|#1\right\|_{\oplus}}
\newcommand{\normbig}[1]{\bigg\|#1\bigg\|_{\oplus}}

\title{Noncommutative Shift-Invariant Spaces}
\author{Davide Barbieri, Eugenio Hern\'andez, Victoria Paternostro}

\begin{document}

\maketitle

\begin{abstract}
The structure of subspaces of a Hilbert space that are invariant under unitary representations of a discrete group is related to a notion of Hilbert modules endowed with inner products taking values in spaces of unbounded operators. A theory of reproducing systems in such modular structures is developed, providing a general framework that includes fundamental results of shift-invariant spaces. In particular,
general characterizations of Riesz and frame sequences associated to group representations are provided, extending previous results for abelian groups and for cyclic subspaces of unitary representations of noncommutative discrete groups.
\end{abstract}

{
\hypersetup{linkcolor=black}
\tableofcontents
}

\newpage

\section{Introduction}

The aim of this paper is to study the structure of subspaces of a Hilbert space that are invariant under the action of unitary representations of discrete groups, and to characterize Riesz bases and frames generated by these representations. This is a central problem in the study of shift-invariant spaces, whose best known examples are provided by those subspaces of $L^2(\R^d)$ that are invariant under integer translations (see e.g. \cite{Helson64, BoorVoreRon94, RonShen95, BenedettoLi93, BenedettoLi98, Bownik00, HLWW02, RonShen05, ChristensenEldar05}). Shift-invariant spaces have been extensively used in connection with approximation theory, wavelets and multiresolution analysis \cite{BoorVoreRon93, DevoreLorentz, DaubechiesHanRonShen03, RonShen05} or with sampling and interpolation theory \cite{AldroubiGrochenig01, GarciaMedinaVillalon08, NashedSun10, FuhrXian14}, and have multiple applications in image processing and data analysis \cite{CaiDongOsherShen10, Mallat12} and learning \cite{AldroubiCabrelliMolter08}.

Recently, effort has been devoted towards the study of Hilbert spaces invariant under the action of wider classes of groups, both with respect to locally compact abelian (LCA) groups \cite{HSWW10, CabrelliPaternostro10, CabrelliPaternostro11, BownikRoss13, JakobsenLemvig2014, Zak-SIS, Iverson14, Saliani14} and with respect to nonabelian groups \cite{CurreyMayeliOussa14, BHM14, BHP14}. A key role to address shift-invariance is played by Fourier analysis, that in the LCA setting relies on the abstract notion of Pontryagin duality. In order to deal with nonabelian discrete groups, in \cite{BHP14} it was proposed to make use of the Fourier duality provided by the group von Neumann algebra \cite{EnockSchwartz92, Connes94}, and to work with the associated noncommutative operator spaces. On the other hand, general noncommutative notions of Riesz bases and frames in the setting of Hilbert $C^*$-modules were introduced in \cite{FrankLarson2002}, and the interplay of questions on reproducing systems in Hilbert spaces with problems of operator algebras and noncommutative geometry is attracting an increasing interest \cite{Larson97, DaiLarson98, PackerRieffel04, Wood04, DutkayHanLarson09, Luef11}.

In this paper we will provide a general characterization in the spirit of \cite{Bownik00}, by reducing the reproducing conditions in subspaces of a Hilbert space that are invariant under unitary actions of discrete groups to equivalent reproducing conditions in Hilbert modules, expressed in terms of their generators. Such Hilbert modules belong to the class introduced in \cite{JungeSherman05}, and their inner product is defined by the bracket map \cite{BHP14}, for which a new relationship with the Gram operator will be given. General isomorphisms between the original Hilbert spaces and the associated Hilbert modules will be provided by noncommutative analogues of the fiberization mapping and of the Zak transform. Moreover, a full theory of reproducing systems in such Hilbert modules will be developed, in the spirit of \cite{FrankLarson2002}, which is able to deal with the natural appearance of unbounded operators.

\vspace{-4pt}

\subsection{Setting}

\vspace{-2pt}

We will consider the general setting of a discrete and countable group $\Gamma$, and of a unitary representation $\Pi$ on a separable Hilbert space $\Hil$, that is, $\Pi$ is a group homomorphism of $\Gamma$ into the group of linear unitary operators over $\Hil$. A closed subspace $V$ of $\Hil$ will be called $(\Gamma,\Pi)$-\emph{invariant} if $\Pi(\gamma)V \subset V$ for all $\gamma \in \Gamma$. It is easy to see, and it will be proved in Section 3, that for any such space there exists a countable set of \emph{generators} $\{\psi_j\}_{j \in \ind} \subset \Hil$, which means that $V$ can be obtained as the closed linear span of the $(\Gamma,\Pi)$-orbits of $\{\psi_j\}_{j \in \ind}$:\vspace{-2pt}
\begin{equation}\label{eq:shiftinvariantspace}
V = \ol{\vsp\{\Pi(\gamma)\psi_j : \gamma \in \Gamma, j \in \ind\}}^\Hil .\vspace{-2pt}
\end{equation}
The main problem that we will address is then to characterize when a countable family $\{\Pi(\gamma)\psi_j : \gamma \in \Gamma, j \in \ind\}$ is a Riesz basis or a frame for its closed linear span, providing conditions that extend all previous characterizations.
\newpage
Closed subspaces of $L^2(\R^d)$ that are invariant under integer translations $\Pi(k)f(x) = f(x - k)$, with $k \in \Z^d$, were studied in the influential works \cite{BoorVoreRon93, BoorVoreRon94, RonShen05, Bownik00} in terms of range functions and Gramian analysis, based on the linear isometric isomorphism, which we refer to as \emph{fiberization mapping}, given by
\begin{equation}\label{eq:Fourierperiodization}
\begin{array}{rccl}
\iso : & L^2(\R^d) & \to & L^2(\T^d,\ell_2(\Z^d))\vspace{4pt}\\
& f & \mapsto & \Big\{\wh{f}(\cdot + k)\Big\}_{k \in \Z^d}
\end{array}
\end{equation}
where $\wh{f}(\xi) = \int_{\R^d} f(x) e^{-2\pi i x \cdot \xi} dx$ is the usual Fourier transform.
One of the most relevant result from this approach is the possibility to characterize the structure of Riesz and frame sequences of countably generated systems of translates in $L^2(\R^d)$ in terms of the behavior of the image under $\iso$ of the set of generators. More precisely, by \cite[Th. 2.3]{Bownik00} one has that the system $E = \{\psi_j(\cdot - k) \, : \, k \in \Z^d, j \in \ind\}$ is a frame (Riesz) sequence with bounds $0 < A \leq B < \infty$ if and only if the system $\mathcal{E} = \{\iso\psi_j(\alpha) \, : \, j \in \ind\}$ is a frame (Riesz) sequence with uniform bounds $0 < A \leq B < \infty$ for a.e. $\alpha \in \T^d$.

This type of results were then extended for the first time to general LCA groups in \cite{CabrelliPaternostro10, KamyabiRaisi08} with respect to translations by discrete subgroups, while nondiscrete cases as well as the relationship with multiplicatively invariant spaces have been shown to fit within this structure in \cite{BownikRoss13}, as well as more general actions of LCA groups \cite{Zak-SIS, Iverson14}. A nonabelian notion of range function, adapted to the structure of translations by discrete subgroups of nilpotent Lie groups that are square-integrable modulo the center, have also been introduced in \cite{CurreyMayeliOussa14} in terms of the Fourier duality provided by the unitary dual of the Lie group (not of the discrete subgroup), which is parametrized by its center.

Our main tool to consider the problem in full generality will be the noncommutative theory of the bracket map developed in \cite{BHP14} in terms of Fourier analysis over group von Neumann algebras, that we briefly recall here (see also \cite{KadisonRingrose83, Conway00, Wegge93, Connes94, Takesaki03, PisierXu03} and the discussion in \cite{BHP14}). In this paper we will consider the \emph{right} von Neumann algebra of $\Gamma$, defined as follows. Let $\rr : \Gamma \to U(\ell_2(\Gamma))$ be the right regular representation, which acts on the canonical basis $\{\delta_\gamma\}_{\gamma \in \Gamma}$ as $\rr(\gamma)\delta_{\gamma'} = \delta_{\gamma'\gamma^{-1}}$, and let us call \emph{trigonometric
polynomials} the operators obtained by finite linear combinations of $\{\rr(\gamma)\}_{\gamma \in \Gamma}$. The right von Neumann algebra can be defined as the weak operator closure of such trigonometric polynomials
$$
\vn(\Gamma) = \ol{\vsp\{\rr(\gamma)\}_{\gamma \in \Gamma}}^{\textrm{WOT}} .
$$
Given $F \in \vn(\Gamma)$, we will denote by $\tau$ the standard trace
$$
\tau(F) = \langle F \delta_\id, \delta_\id \rangle_{\ell_2(\Gamma)}
$$
where $\id$ denotes the identity element of $\Gamma$, so that $\tau$ defines a normal, finite and faithful tracial linear functional.

Any $F \in \vn(\Gamma)$ has a \emph{Fourier series} 
$$
F = \sum_{\gamma \in \Gamma} \wh{F}(\gamma) \rr(\gamma)^* \, ,
$$
which converges in the weak operator topology, where the Fourier coefficients $\wh{F} \in \ell_2(\Gamma)$ of $F$ are given by
\begin{equation}\label{eq:Fouriercoefficients}
\wh{F}(\gamma) = \tau(F \rr(\gamma)) .
\end{equation}
Any $F \in \vn(\Gamma)$ is a bounded right convolution operator by $\wh{F}$: given $u \in \ell_2(\Gamma)$
$$
Fu = u \ast \wh{F}
$$
where $\ast$ stands for $\Gamma$ group convolution
$$
u \ast v (\gamma) = \sum_{\gamma' \in \Gamma} u(\gamma') v(\gamma'^{-1}\gamma) .
$$
The algebra $\vn(\Gamma)$ corresponds to the right convolution algebra over $\ell_2(\Gamma)$, i.e. it is formed by the bounded right convolution operators over $\ell_2(\Gamma)$. Accordingly, we will denote with $\vnL(\Gamma)$ the \emph{left} von Neumann algebra of $\ell_2(\Gamma)$, that is generated by the left regular representation $\lr: \Gamma \to U(\ell_2(\Gamma))$, defined by $\lr(\gamma)\delta_{\gamma'} = \delta_{\gamma\gamma'}$. The algebra $\vnL(\Gamma)$ coincides with the \emph{commutant} of $\vn(\Gamma)$, that is the algebra of bounded operators on $\ell_2(\Gamma)$ which commute with all the operators of $\vn(\Gamma)$ (more details can be found e.g. in \cite[\S 6.7]{KadisonRingrose83}).

For any $1 \leq p < \infty$ let $\|\cdot\|_p$ be the norm over $\vn(\Gamma)$ given by
$$
\|F\|_p = \tau(|F|^p)^\frac1p
$$
where the absolute value is the selfadjoint operator defined as $|F| = \sqrt{F^*F}$ and the $p$-th power is defined by functional calculus of $|F|$. Following \cite{Nelson74, PisierXu03, BHP14}, we define the noncommutative $L^p(\vn(\Gamma))$ spaces for $1 \leq p < \infty$ as
$$
L^p(\vn(\Gamma)) = \ol{\vsp\{\rr(\gamma)\}_{\gamma \in \Gamma}}^{\|\cdot\|_p}
$$
while for $p = \infty$ we set $L^\infty(\vn(\Gamma)) = \vn(\Gamma)$ endowed with the operator norm.

When $p < \infty$, the elements of $L^p(\vn(\Gamma))$ are the linear operators on $\ell_2(\Gamma)$ that are affiliated to $\vn(\Gamma)$, i.e. the densely defined closed operators that commute with all unitary elements of $\vnL(\Gamma)$, whose $\|\cdot\|_p$ norm is finite (see also \cite{Terp81}). In particular they are not necessarily bounded, while a bounded operator that is affiliated to $\vn(\Gamma)$ automatically belongs to $\vn(\Gamma)$ as a consequence of von Neumann's Double Commutant Theorem (see also \cite[Th. 4.1.7]{KadisonRingrose83}). For $p = 2$ one obtains a separable Hilbert space with scalar product
$$
\langle F_1, F_2\rangle_2 = \tau(F_2^* F_1)
$$
for which the monomials $\{\rr(\gamma)\}_{\gamma \in \Gamma}$ form an orthonormal basis.
For these spaces the usual statement of H\"older inequality still holds, so that in particular for any $F \in L^p(\vn(\Gamma))$ with $1 \leq p \leq \infty$ its Fourier coefficients are well defined, and the finiteness of the trace implies that $L^p(\vn(\Gamma)) \subset L^q(\vn(\Gamma))$ whenever $q < p$. Moreover, fundamental results of Fourier analysis such as $L^1(\vn(\Gamma))$ Uniqueness Theorem, Plancherel Theorem between $L^2(\vn(\Gamma))$ and $\ell_2(\Gamma)$, and Hausdorff-Young inequality still hold in the present setting (see e.g. \cite[\S 2.2]{BHP14}).

A relevant class of operators in von Neumann algebras are the orthogonal projections onto closed subspaces of $\ell_2(\Gamma)$. An orthogonal projection onto the subspace $W \subset \ell_2(\Gamma)$ belongs to $\vn(\Gamma)$ if and only if $W$ is an invariant space for the left regular representation\footnote{A proof is the following. Since $\vn(\Gamma) = \vnL(\Gamma)' = \lr(\Gamma)'$ we have that
$\Proj_W$ belongs to $\vn(\Gamma)$ if and only if \ $\Proj_W \lr(\gamma) = \lr(\gamma) \Proj_W$ for all $\gamma \in \Gamma$. Let us then first assume that $\lr(\Gamma)W \subset W$. Then also $W^\bot$ is invariant, because for all $\gamma \in \Gamma$, $w \in W$, $w' \in W^\bot$ it holds
$$
\langle w , \lr(\gamma) w'\rangle = \langle \lr(\gamma)^*w, w'\rangle = 0
$$
so $\lr(\gamma)w' \bot w$, and hence $\lr(\Gamma)W^\bot \subset W^\bot$. Then $\Proj_W \in \vn(\Gamma)$ because for all $u \in \ell_2(\Gamma)$
$$
\Proj_W \lr(\gamma) u = \Proj_W \lr(\gamma) \Proj_W u + \Proj_W \lr(\gamma) \Proj_{W^\bot} u = \lr(\gamma) \Proj_W u .
$$
Conversely, let $\Proj_W \in \vn(\Gamma)$. Then for all $w \in W$ we have $\lr(\gamma) w = \lr(\gamma) \Proj_W w = \Proj_W \lr(\gamma) w$.
}, i.e. $\lr(\Gamma)W \subset W$. A natural source for these projections are the spectral projections of any selfadjoint operator $F$ that is affiliated to $\vn(\Gamma)$ (see \cite[Prop. 5.3.4]{Sunder97}). A special case is the spectral projection over the set $\R \setminus \{0\}$, that is called the support of $F$. It is the minimal orthogonal projection $s_F$ of $\ell_2(\Gamma)$ such that $F = F s_F = s_F F$, and reads explicitly
$$
s_F = \Proj_{(\Ker(F))^\bot} = \Proj_{\ol{\Ran(F)}} .
$$

Finally, we would like to recall to the readers that are less accustomed to von Neumann algebra theory that, when the group $\Gamma$ is abelian, Pontryagin duality defines a Banach algebra isomorphism, that preserves the involution given by complex conjugation, between $\vn(\Gamma)$ and the usual $L^\infty(\wh{\Gamma})$ (see e.g. \cite{KadisonRingrose83, Connes94} or \cite[\S 5.1]{BHP14}). This isomorphism provides indeed the main motivation for considering Fourier analysis of discrete groups in terms of their group von Neumann algebra, and in particular it allows to translate almost flawlessly the results expressed in terms of operators on $\ell_2(\Gamma)$ into results expressed in terms of functions of $\wh{\Gamma}$, by the rule-of-thumb of replacing $\{\rr(\gamma)\}_{\gamma \in \Gamma}$ with the characters $\{\chi_\gamma : \wh{\Gamma} \to \T\}_{\gamma \in \Gamma}$.

Let us now recall the following definition from \cite{HSWW10, BHP14}.
\begin{defi}
Let $\Pi$ be a unitary representation of a discrete group $\Gamma$ on a separable Hilbert space $\Hil$. We say that $\Pi$ is \emph{dual integrable} if there exists a sesquilinear map $[\cdot,\cdot] : \Hil \times \Hil \to L^1(\vn(\Gamma))$, called \emph{bracket map}, satisfying
$$
\langle \varphi, \Pi(\gamma)\psi\rangle_\Hil = \tau([\varphi,\psi]\rr(\gamma)) \quad \forall \, \varphi, \psi \in \Hil \, , \ \forall \, \gamma \in \Gamma .
$$
In such a case we will call $(\Gamma,\Pi,\Hil)$ a \emph{dual integrable triple}.
\end{defi}
According to \cite[Th. 4.1]{BHP14}, $\Pi$ is dual integrable if and only if it is square integrable, in the sense that there exists a dense subspace $\mathcal{D}$ of $\Hil$ such that
$$
\big\{\langle \varphi, \Pi(\gamma)\psi\rangle_\Hil\big\}_{\gamma \in \Gamma} \in \ell_2(\Gamma) \quad \forall \, \varphi \in \Hil \, , \ \forall \, \psi \in \mathcal{D} .
$$
Moreover we recall that, by \cite[Prop. 3.2]{BHP14} the bracket map satisfies the properties
\begin{itemize}
\item[\textnormal{I)}] $[\psi_1, \psi_2]^* = [\psi_2, \psi_1]$
\item[\textnormal{II)}] $[\psi_1, \Pi(\gamma)\psi_2] = \rr(\gamma)[\psi_1, \psi_2]$ \ , \ \ $[\Pi(\gamma)\psi_1, \psi_2] = [\psi_1, \psi_2]\rr(\gamma)^*$ \ , \ \  $\forall \, \gamma \in \Gamma$
\item[\textnormal{III)}] $[\psi,\psi]$ is nonnegative, and $\|[\psi, \psi]\|_1 = \|\psi\|^2_{\Hil}$
\end{itemize}
for all $\psi, \psi_1, \psi_2 \in \Hil$.
Since we are using here a bracket map in terms of the right regular representation, we provide a proof of Property II). This is a consequence of the definition of the bracket map, the traciality of $\tau$ and the $L^1(\vn(\Gamma))$ uniqueness of Fourier coefficients. Indeed
\begin{align*}
\tau([\psi_1, \Pi(\gamma_0)\psi_2] \rr(\gamma)) & = \langle \psi_1 , \Pi(\gamma) \Pi(\gamma_0) \psi_2 \rangle_{\Hil} = \langle \psi_1 , \Pi(\gamma \gamma_0) \psi_2 \rangle_{\Hil}\\
& = \tau([\psi_1, \psi_2] \rr(\gamma \gamma_0)) = \tau([\psi_1, \psi_2] \rr(\gamma) \rr(\gamma_0))\\
& = \tau(\rr(\gamma_0)[\psi_1, \psi_2] \rr(\gamma)) \, , \quad \forall \, \gamma, \gamma_0 \in \Gamma .
\end{align*}
The other equality is proved from this result and Property I). This notion of bracket map will be our key object in this study, and we recall that it corresponds (by Pontryagin duality) to the notion introduced in \cite{HSWW10} in the LCA setting.
\newpage

\subsection{Results}

We will begin, in Section \ref{sec:basicresults}, by discussing some fundamental issues that lie at the basis of the problem of reproducing systems (Riesz sequences and frames) in $(\Gamma,\Pi)$-invariant spaces. There, we will prove a new basic relationship between the usual Gram operator and the bracket map, that clarifies the nature of this crucial object and allows us to reobtain known characterizations of Riesz and frame sequences for principal invariant subspaces in terms of general arguments and simple proofs (hence avoiding the - as the authors term them - ``surprisingly intricate'' arguments invoked in \cite{BenedettoLi98} and still present in all subsequent works). More precisely, let $\orb = \{\Pi(\gamma)\psi\}_{\gamma \in \Gamma}$ be the orbit of a vector $\psi \in \Hil$, let $\psis{\psi} = \ol{\vsp\,\orb}^\Hil$ be its linearly generated space, and let $\Grop_\orb$ be the linear operator on $\ell_2(\Gamma)$ associated to the Gram matrix
$$
\Gram(\gamma' , \gamma) = \langle \Pi(\gamma)\psi, \Pi(\gamma')\psi\rangle_\Hil \ , \quad \gamma,\gamma' \in \Gamma .
$$
Our first result, which can be deduced from Proposition \ref{prop:G=B}, is that this operator coincides with the bracket map for a large set of $\psi \in \Hil$, in the precise sense given by the following.
\begin{theo}\label{theo:bracket=Gram}
Let $(\Gamma,\Pi,\Hil)$ be a dual integrable triple, and let $\psi \in \Hil$ be such that $[\psi,\psi] \in L^2(\vn(\Gamma))$. Then
$[\psi,\psi] = \Grop_\orb$.
\end{theo}
Observe indeed that the set of $\psi \in \Hil$ such that $[\psi,\psi] \in L^2(\vn(\Gamma))$ is a dense set in $\Hil$, by \cite[Th. 4.1]{BHP14} and Plancherel theorem.
This result allows us, in particular, to prove the characterizations given in \cite[Th. A]{BHP14}, in terms of simple arguments based on the structure of the Gram operator (see Corollary \ref{cor:principalcharacterization}). We will also construct the Gram operator under weaker conditions than the ones generally considered in the literature, and this will allow us to explicitly characterize the Fourier coefficients of a noncommutative weighted Hilbert space that plays a central role in next sections, and that will be described at the beginning of Section \ref{sec:isometries}. 
In that section we will show that dual integrability is equivalent to the existence of a class of isometries, that we will call Helson maps, that are proper generalizations of the fiberization mapping (\ref{eq:Fourierperiodization}) and of the Zak transform.
\begin{defi}\label{def:Helsonmap}
Let $\Gamma$ be a discrete group and $\Pi$ a unitary representation of $\Gamma$ on the separable Hilbert space $\Hil$. We say that the triple $(\Gamma,\Pi,\Hil)$ admits a Helson map if there exists a $\sigma$-finite measure space $(\M,\nu)$ and an isometry
$$
\iso : \Hil \to L^2((\M,\nu),L^2(\vn(\Gamma)))
$$
satisfying
\begin{equation}\label{eq:Helsonproperty}
\iso[\Pi(\gamma)\varphi] = \iso[\varphi] \rr(\gamma)^* \quad \forall \, \gamma \in \Gamma , \ \forall \, \varphi \in \Hil .
\end{equation}
\end{defi}

Observe that for $\varPsi \in L^2((\M,\nu),L^2(\vn(\Gamma)))$ and $F \in \vn(\Gamma)$ we are denoting with $\varPsi F$ the element of $L^2((\M,\nu),L^2(\vn(\Gamma)))$ that for a.e. $x \in \M$ is given by
\begin{equation}\label{eq:rightaction}
(\varPsi F )(x) = \varPsi(x) F .
\end{equation}

The following equivalence between dual integrability and the existence of a Helson map will be proved in Section \ref{sec:isometries} on the basis of Propositions \ref{prop:propertyii} and \ref{prop:periodization}.

\begin{theo}\label{theo:DualHelsonEquivalence}
The triple $(\Gamma,\Pi,\Hil)$ is dual integrable if and only if it admits a Helson map.
\end{theo}

The role of Helson maps is captured, in Section \ref{sec:Hilbertmodules}, by Theorem \ref{theo:multiplicatively}, which shows that a closed subspace of $\Hil$ is $(\Gamma,\Pi)$-invariant if and only if its image under a Helson map is invariant under the right composition (\ref{eq:rightaction}) with elements of $\vn(\Gamma)$. This result generalizes the notion of \emph{multiplicatively invariant} space introduced in \cite{Helson86, BownikRoss13} to the present setting, and naturally leads to the notion of modules over $\vn(\Gamma)$, that are vector spaces endowed with a structure of linear combinations with noncommutative coefficients belonging to $\vn(\Gamma)$.

Furthermore, we will show that the $L^1(\vn(\Gamma))$-valued bracket map composed with the inverse of a Helson map $\iso$ endows such modules with an $L^2(\vn(\Gamma))$-Hilbert modular structure, as the one introduced in \cite{JungeSherman05}. More precisely, we can define an operator-valued inner product as a sesquilinear map on $\Kil = \iso(\Hil)$ as
\begin{equation}\label{eq:Kilinnerproduct}
\begin{array}{rccl}
\ipr{\cdot}{\cdot}_\Kil = [\cdot , \cdot]\circ \iso^{-1} : & \Kil \times \Kil & \to & L^1(\vn(\Gamma))\\
& (\varPhi,\varPsi) & \mapsto & \ipr{\varPhi}{\varPsi}_\Kil = [\iso^{-1}\varPhi, \iso^{-1}\varPsi] .
\end{array}
\end{equation}
This map provides an inner product that endows $\Kil$ and its submodules with an $L^2(\vn(\Gamma))$-Hilbert module structure, in the sense of \cite{JungeSherman05}, as follows.

\begin{theo}\label{theo:Hilmodshiftinv}
Let $(\Gamma,\Pi,\Hil)$ be a dual integrable triple with Helson map $\iso$, let $V$ be a closed subspace of $\Hil$, and let $M = \iso(V)$. Then $V$ is $(\Gamma,\Pi)$-invariant if and only if $(M,\ipr{\cdot}{\cdot}_\Kil)$ is an $L^2(\vn(\Gamma))$-Hilbert module with right composition (\ref{eq:rightaction}) with $\vn(\Gamma)$.
\end{theo}

Inner products taking values in noncommutative $L^p$ spaces were already considered in \cite{Popa86, Connes94, Falcone00}. They provide an extension of the theory of Hilbert modules over $C^*$-algebras, which dates back to \cite{Kaplansky53} and \cite{Paschke73, Rieffel74}, and that is clearly reviewed in the monographs \cite{Lance95, ManuilovTroitsky05} (see also \cite{FrankBiblio} for an extensive bibliography on Hilbert $C^*$-modules and their applications). As $L^2(\vn(\Gamma))$-Hilbert modules are far less known objects, we will develop in the rest of Section \ref{sec:Hilbertmodules} a theory of linear combinations that is needed for our purposes, i.e. with coefficients that do not necessarily belong to the von Neumann algebra but may be unbounded operators. Also, we will introduce new modular notions of synthesis operator, as the operator performing linear combinations with operator-valued coefficients, and of analysis operators, as the operator providing sequences of $L^1(\vn(\Gamma))$-valued modular inner products. In particular, the need of using the result of such a modular analysis as coefficients for a modular synthesis in order to define a modular frame operator, is what requires the provided extension of linear combinations to coefficients that lay out of the von Neumann algebra.

In Section \ref{sec:noncommutativereproducing} we will develop a theory of reproducing systems in $L^2(\vn(\Gamma))$-Hilbert modules, providing definitions of \emph{modular Riesz and frame sequences}. In particular, with Theorems \ref{theo:ncnumerical} and  \ref{theo:ncGramianCharacterization} we will recover the same characterizations, known in the usual Hilbert space setting, of modular Riesz and frame sequences in terms of bounds on the modular analogues of the frame and Gram operators. We will also prove a modular reproducing formula for modular frames with Theorem \ref{theo:reproducing}, in terms of a properly defined canonical modular dual frame. Finally, we will prove the following theorem, which shows that the condition for a system as in (\ref{eq:shiftinvariantspace}) to be a Riesz or a frame sequence is equivalent to a condition on the Helson map of the generators to form a modular Riesz or a modular frame sequence in the target Hilbert module.

\begin{theo}\label{theo:main}
Let $(\Gamma,\Pi,\Hil)$ be a dual integrable triple and let $\iso : \Hil \to \Kil$ be a Helson map. Let $\{\phi_j\}_{j \in \ind} \subset \Hil$ be a countable family, denote with $E$ the system
$$
E = \{\Pi(\gamma)\phi_j : \gamma \in \Gamma, j \in \ind\} \subset \Hil
$$
and with $\Phi$ the system
$$
\Phi = \{\iso[\phi_j] : j \in \ind\} \subset \Kil .
$$
Then
\begin{itemize}
\item[i.] The system $E$ is a Riesz sequence if and only if $\Phi$ is a modular Riesz sequence with the same Riesz bounds
\item[ii.] The system $E$ is a frame sequence if and only if $\Phi$ is a modular frame sequence with the same frame bounds.
\end{itemize}
\end{theo}

In Section \ref{sec:noncommutativereproducing} we will also discuss how this result directly contains several previous results concerning discrete groups, such as the characterizations of Riesz and frame sequences associated to integer translations of countable families in $L^2(\R^d)$ given in terms of the so-called \emph{Gramian} analysis, developed in the seminal works \cite{BoorVoreRon93, BoorVoreRon94, RonShen95, Bownik00}, the corresponding results for translations by discrete subgroups in the $L^2$ space of an LCA group in \cite{CabrelliPaternostro10, KamyabiRaisi08}, as well as results concerning single orbits of general unitary representations present in \cite{BenedettoLi93, HSWW10, BHP14}.

\

Finally, we would like to recall that the notions of Riesz bases and frames in Hilbert $C^*$-modules were introduced in \cite{FrankLarson2002}, where the authors considered inner products with values in a $C^*$-algebra. While several issues are common with that formulation, we remark that one of the main differences in the present setting is the need to deal with unbounded operators. We also recall that the idea of a bracket map as an algebra valued inner product was first implicitly considered in \cite{BoorVoreRon94}. A nice discussion on this respect, with many references, is contained in \cite{FrankLarson2002}. Finitely generated projective $C^*$-modules over commutative $C^*$-algebras were introduced for the investigation of multiresolution analysis wavelets already in \cite{PackerRieffel04}, and other related results are those of \cite{CasazzaLammers99, Wood04, Roysland11}.

\

Since the results of this work are addressed to, and hope to meet the interest of, different communities, we have tried to keep the exposition self-contained and to provide detailed proofs or precise references to all statements. In doing this we sometimes have also encountered the need to provide details on known issues that were not explicit in the literature, as well as to fully recall some basic definitions in order to help the readability of the text.

\paragraph{Acknowledgements.} D. Barbieri was supported by a Marie Curie Intra Euro\-pean Fellowship (Prop. N. 626055) within the 7th European Community Framework Programme. D. Barbieri and E. Hern\'andez were supported by Grant MTM2013-40945-P (Ministerio de Econom\'ia y Competitividad, Spain). V. Paternostro was supported by a fellowship for postdoctoral researchers from the Alexander von Humboldt Foundation and by Grants UBACyT  2002013010022BA and CONICET-PIP 11220110101018.
\newpage

\section{Hilbert space setting and unitary shifts}\label{sec:basicresults}

Let $\Hil$ be a separable Hilbert space. For a countable set of indices $\ind$ consider the family  $\Psi = \{\psi_j\}_{j \in \ind} \subset \Hil$ and call $\Hil_\Psi = \ol{\vsp(\Psi)}^\Hil$. The following standard definitions can be found e.g. in \cite{Daubechies92, Meyer93, HW96, Christensen03}. The system $\Psi$ is said to be a \emph{Riesz sequence} with Riesz bounds $0 < A \leq B < \infty$ if it satisfies
\begin{equation}\label{eq:Riesz}
A \|c\|^2_{\ell_2(\ind)} \leq \|\sum_{j \in \ind} c_j\psi_j\|^2_{\Hil} \leq B \|c\|^2_{\ell_2(\ind)}
\end{equation}
for all finite sequence $c = \{c_j\}_{j \in \ind} \in \ell_0(\ind)$. Since finite sequences are dense in $\ell_2(\ind)$, this condition is equivalent to say that (\ref{eq:Riesz}) holds for all $c \in \ell_2(\ind)$. 
The system $\Psi$ is a \emph{frame sequence} with frame bounds $0 < A \leq B < \infty$ if it satisfies
\begin{equation}\label{eq:frames}
A \|\varphi\|^2_{\Hil} \leq \sum_{j \in \ind} |\langle \varphi, \psi_j \rangle_{\Hil} |^2 \leq B \|\varphi\|^2_{\Hil}
\end{equation}
for all $\varphi \in \Hil_\Psi$. Recall also that a Riesz sequence is a \emph{Riesz basis for} $\Hil_\Psi$, and a frame sequence is a \emph{frame for} $\Hil_\Psi$. The system $\Psi$ is a \emph{Bessel sequence} if the right inequality in (\ref{eq:frames}) holds for all $\varphi \in \Hil$.

\subsection{Fundamental operators}

In this subsection we will provide some foundational definitions and results concerning the key operators involved in the study of Riesz bases and frames. We will start with weak conditions on a countable family $\Psi = \{\psi_j\}_{j \in \ind} \subset \Hil$, of norm boundedness and square integrability, and then recall some standard facts under the stronger Bessel condition.

Without any assumption on $\Psi$ one can define the \emph{synthesis operator}, that is a densely defined operator from $\ell_2(\ind)$ to $\Hil$ which, on finite sequences, reads
$$
\begin{array}{rccc}
\Sy_\Psi : & \ell_0(\ind) & \rightarrow & \vsp(\Psi) \subset \Hil\vspace{4pt}\\
& c & \mapsto & \displaystyle\sum_{j \in \ind} c_j \psi_j .
\end{array}
$$

\subsubsection{Norm bounds and the domain of the synthesis operator}\label{sec:boundedness}

Let us consider a family $\Psi = \{\psi_j\}_{j \in \ind}$ that is uniformly bounded, i.e. such that
\begin{equation}\label{eq:hpunifbound}
\|\psi_j\|_\Hil \leq C < \infty \quad \textnormal{for all} \ j \in \ind .
\end{equation}
This condition allows to define other relevant operators and to characterize in a simple way the domain of the synthesis operator.

Under hypothesis (\ref{eq:hpunifbound}) the synthesis operator  $\Sy_\Psi$ extends to a bounded operator $\Sy_\Psi: \ell_1(\ind) \to \Hil$, because
$$
\|\Sy_\Psi c\|_\Hil \leq \displaystyle\sum_{j \in \ind} |c_j| \|\psi_j\|_\Hil \leq C \|c\|_{\ell_1(\ind)} .
$$
Furthermore, one can define its dual operator $\An_\Psi : \Hil^* \rightarrow (\ell_1(\ind))^*$ which, when applied to $\varphi \in \Hil \approx \Hil^*$, acts linearly on $c \in \ell_1(\ind)$ as
$$
\An_\Psi(\varphi)(c) = \langle \Sy_\Psi c, \varphi\rangle_\Hil = \langle \sum_{j \in \ind} c_j \psi_j, \varphi\rangle_\Hil = \sum_{j \in \ind} c_j \langle \psi_j, \varphi\rangle_\Hil .
$$
It is called \emph{analysis operator} associated to $\Psi$, and reads
$$
\begin{array}{rccc}
\An_\Psi : & \Hil & \rightarrow & \ell_{\infty}(\ind)\vspace{4pt}\\
& \varphi & \mapsto & \big\{\langle\varphi,\psi_j\rangle_\Hil\big\}_{j \in \ind}
\end{array}
$$
where we have used that $(\ell_1(\ind))^* \approx \ell_{\infty}(\ind)$. The operator $\An_\Psi$ is bounded from $\Hil$ to $\ell_{\infty}(\ind)$.

Assuming (\ref{eq:hpunifbound}) by composition of analysis and synthesis operators one obtains the \emph{Gram operator} associated to $\Psi$ as a bounded operator from $\ell_1(\ind)$ to $\ell_\infty(\ind)$ that reads
$$
\begin{array}{rccc}
\Grop_\Psi = \An_\Psi \Sy_\Psi : & \ell_1(\ind) & \rightarrow & \ell_\infty(\ind)\vspace{4pt}\\
& c & \mapsto & \displaystyle\Big\{\langle \sum_{j \in \ind} c_j \psi_j,\psi_k\rangle_\Hil\Big\}_{k \in \ind}
\end{array}
$$
whose name is motivated by the observation that
$$
\Big(\Grop_\Psi c\Big)_k = \sum_{j \in \ind} c_j \langle \psi_j,\psi_k\rangle_\Hil  = \sum_{j \in \ind} \Gram_\Psi^{k,j} c_j \ , \quad k \in \ind
$$
where $\Gram_\Psi = (\Gram_\Psi^{k,j}) = (\langle \psi_j, \psi_k\rangle_\Hil)$ is the Gram matrix of $\Psi$ (note the ordering of indices). Such an operator allows to introduce a positive semi-definite Hermitian form over $\ell_1(\ind)$, defined for $f, g \in \ell_1(\ind)$ by
\begin{equation}\label{eq:positiveform}
\langle f, g\rangle_{\Grop_\Psi} = \sum_{k \in \ind} \big(\Grop_\Psi f\big)_k \,\ol{g_k} .
\end{equation}
It is indeed finite, because $g \in \ell_1(\ind)$ and $\Grop_\Psi f \in \ell_\infty(\ind)$, and positive, because
$$
\langle f, f\rangle_{\Grop_\Psi} = \|\Sy_\Psi f\|_\Hil^2 \geq 0 \quad \forall \, f \in \ell_1(\ind).
$$
If we let $\mathcal{N} = \{f \in \ell_1(\ind) : \|f\|_{\Grop_\Psi} = 0\}$ to be the null space of the norm associated to the Hermitian form (\ref{eq:positiveform}), we can obtain a Hilbert space as
\begin{equation}\label{eq:RKHS}
\Hil_{\Grop_\Psi} = \ol{\ell_1(\ind)/\mathcal{N}}^{\|\cdot\|_{\Grop_\Psi}} .
\end{equation}
This construction is analogous to the Mercer-Moore-Aronszajn construction of reproducing kernel Hilbert spaces \cite{Aronszajn50, Saitoh97}, and we will later consider it in the context of group theory where it is associated to the functions of positive type (see e.g. \cite[\S 3.3]{Folland95}). A useful consequence that one can obtain is the following.
\begin{lem}\label{lem:synthesisdomain}
Let $\Psi = \{\psi_j\}_{j \in \ind} \subset \Hil$ be a countable family satisfying (\ref{eq:hpunifbound}). Then the synthesis operator associated to $\Psi$ defines a surjective isometry
$$
\Sy_\Psi : \Hil_{\Grop_\Psi} \to \Hil_\Psi \ .
$$
\end{lem}
\begin{proof}
Let $c \in \ell_1(\ind)$ and call $\varphi = \Sy_\Psi c = \displaystyle\sum_{j \in \ind} c_j \psi_j \in \Hil_\Psi$.
Then
$$
\|\varphi \|^2_\Hil = \sum_{j, k \in \ind} c_j \ol{c_k} \langle \psi_j, \psi_{k} \rangle_\Hil = \|c\|^2_{\Grop_\Psi}
$$
so that the isometry is proved by density.
In order to prove surjectivity, assume that there exists $\varphi \in \Hil_\Psi$ such that
$\langle \varphi, \Sy_\Psi c\rangle_\Hil = 0$ for all $c \in \Hil_{\Grop_\Psi}$. This means in particular that $\varphi$ is orthogonal to $\Psi$, which implies $\varphi = 0$.
\end{proof}

\subsubsection{Square integrability and closability of the Gram operator}\label{sec:squareintegrability}

Let us now consider a different condition on the family $\Psi = \{\psi_j\}_{j \in \ind}$, namely
\begin{equation}\label{eq:generalsquareintegrability}
t_j = \sum_{k \in \ind} |\langle \psi_j, \psi_k\rangle_\Hil|^2 \ \textnormal{is finite for all} \ j \in \ind .
\end{equation}
This condition does not imply (\ref{eq:hpunifbound}), but it is sufficient to allow to work in $\ell_2(\ind)$, with the previously introduced operators being densely defined and closed. In particular, observe that $\An_\Psi$ can be equivalently obtained as the adjoint of the densely defined operator $\Sy_\Psi$ on $\ell_2(\ind)$, so in particular it is closed (see e.g. \cite[Chapter X, \S 1]{Conway90}).

\begin{lem}\label{lem:generalsquareintegrability}
Let $\Psi = \{\psi_j\}_{j \in \ind} \subset \Hil$ be a countable family. The analysis operator $\An_\Psi$ maps $\vsp(\Psi)$ to $\ell_2(\ind)$ if and only if (\ref{eq:generalsquareintegrability}) holds.
\end{lem}
\begin{proof}
Let us first assume (\ref{eq:generalsquareintegrability}) and let us take $\varphi \in \vsp(\Psi)$. Then for some finite subset $\Lambda \subset \ind$ we can write $\varphi = \sum_{j \in \Lambda} c_j \psi_j$ and
\begin{align*}
\|\An_\Psi \varphi\|_{\ell_2(\ind)}^2 & = \sum_{k \in \ind} \Big| \sum_{j \in \Lambda} c_j \langle \psi_j, \psi_k\rangle_\Hil\Big|^2 \leq |\Lambda| \sum_{j \in \Lambda} |c_j|^2 \sum_{k \in \ind} |\langle \psi_j, \psi_k\rangle_\Hil|^2 \\
& \leq |\Lambda| \|c\|^2_{\ell_2(\ind)} \max_{j \in \Lambda} t_j \, ,
\end{align*}
where the first inequality is Cauchy-Schwartz inequality applied to the sequences $\{c_j \langle \psi_j, \psi_k\rangle_\Hil\}_{j \in \Lambda}$ and $\{1\}_{j \in \Lambda}$.

Conversely, since $\An_\Psi : \vsp(\Psi) \to \ell_2(\ind)$, then in particular
\begin{displaymath}
\|\An_\Psi \psi_j\|_{\ell_2(\ind)}^2 = t_j < \infty \quad \forall \, j \in \ind. \qedhere
\end{displaymath}
\end{proof}

\begin{cor}\label{cor:closableGram}
Let $\Psi = \{\psi_j\}_{j \in \ind} \subset \Hil$ be a countable family for which  (\ref{eq:generalsquareintegrability}) holds. Then, the Gram operator is a closable densely defined operator on $\ell_2(\ind)$ whose domain contains finite sequences, i.e.
$$
\Grop_\Psi : \ell_0(\ind) \rightarrow \ell_2(\ind) .
$$
\end{cor}
\begin{proof}
$\Grop_\Psi$ is densely defined by Lemma \ref{lem:generalsquareintegrability}. Moreover, since $\An_\Psi$ is densely defined on $\Hil_\Psi$, then $\Sy_\Psi$ is closable, and its closure is given by $\Sy_\Psi^{**}$ (see e.g. \cite[Chapter X, \S 1]{Conway90}). To see that also $\Grop_\Psi$ is closable, let $\{f^n\}_{n \in \N} \subset \ell_0(\ind)$ be a sequence converging to $f \in \ell_2(\ind)$ such that $\{\Grop_\Psi f^n\}_{n \in \N}$ converges to $g \in \ell_2(\ind)$. This implies that $\{\Sy_\Psi f^n\}_{n \in \N} \subset \Hil$ is convergent, because
\begin{align*}
\|\Sy_\Psi f^n - \Sy_\Psi f^m\|^2_{\Hil} & = \langle f^n - f^m, \Grop_\Psi(f^n - f^m)\rangle_{\ell_2(\ind)}\\
& \leq \|f^n - f^m\|_{\ell_2(\ind)} \|\Grop_\Psi(f^n - f^m)\|_{\ell_2(\ind)} .
\end{align*}
Since $\Sy_\Psi$ is closable, then there exists $\varphi \in \Hil$ such that $\Sy_\Psi^{**} f = \varphi$, while the closedness of $\An_\Psi$ implies that $\An_\Psi \varphi = g$, so the extension of the Gram operator defined by $\An_\Psi\Sy_\Psi^{**}$ is closed.
\end{proof}

\noindent
Since $\Grop_\Psi$ is closable, we will always consider its closed extension and denote it with the same symbol.

\subsubsection{The Bessel condition}

Observe that condition (\ref{eq:generalsquareintegrability}) is less restrictive than Bessel, which would imply that both $\Sy_\Psi$ and its adjoint $\An_\Psi$ are bounded. However, in order to perform the composition of synthesis and analysis operators in reversed order with respect to the one providing the Gram operator, it is well known (see e.g. \cite{Christensen03}) that one needs $\Sy_\Psi$ to be a well defined operator from $\ell_2(\ind)$ to $\Hil$, in the sense that
\begin{equation}\label{eq:hp}
\textnormal{the series} \ \displaystyle\sum_{i \in \ind} c_i \psi_i \ \textnormal{converges in} \ \Hil \ \textnormal{for all} \ c \in \ell_2(\ind) .
\end{equation}
In this case, by uniform boundedness principle, $\Sy_\Psi$ is a bounded operator from $\ell_2(\ind)$ to $\Hil_\Psi$. Its adjoint operator $\An_\Psi : \Hil \to \ell_2(\ind)$ is then also bounded, that is equivalent to say that $\Psi$ is a Bessel sequence. Since assuming $\Psi$ to a Bessel sequence implies that $\Sy_\Psi$ is bounded, one then has that condition (\ref{eq:hp}) is equivalent to the Bessel condition.

Assuming (\ref{eq:hp}), one can then define the \emph{frame operator} as the bounded positive selfadjoint operator
$$
\begin{array}{rccc}
\Frame_\Psi = \Sy_\Psi \An_\Psi : & \Hil & \rightarrow & \Hil_\Psi\vspace{4pt}\\
& \varphi & \mapsto & \displaystyle\sum_{i \in \ind} \langle\varphi,\psi_i\rangle_\Hil \psi_i
\end{array}
$$
and under this hypothesis one also has that $\Grop_\Psi : \ell_2(\ind) \to \ell_2(\ind)$ is a bounded positive selfadjoint operator.

\subsection{Basic characterizations}

It is well known (see e.g. \cite[\S 3 Lem. 2]{Meyer93} or \cite[Th. 3.6.6]{Christensen03}) that $\Psi$ is a Riesz basis for $\Hil_\Psi$ if and only if its Gram operator is a bounded invertible operator on $\ell_2(\ind)$, i.e. if there exist two constants $0 < A \leq B < \infty$ such that
\begin{equation}\label{eq:RieszOperatorcondition}
A \Id_{\ell_2(\ind)} \leq \Grop_\Psi \leq B \Id_{\ell_2(\ind)}
\end{equation}
where $\Id_{\ell_2(\ind)}$ is the identity operator on $\ell_2(\ind)$. This is indeed a direct consequence of the definition of Riesz basis, since the central term in the inequalities (\ref{eq:Riesz}) reads
$$
\|\Sy_\Psi c\|^2_{\Hil} = \|\sum_{j \in \ind} c_j\psi_j\|^2_{\Hil} = \sum_{j,k \in \ind} c_j \ol{c_k} \langle \psi_j, \psi_k\rangle_{\Hil} = \langle c , \Grop_\Psi c\rangle_{\ell_2(\ind)} .
$$
Analogously (see e.g. \cite[\S 3.2]{Daubechies92}, \cite[\S 8.1]{HW96}), the condition of $\Psi$ being a frame for $\Hil_\Psi$ can be equivalently stated as
\begin{equation}\label{eq:frameOperatorcondition}
A \Id_{\Hil_\Psi} \leq \Frame_\Psi \leq B \Id_{\Hil_\Psi}
\end{equation}
where $\Id_{\Hil_\Psi}$ is the identity operator on $\Hil_\Psi$. Again this is a direct consequence of the definition of frames, since the central term in the inequalities (\ref{eq:frames}) reads
$$
\|\An_\Psi \varphi\|^2_{\ell_2(\ind)} = \sum_{j \in \ind} |\langle \varphi, \psi_j \rangle_{\Hil} |^2 = \sum_{j \in \ind} \langle \varphi, \psi_j \rangle_{\Hil}\langle \psi_j , \varphi\rangle_{\Hil} = \langle \varphi, \Frame_\Psi \varphi\rangle_{\Hil} .
$$

The relations between conditions (\ref{eq:RieszOperatorcondition}) and (\ref{eq:frameOperatorcondition}) may be also considered in view the following basic result.
\begin{lem}\label{lem:duallemma}
Let $\Hil_1$ and $\Hil_2$ be separable Hilbert spaces, let $K : \Hil_1 \to \Hil_2$ be a bounded linear operator and denote with $K^* : \Hil_2 \to \Hil_1$ its adjoint. Let us call $G = |K|^2 = K^*K$, and $F = |K^*|^2 = KK^*$. Then for fixed $0 < A \leq B < +\infty$ the following are equivalent
\begin{itemize}
\item[i.] $A \langle G c, c\rangle_{\Hil_1} \leq \langle G^2c, c\rangle_{\Hil_1} \leq B \langle G c, c\rangle_{\Hil_1}$ for all $c \in \Hil_1$
\item[ii.] $A \|\varphi\|_{\Hil_2}^2 \leq \langle F \varphi, \varphi\rangle_{\Hil_2} \leq B \|\varphi\|_{\Hil_2}^2$ for all $\varphi \in \ol{\Ran(K)}$.
\end{itemize}
Moreover, each of the conditions $i.$ and $ii.$ may be equivalently written in terms of the \emph{spectrum} $\sigma$ of the associated operator, as
$$
\begin{array}{ccc}
i. & \iff & \sigma(G) \subset [A,B] \cup \{0\}\vspace{4pt}\\
ii. & \iff & \sigma(F\big|_{\ol{\Ran(K)}}) \subset [A,B]  .
\end{array}
$$
\end{lem}
\begin{proof}
In order to prove the equivalence of $i.$ and $ii.$ it suffices to note that
$\langle G^2c, c\rangle_{\Hil_1} = \langle K^*KK^*K c, c\rangle_{\Hil_1} = \langle FK c, Kc\rangle_{\Hil_2}$ and that $\langle Gc, c\rangle_{\Hil_1} = \|Kc\|^2_{\Hil_2}$. The spectral conditions are then a consequence of writing $i.$ and $ii.$ as operator inequalities. Indeed condition $i.$ is equivalent to $A G \leq G^2 \leq B G$, which can be also written as $A \Proj_{\ol{\Ran(K^*)}} \leq G \leq B \Proj_{\ol{\Ran(K^*)}}$, while condition $ii.$ is equivalent to $A \Proj_{\ol{\Ran(K)}} \leq F \leq B \Proj_{\ol{\Ran(K)}}$.
\end{proof}

Since the condition for $\Psi$ to be a frame system has the form of point $ii.$ of the previous lemma with $K = \Sy_\Psi$, one immediately has the following result (that is equivalent to \cite[Lem. 5.5.4]{Christensen03}).
\begin{cor}\label{cor:principalframes}
Let $\Proj_{V_\Psi}$ denote the orthogonal projection in $\ell_2(\ind)$ onto the closed subspace $V_\Psi = \ol{\Ran(\Sy_\Psi^*)} = \Ker(\Sy_\Psi)^\bot$. Then condition (\ref{eq:frameOperatorcondition}) is equivalent to
\begin{equation}\label{eq:frameOperatorconditionEquiv}
A \Proj_{V_\Psi} \leq \Grop_\Psi \leq B \Proj_{V_\Psi} .
\end{equation}
\end{cor}

This also implies that Riesz systems are frame systems. Indeed, by (\ref{eq:RieszOperatorcondition}) we have $\sigma_{\ell_2(\ind)}(\Grop_\Psi) \subset [A,B]$ and by (\ref{eq:frameOperatorconditionEquiv}) we have $\sigma_{\ell_2(\ind)}(\Grop_\Psi) \subset [A,B] \cup \{0\}$.
\begin{rem}
When $\Psi$ is a Riesz system, by the characterization (\ref{eq:RieszOperatorcondition}) the bilinear form (\ref{eq:positiveform}) is equivalent to the $\ell_2(\ind)$ scalar product, so the Hilbert space given by (\ref{eq:RKHS}) is $\Hil_{\Grop_\Psi} \approx \ell_2(\ind)$. For frames that are not Riesz sequences, the lack of linear independence leads to a nonzero kernel of the Gram operator, reflecting the well known lack of uniqueness of the expansion coefficients. However, by (\ref{eq:frameOperatorcondition}) we have that if $\Psi$ is a frame, then\footnote{Since $\Grop_\Psi$ is bounded, then its kernel is closed. In the orthogonal complement of the kernel, the bounds from above and below make the $\ell_2$ norm equivalent to the $\Grop_\Psi$ norm.} we have $\Hil_{\Grop_\Psi} \approx \ell_2(\Gamma)/\Ker(\Grop_\Psi)$.
\end{rem}

\subsection{Cyclic subspaces for unitary actions of discrete groups}
In this subsection we will show that the Gram operator and the bracket map coincide at least for a dense set of elements in $\Hil$. A remarkable consequence, given by Corollary \ref{cor:principalcharacterization}, is the possibility to recover the main characterizations of frames in principal shift-invariant spaces such as \cite[Th. 2.16]{BoorVoreRon94}, \cite[Th. 3.4]{BenedettoLi98}, \cite[Th. 5.7]{HSWW10} and \cite[Th. A]{BHP14} directly from Corollary \ref{cor:principalframes} and, for abelian groups, by making use of the $*$-isomorphism provided by Pontryagin duality.

Consider a family $\Psi$ that is an orbit $\orb$ of a single vector $\psi \in \Hil$ under a unitary representation $\Pi$ of a discrete countable group $\Gamma$, i.e. $\orb = \{\Pi(\gamma)\psi\}_{\gamma \in \Gamma}$, and denote its linearly generated space by $\psis{\psi} = \ol{\vsp\,\orb}^\Hil$.
\newpage
\noindent
The associated Gram matrix reads
$$
\Gram_\orb^{\gamma' , \gamma} = \langle \Pi(\gamma)\psi, \Pi(\gamma')\psi\rangle_\Hil = \g(\gamma^{-1}\gamma') \ , \quad \gamma,\gamma' \in \Gamma
$$
where we have introduced the notation
$$
\g(\gamma) = \langle \psi, \Pi(\gamma)\psi\rangle_\Hil .
$$
The function $\g \in \ell_\infty(\Gamma)$ is the prototype of a function of positive type (see e.g. \cite[\S 3.3]{Folland95}), and $\Gram_\orb$ is the associated positive definite kernel.

If we assume condition (\ref{eq:generalsquareintegrability}), the Gram operator $\Grop_\orb$ is then a densely defined right convolution operator on $\ell_2(\Gamma)$. Indeed, if $a = \{a(\gamma)\}_{\gamma \in \Gamma} \in \ell_0(\Gamma)$
$$
\Grop_\orb a (\gamma') = \sum_{\gamma \in \Gamma} a(\gamma) \langle \psi, \Pi(\gamma^{-1}\gamma')\psi\rangle_\Hil
= a \ast \g (\gamma') .
$$
On the basis of this simple observation, we can deduce the following proposition.
\begin{prop}\label{prop:G=B}
Let $(\Gamma,\Pi,\Hil)$ be a dual integrable triple.
\begin{itemize}
\item[i.] If $\psi$ is such that $\Grop_\orb$ is a closed and densely defined operator on $\ell_2(\ind)$, then
$$
[\psi,\psi] = \Grop_\orb .
$$
\item[ii.] If $\psi$ is such that $[\psi,\psi] \in L^2(\vn(\Gamma))$, then $\Grop_\orb$ is a closed and densely defined operator on $\ell_2(\ind)$.
\end{itemize}
\end{prop}
\begin{proof}
To prove $i.$, let us first see that $\Grop_\orb \in L^1(\vn(\Gamma))$. As a right convolution operator, it is affiliated with $\vn(\Gamma)$. So, since it is a positive operator, it suffices to check that its trace is finite. This is true because
$$
\tau(\Grop_\orb) = \langle \An_\orb \Sy_\orb \delta_\id, \delta_\id \rangle_{\ell_2(\Gamma)} = \|\Sy_\orb \delta_\id\|^2_{\Hil} = \|\psi\|_\Hil^2 .
$$
In order to see the desired claim, by $L^1(\vn(\Gamma))$ uniqueness of Fourier coefficients (see e.g. \cite[Lem. 2.1]{BHP14}) we need only to prove that
$$
\langle \psi, \Pi(\gamma)\psi\rangle_\Hil = \tau(\Grop_\orb\rr(\gamma)) \quad \forall \, \psi \in \Hil \, , \ \forall \, \gamma \in \Gamma .
$$
Since $\Grop_\orb \delta_\id (\gamma) = \langle \psi, \Pi(\gamma)\psi\rangle_\Hil$, using the traciality of $\tau$ we have indeed
\begin{displaymath}
\tau(\Grop_\orb\rr(\gamma)) = \tau(\rr(\gamma)\Grop_\orb) = \langle \Grop_\orb \delta_\id, \delta_\gamma\rangle_{\ell_2(\Gamma)} = \langle \psi, \Pi(\gamma)\psi\rangle_\Hil .
\end{displaymath}
To prove $ii.$ observe that, by Plancherel Theorem (see e.g. \cite[Lem. 2.2]{BHP14}), $[\psi,\psi] \in L^2(\vn(\Gamma))$ implies that $\sum_{\gamma \in \Gamma}|\tau([\psi,\psi]\rr(\gamma))|^2 < \infty$. By definition of dual integrability, this is equivalent to $\sum_{\gamma \in \Gamma}|\langle \psi, \Pi(\gamma)\psi\rangle_\Hil|^2 < \infty$, which coincides with condition (\ref{eq:generalsquareintegrability}), so the conclusion follows by Corollary \ref{cor:closableGram}.
\end{proof}

\begin{cor}\label{cor:principalcharacterization}
Let $(\Gamma,\Pi,\Hil)$ be a dual integrable triple. Then $\orb$ is a frame sequence if and only if
\begin{equation}\label{eq:opframes}
A \Proj_{(\Ker[\psi,\psi])^\bot} \leq [\psi,\psi] \leq B \Proj_{(\Ker[\psi,\psi])^\bot} .
\end{equation}
It is a Riesz sequence whenever the same condition holds with $\Ker[\psi,\psi] = \{0\}$.
\end{cor}
\begin{proof}
Let us first assume that $\orb$ is a frame sequence. Then $\Grop_\orb$ is bounded, and the conclusion follows as a consequence of Corollary \ref{cor:principalframes} by using point i. of Proposition \ref{prop:G=B}. Conversely, assume (\ref{eq:opframes}). Then $[\psi,\psi]$ is bounded, so in particular it belongs to $\vn(\Gamma) \subset L^2(\vn(\Gamma)$ and the conclusion is again a consequence of Corollary \ref{cor:principalframes} by using Proposition \ref{prop:G=B}.
\end{proof}

\section{Global isometries}\label{sec:isometries}

In this section we will introduce a class of isometries, that we will call \emph{Helson maps}, and provide the crucial properties that will be used in next sections to associate a modular structure to any $(\Gamma,\Pi)$-invariant space.

Before proceeding, we take a moment to make the following observation.
\begin{lem}\label{lem:generators}
Let $\Pi$ be a unitary representation of a discrete and countable group $\Gamma$ on a separable Hilbert space $\Hil$, and let $V \subset \Hil$ a closed $(\Gamma,\Pi)$-invariant subspace. Then there exists a countable family $\{\psi_i\}_{i \in \ind}$ satisfying $\psis{\psi_i} \bot \psis{\psi_j}$ for $i \neq j$ and such that $V$ decomposes into the orthogonal direct sum
\begin{equation}\label{eq:GramSchmidt}
V = \bigoplus_{i \in \ind} \psis{\psi_i} .
\end{equation}
This implies in particular that $V = \ol{\vsp\{\Pi(\gamma)\psi_i\}_{\myatop{\gamma \in \Gamma}{i \in \ind}}}^{\Hil}$, so $\{\psi_i\}_{i \in \ind}$ is a countable family of generators for $V$.
\end{lem}
\begin{proof}
Such a family can be constructed as follows. Let $\{e_n\}_{n \in \N}$ be an orthonormal basis for $V$, and define $\psi_1 = e_1$. Then for all $n > 1$, if\vspace{-4pt}
$$
e_{n+1} \in V_n = \bigoplus_{k = 1}^n \psis{\psi_k}\vspace{-4pt}
$$
then set $\psi_{n+1} = 0$, otherwise define $\psi_{n+1} = \Proj_{V_n^\bot}e_{n+1}$. This proves (\ref{eq:GramSchmidt}) because $\vsp\{e_1, \dots , e_n\} \subset V_n$ for all $n$.
\end{proof}

\subsection{Weighted \texorpdfstring{$L^2(\vn(\Gamma))$}{L2} spaces}
We will often need to work with $L^2(\vn(\Gamma))$ spaces with positive weights provided by the bracket map, and with associated subspaces of $L^2(\vn(\Gamma))$ of operators composed with the support of the weight. This subsection contains all the fundamental results that we will use.

For all positive $\Omega \in L^1(\vn(\Gamma))$, the functional
\begin{equation}\label{eq:semin}
\|F\|_{2,\Omega} = \Big(\tau\big(|F^*|^2\Omega\big)\Big)^\frac12 = \|\Omega^\frac12 F\|_2 \ , \quad F \in \vn(\Gamma)
\end{equation}
defines a seminorm on $\vn(\Gamma)$, and let us call $N_\Omega$ the null space associated with (\ref{eq:semin}), that is $N_\Omega = \{F \in \vn(\Gamma) : \|F\|_{2,\Omega} = 0\}$. Let also $\mathfrak{h}$ be the space of bounded affiliated operators whose left support is contained in that of $\Omega$, i.e.
$$
\mathfrak{h} = \{H \in \vn(\Gamma) \ | \ \exists\,\, F \in \vn(\Gamma) : H = s_{\Omega} F\} .
$$
As in \cite[Lemma 3.3]{BHP14}, by defining the linear surjective map $\Sigma : \vn(\Gamma) \to \mathfrak{h}$ as
$$
\Sigma(F) = s_\Omega F
$$
and observing that $\Ker\, \Sigma = N_\Omega$, we have that $\mathfrak{h}$ can be identified with $\vn(\Gamma) / N_\Omega$.
The completion of $\mathfrak{h}$ with respect to the seminorm (\ref{eq:semin}) defines a Hilbert space that we denote
\begin{equation}\label{eq:weighted}
L^2(\vn(\Gamma),\Omega) = \ol{\mathfrak{h}}^{\|\cdot\|_{2,\Omega}} .
\end{equation}
On the other hand, if we consider the completion of $\mathfrak{h}$ with respect to the $\|\cdot\|_2$ norm we obtain a closed subspace of $L^2(\vn(\Gamma))$ as in the following lemma.
\begin{lem}
Let $\mathfrak{h}$ be defined as above, and denote with
$$
s_\Omega L^2(\vn(\Gamma)) = \{G \in L^2(\vn(\Gamma)) \ | \ \exists \, F \in L^2(\vn(\Gamma)) : G = s_\Omega F \} .
$$
Then $s_\Omega L^2(\vn(\Gamma))$ is a closed Hilbert subspace of $L^2(\vn(\Gamma))$, and
$$
s_\Omega L^2(\vn(\Gamma)) = \ol{\mathfrak{h}}^{\|\cdot\|_{2}} .
$$
\end{lem}
\begin{proof}
For all $F \in L^2(\vn(\Gamma))$, $s_\Omega F \in L^2(\vn(\Gamma))$ because $s_\Omega$ belongs to $\vn(\Gamma)$.
To prove that $s_\Omega L^2(\vn(\Gamma))$ is closed, observe that $F \in s_\Omega L^2(\vn(\Gamma))$ if and only if $F \in L^2(\vn(\Gamma))$ and $s_\Omega F = F$. Let then $\{F_n\}_{n \in \N} \subset s_\Omega L^2(\vn(\Gamma))$ be a sequence converging to a given $F$ in $L^2(\vn(\Gamma))$. The claim is proved if $s_\Omega F = F$, which is true because
\begin{align*}
\|s_\Omega F - F\|_2 & \leq \inf_{n \in \N} \|s_\Omega F - F_n\|_2 + \inf_{n \in \N}\|F_n - F\|_2\\
& = \inf_{n \in \N} \|s_\Omega (F - F_n)\|_2 \leq \inf_{n \in \N} \|F - F_n\|_2  = 0
\end{align*}
where the last inequality is H\"older's inequality. By the same argument we also have that $\ol{\mathfrak{h}}^{\|\cdot\|_{2}} \subset s_\Omega L^2(\vn(\Gamma))$, so let us consider $G \in s_\Omega L^2(\vn(\Gamma))$, and let $F \in L^2(\vn(\Gamma))$ be such that $G = s_\Omega F$. If $\{F_n\}_{n \in \N}$ is a sequence in $\vn(\Gamma)$ converging to $F$ in $L^2(\vn(\Gamma))$, then $\{s_\Omega F_n\}_{n \in \N}$ is a sequence in $\mathfrak{h}$ converging to $G$ in $L^2(\vn(\Gamma))$, so that $s_\Omega L^2(\vn(\Gamma)) \subset \ol{\mathfrak{h}}^{\|\cdot\|_{2}}$.
\end{proof}
The space $s_\Omega L^2(\vn(\Gamma))$ endowed with the $\|\cdot\|_2$ is actually isometrically isomorphic to the weighted Hilbert space defined by (\ref{eq:weighted}).
\begin{lem}\label{lem:weightmap}
The map $\omega : L^2(\vn(\Gamma),\Omega) \to L^2(\vn(\Gamma))$ defined by
\begin{equation}\label{eq:mapweighttosupport}
F \mapsto \Omega^\frac12 F
\end{equation}
is a Hilbert space isomorphism onto $s_\Omega L^2(\vn(\Gamma))$.
\end{lem}
\begin{proof}
That $\omega$ is an isometry is a direct consequence of the definition of the $\|\cdot\|_{2,\Omega}$ norm (observing that since $\Omega$ and $F$ are affiliated, then $\Omega^\frac12 F$ is affiliated). Moreover, since $s_\Omega \Omega^\frac12 = \Omega^\frac12$, we have that $\omega$ maps $L^2(\vn(\Gamma),\Omega)$ to $s_\Omega L^2(\vn(\Gamma))$. To prove surjectivity, let $F_0 \in s_\Omega L^2(\vn(\Gamma)) \cap \omega\big(L^2(\vn(\Gamma),\Omega)\big)^\bot$, i.e.
$$
\langle F_0, \Omega^\frac12 F\rangle_2 = 0 \quad \forall \, F \in L^2(\vn(\Gamma),\Omega).
$$
Since $s_\Omega \rr(\gamma)$ belongs to $L^2(\vn(\Gamma),\Omega)$ for all $\gamma \in \Gamma$, this implies that
$$
\tau(\Omega^\frac12 F_0 \rr(\gamma)) = 0 \quad \forall \, \gamma \in \Gamma
$$
so by $L^1(\vn(\Gamma))$ Uniqueness Theorem $\Omega^\frac12 F_0 = 0$, i.e. $\Ran(F_0) \subset \Ker(\Omega)$.
On the other hand $s_\Omega F_0 = F_0$, which implies $\Ran(F_0) \subset \Ker(\Omega)^\bot$, hence $F_0 = 0$.
\end{proof}

When $\Pi$ is dual integrable, the bracket $[\psi,\psi]$ for nonzero $\psi \in \Hil$ provides the positive $L^1(\vn(\Gamma))$ weight that we will use. Explicitly, the induced norm is
$$
\|F\|_{2,[\psi,\psi]} = \Big(\tau(|F^*|^2[\psi,\psi])\Big)^\frac12 = \|[\psi,\psi]^\frac12 F \|_2 .
$$
The associated weighted space is needed for the following result, which was proved in \cite[Prop. 3.4]{BHP14} and lays at the basis of our subsequent constructions.
\begin{prop}\label{prop:isometry}
Let $(\Gamma,\Pi,\Hil)$ be a dual integrable triple. Then for any nonzero $\psi \in \Hil$ the map $S_\psi : \vsp\,\orb \to \vsp\{\rr(\gamma)\}_{\gamma \in \Gamma} \subset \vn(\Gamma)$ given by
\begin{equation}\label{eq:isometry}
S_\psi : \sum_{\gamma\in \Gamma} a(\gamma) \Pi(\gamma)\psi \mapsto \sum_{\gamma\in \Gamma} a(\gamma) \rr(\gamma)^*
\end{equation}
extends to a linear surjective isometry $S_\psi : \psis{\psi} \to L^2(\vn(\Gamma),[\psi,\psi])$ satisfying
\begin{equation}\label{eq:intertwining}
S_\psi[\Pi(\gamma)\varphi] = S_\psi[\varphi] \rr(\gamma)^* \, , \quad \forall \, \varphi \in \psis{\psi}.
\end{equation}
\end{prop}
\begin{proof}[Sketch of the proof.]
Since this proposition differs from \cite[Prop. 3.4]{BHP14} only for the presence of the right regular representation, we restrict ourselves to show (\ref{eq:intertwining}), observing that it suffices to establish it on a dense subspace.

If $\varphi = \sum_{\gamma'\in \Gamma} a(\gamma') \Pi(\gamma')\psi \in \vsp\,\orb$, then
\begin{displaymath}
S_\psi[\Pi(\gamma)\varphi] = S_\psi\Big[\sum_{\gamma'\in \Gamma} a(\gamma') \Pi(\gamma\gamma')\psi\Big]
\!= \!\sum_{\gamma'\in \Gamma} a(\gamma') \rr(\gamma\gamma')^* = \!\sum_{\gamma'\in \Gamma} a(\gamma') \rr(\gamma')^* \rr(\gamma)^*
\end{displaymath}
where the last identity is due to $\rr(\gamma\gamma')^* = \rr(\gamma'^{-1}\gamma^{-1}) = \rr(\gamma'^{-1})\rr(\gamma^{-1})$.
\end{proof}

We observe that Proposition \ref{prop:isometry}, together with Lemma \ref{lem:synthesisdomain}, show that the weighted Hilbert space $L^2(\vn(\Gamma),[\psi,\psi])$ is isometrically isomorphic to the domain $\Hil_{\Grop_\orb}$ of the synthesis operator associated to the orbit $\orb$ of the dual integrable representation $\Pi$, via the map
$$
S_\psi\Sy_\orb : \Hil_{\Grop_\orb} \to L^2(\vn(\Gamma),[\psi,\psi]) .
$$
Heuristically, one can then consider $\Hil_{\Grop_\orb}$ as the space of Fourier coefficients for $L^2(\vn(\Gamma),[\psi,\psi])$. Indeed, for all $f$ such that the right hand side of the next  expression makes sense, for example for $f \in \ell_1(\Gamma)$, one has
$$
S_\psi[\Sy_\orb f] = \sum_{\gamma \in \Gamma} f(\gamma) \rr(\gamma)^*
$$
 so that $\{f(\gamma)\}_{\gamma \in \Gamma}$ play the role of Fourier coefficients, in the sense of (\ref{eq:Fouriercoefficients}).

\

For practical purposes, we write the following corollary.

\begin{cor}\label{cor:correspondence}
Let $(\Gamma,\Pi,\Hil)$ be a dual integrable triple and let $\psi \in \Hil$. For $F \in L^2(\vn(\Gamma),[\psi,\psi])$, let us denote with
$$
\proj_F \psi = S^{-1}_\psi F.
$$
Then $\proj_F \psi \in \psis{\psi}$ and
\begin{equation}\label{eq:boundedness}
\|\proj_F \psi\|_\Hil = \|F\|_{2,[\psi,\psi]} .
\end{equation}
Moreover, if $F \in \vn(\Gamma)$, then $\proj_F$ extends to a bounded operator on $\Hil$ that reads
\begin{equation}\label{eq:correspondence}
\proj_F \psi = \sum_{\gamma \in \Gamma} \wh{F}(\gamma) \Pi(\gamma)\psi \,, \qquad \psi \in \mathcal \Hil .
\end{equation}
\end{cor}
\begin{proof}
We need only to prove the boundedness of $\proj_F$ when $F \in \vn(\Gamma)$. In this case $|F^*|^2$ is bounded, so by H\"older inequality one obtains
$$
\|F\|_{2,[\psi,\psi]}^2 \leq \||F^*|^2\|_\infty \|[\psi,\psi]\|_1 = \||F^*|^2\|_\infty \|\psi\|_\Hil^2
$$
where $\||F^*|^2\|_\infty$ stands for the usual operator norm, and the last identity is due to Property III of the bracket map. Hence (\ref{eq:boundedness}) implies that $\qroj_F$ is bounded.
\end{proof}

\subsection{Helson maps}

These results on weighted spaces allow us to exploit the notion of Helson map and to prove, in Theorem \ref{theo:DualHelsonEquivalence}, that its existence is equivalent to the dual integrability of the triple $(\Gamma,\Pi,\Hil)$. As we will see in the next section, Helson maps generalizes the fiberization mapping (\ref{eq:Fourierperiodization}) as well as the Zak transform, and this notion will be used to characterize the modular structures associated with $(\Gamma,\Pi)$-invariant subspaces of $\Hil$. 

When the triple $(\Gamma,\Pi,\Hil)$ admits a Helson map $\iso$, see Definition \ref{def:Helsonmap}, we denote with $\Kil = \iso(\Hil)$ its range, that is a Hilbert space endowed with the norm of $L^2((\M,\nu),L^2(\vn(\Gamma)))$ and that we denote for short with
$$
\norm{\varPhi} = \left(\int_\M \|\varPhi(x)\|_2^2 d\nu(x)\right)^\frac12 \quad \forall \, \varPhi \in L^2((\M,\nu),L^2(\vn(\Gamma))).
$$

The first step for the proof of Theorem \ref{theo:DualHelsonEquivalence} is showing that the existence of a Helson map implies dual integrability.
\begin{prop}\label{prop:propertyii}
Let $(\Gamma,\Pi,\Hil)$ admit a Helson map $\iso$. Then it is a dual integrable triple, and the bracket map can be expressed as
\begin{equation}\label{eq:Helsonbracket}
\displaystyle[\varphi,\psi] = \int_\M \isom{\psi}^* \isom{\varphi} d\nu(x) .
\end{equation}
\end{prop}
\begin{proof}
Since $\iso$ is an isometry satisfying (\ref{eq:Helsonproperty}), we have
\begin{align*}
\langle & \varphi, \Pi(\gamma)\phi \rangle_\Hil = \langle \iso[\varphi], \iso[\Pi(\gamma)\phi] \rangle_\oplus = \int_\M \langle \isom{\varphi}, \isom{\Pi(\gamma)\phi}\rangle_{2} d\nu(x)  \\
& = \int_\M \langle \isom{\varphi}, \isom{\phi}\rr(\gamma)^*\rangle_{2} d\nu(x)  = \tau\bigg(\rr(\gamma) \int_\M \isom{\phi}^* \isom{\varphi} d\nu(x)\bigg)
\end{align*}
where the last identity is due to Fubini's Theorem, which holds by the normality of $\tau$. Let us  prove that the right hand side of  (\ref{eq:Helsonbracket}) is in $L^1(\vn(\Gamma))$. For this, we only need to see that its norm is finite, which is true because
\begin{align*}
\Big\| & \int_\M \isom{\psi}^* \isom{\varphi} d\nu(x) \Big\|_{1} \leq \int_\M \|\isom{\psi}^* \isom{\varphi}\|_1 d\nu(x)\\
& \leq \int_\M \|\isom{\psi}\|_2 \|\isom{\varphi}\|_2 d\nu(x) \leq \norm{\iso[\psi]} \norm{\iso[\varphi]} = \|\psi\|_\Hil\|\varphi\|_\Hil \end{align*}
where we have used H\"older's inequality on $L^2(\vn(\Gamma))$ and on $L^2(\M,d\nu)$.
Now, since we have that the Fourier coefficients of $[\varphi,\psi]$ and $\int_\M \isom{\psi}^* \isom{\varphi} d\nu(x)$ coincide, (\ref{eq:Helsonbracket}) holds by $L^1(\vn(\Gamma))$ Uniqueness Theorem.
\end{proof}
We remark that, by a similar argument to the one used in the proof of \cite[Th. 4.1]{BHP14}, the map
$$
\begin{array}{rcl}
\Gamma \times \Kil & \to & \Kil\\
(\gamma, \varPhi) & \mapsto & \varPhi \rr(\gamma)^*
\end{array}
$$
defines a unitary representation of $\Gamma$ on $\Kil$ that is unitarily equivalent to a summand of a direct integral decomposition of the left regular representation. Indeed, a Helson map is actually the intertwining operator between $\Pi$ and such a representation.

We are now ready to provide a result associated with this kind of isometry, which will be crucial in the next sections for the characterization of Riesz and frame sequences, as well as for the module structure and the equivalence between shift-invariance and noncommutative multiplicative invariance.

\begin{prop}\label{prop:fundamental}
Let $(\Gamma,\Pi,\Hil)$ be a dual integrable triple admitting a Helson map $\iso$. Let $F \in \vn(\Gamma)$ and let $\qroj_F$ be the associated bounded operator on $\Hil$ given by (\ref{eq:correspondence}). Then
\begin{itemize}
\item[i.] $\iso[\qroj_F\varphi] = \iso[\varphi] F$ \quad for all $\varphi \in \Hil$
\item[ii.] $[\qroj_F \varphi, \psi] = [\varphi, \psi]F$ , \, $[\varphi, \qroj_F \psi] = F^*[\varphi, \psi]$ \quad for all $\varphi, \psi \in \Hil$
\item[iii.] if $V \subset \Hil$ is a $(\Gamma,\Pi)$-invariant subspace, then $\qroj_F V \subset V$.
\end{itemize}
\end{prop}
\begin{proof}

To prove $i.$ observe first that the identity holds for trigonometric polynomials as a consequence of (\ref{eq:Helsonproperty}). Let then $\{F_n\}_{n \in \N}$ be a sequence of trigonometric polynomials such that $\{F_n^*\}_{n \in \N}$ converges strongly to $F^*$, i.e.
$$
\|F_n^* u - F^* u\|_{\ell_2(\Gamma)} \to 0, \quad \forall \, u \in \ell_2(\Gamma).
$$
Observe that such a sequence always exists because $\vn(\Gamma)$ coincides with the SOT-closure of trigonometric polynomials by von Neumann's Double Commutant Theorem (see e.g. \cite{Conway00}). This implies that for all $\psi \in \Hil$
\begin{equation}\label{eq:convergence}
\|F_n - F\|_{2,[\psi,\psi]} \to 0 .
\end{equation}
Indeed, by definition of the weighted norm we have
\begin{align*}
\|F_n - F\|_{2,[\psi,\psi]}^2 & = \|[\psi,\psi]^\frac12 (F_n - F)\|_2^2  = \tau((F_n - F) (F_n - F)^* [\psi,\psi])\\
& = \langle (F_n - F)^* [\psi,\psi] \delta_\id, (F_n - F)^* \delta_\id\rangle_{\ell_2(\Gamma)}\\
& \leq \|(F_n - F)^* [\psi,\psi] \delta_\id\|_{\ell_2(\Gamma)} \|(F_n - F)^* \delta_\id\|_{\ell_2(\Gamma)}
\end{align*}
where $[\psi,\psi] \delta_\id \in \ell_2(\Gamma)$ because the domain of $[\psi,\psi] \in L^1(\vn(\Gamma)$ contains finite sequences. Then (\ref{eq:convergence}) follows because $\{F_n^*\}_{n \in \N}$ converges strongly to $F^*$.
Now, by (\ref{eq:boundedness}),  we have
$$
\|\qroj_{F} \varphi - \qroj_{F_n} \varphi \|_\Hil = \|\qroj_{F - F_n} \varphi\|_\Hil = \| F - F_n\|_{2,[\varphi,\varphi]}
$$
for all $\varphi\in\Hil$ and thus (\ref{eq:convergence}) implies that $\qroj_{F_n}$ converges strongly to $\qroj_{F}$. As a consequence, since $\iso$ is continuous, we obtain 
\begin{equation}\label{eq:intermezzo}
\norm{\iso[\qroj_{F}\varphi] - \iso[\qroj_{F_n}\varphi]} \to 0 \quad \forall \, \varphi \in \Hil.
\end{equation}
On the other hand, observe that $\iso[\varphi]F$ belongs to $L^2\big((\M,\nu),L^2(\vn(\Gamma))\big)$ while $\iso[\varphi] F_n \in \Kil$ for all $n$. Now we have
\begin{align*}
\norm{\iso[\varphi]F - \iso[\varphi]F_n}^2 & = \tau \left(\int_\M |\isom{\varphi}(F - F_n)|^2 d\nu(x)\right)\\
& = \tau \left(|(F - F_n)^*|^2\int_\M |\isom{\varphi}|^2 d\nu(x)\right) = \|F - F_n\|^2_{2,[\varphi,\varphi]} ,
\end{align*}
where the last identity is due to Proposition \ref{prop:propertyii}. This implies that $\{\iso[\varphi]F_n\}_{n \in \N}$ converges to $\iso[\varphi] F$ in  $L^2\big((\M,\nu),L^2(\vn(\Gamma))\big)$, so in particular $\iso[\varphi] F \in \Kil$. Therefore, since $\iso[\varphi]F_n=\iso[\qroj_{F_n}\varphi]$ for all $n$, by (\ref{eq:intermezzo}) we obtain $\iso[\varphi]F=\iso[\qroj_{F}\varphi]$ which gives $i.$ 
\newpage
To prove the first identity in $ii.$ observe that since both $[\qroj_F \varphi, \psi]$ and $[\varphi, \psi]F$ belong to $L^1(\vn(\Gamma))$, by $L^1(\vn(\Gamma))$ Uniqueness Theorem it suffices to show that their Fourier coefficients coincide. This is true by definition of dual integrability, because
\begin{align*}
\tau([\qroj_F \varphi, \psi]\rr(\gamma)) & = \langle \qroj_F \varphi, \Pi(\gamma)\psi\rangle_\Hil = \sum_{\gamma' \in \Gamma} \wh{F}(\gamma') \langle \Pi(\gamma') \varphi, \Pi(\gamma)\psi\rangle_\Hil\\
& = \sum_{\gamma' \in \Gamma} \wh{F}(\gamma') \langle \varphi, \Pi(\gamma'^{-1}\gamma)\psi\rangle_\Hil = \sum_{\gamma' \in \Gamma} \wh{F}(\gamma') \tau([\varphi, \psi]\rr(\gamma'^{-1}\gamma))\\
& = \sum_{\gamma' \in \Gamma} \wh{F}(\gamma') \tau([\varphi, \psi]\rr(\gamma')^*\rr(\gamma)) = \tau([\varphi,\psi]F\rr(\gamma))
\end{align*}
where the last identity is due to the normality of the trace. The second identity in $ii.$ is a consequence of this and of Property I) of the bracket map.

To prove $iii.$ we first observe that, by the same argument used to prove $i.$, for all $\psi \in \Hil$ we have that $\qroj_F\psi \in \psis{\psi}$. Indeed by choosing the same sequence $\{F_n\}_{n \in \N}$ of trigonometric polynomials we have that $\qroj_{F_n}\psi \in \vsp\,\orb$ so that $\qroj_F\psi$ belongs to its closure by $i.$. The conclusion then follows because if $\psi \in V$ then the whole $\psis{\psi}$ is a closed subspace of $V$.
\end{proof}

\subsection{Fiberization mapping}

We now show that every dual integrable triple possesses a Helson map, by providing a concrete construction based on Fourier analysis. We also show that Helson maps extend respectively the fiberization mapping \cite{Helson64, BoorVoreRon93, BoorVoreRon94, Bownik00, CabrelliPaternostro10, CabrelliPaternostro11} and the Zak transform \cite{Weil64, Zak67, Janssen82, HSWW10, Zak-SIS} to the present noncommutative setting, modifying and slightly extending the constructions made in \cite{BHP14}.

Let $(\Gamma,\Pi,\Hil)$ be a dual integrable triple and, for a given $0 \neq \psi \in \Hil$, let
$$
U_\psi : \psis{\psi} \to s_{[\psi,\psi]} L^2(\vn(\Gamma))
$$ be the surjective isometry obtained from (\ref{eq:isometry}) by means of the map (\ref{eq:mapweighttosupport}) as
$$
U_\psi = [\psi,\psi]^\frac12 S_\psi .
$$
A slight modification of \cite[Th. 4.1]{BHP14} provides then the following.

\begin{prop}\label{prop:periodization}
Let $(\Gamma,\Pi,\Hil)$ be a dual integrable triple and let $\Psi=\{\psi_i\}_{i\in\ind}$ be a family as in Lemma \ref{lem:generators} for $\Hil$. For any $\varphi \in \Hil$ let $\varphi^{(i)}$ be its orthogonal projection onto $\psis{\psi_i}$, i.e. $\varphi^{(i)} = \Proj_{\psis{\psi_i}} \varphi$.
Then the map $U_{\Psi} : \varphi \mapsto \displaystyle\big\{U_{\psi_i}[\varphi^{(i)}]\big\}_{i \in \ind}$ defines a linear surjective isometry
$$
U_{\Psi} : \Hil \to \mathcal{K}_\Pi^\Psi = \displaystyle\bigoplus_{i \in \ind} s_{[\psi_i,\psi_i]}L^2(\vn(\Gamma)) \ \subset \ \ell_2(\ind,L^2(\vn(\Gamma)))
$$
that satisfies
\begin{equation}\label{eq:intertwining2}
U_\Psi[\Pi(\gamma)\varphi] =  U_\Psi[\varphi] \rr(\gamma)^*\, , \quad \forall \, \gamma \in \Gamma , \ \forall \, \varphi \in \Hil .
\end{equation}
\end{prop}
\begin{proof}
Since each map $U_{\psi_i}$ is an isometric isomorphism from $\psis{\psi_i}$ onto the space $s_{[\psi_i,\psi_i]}L^2(\vn(\Gamma))$, we have
\begin{displaymath}
\|U_{\Psi}[\varphi]\|_{\mathcal{K}_\Pi^\Psi}^2 = \sum_{i \in \ind} \|U_{\psi_i}[\varphi^{(i)}]\|_2^2 = \sum_{i \in \ind} \|\varphi^{(i)}\|_{\Hil}^2 = \|\varphi\|_\Hil^2
\end{displaymath}
while identity (\ref{eq:intertwining2}) is a consequence of (\ref{eq:intertwining}).
\end{proof}

This abstract isometry was introduced in the abelian setting in \cite{HSWW10}, and it has been recently used for the study of spaces invariant under the action of an LCA group in \cite{Iverson14}. In the present setting, it provides the conclusion of the proof that having a Helson map is equivalent to dual integrability.

\begin{proof}[Proof of Theorem \ref{theo:DualHelsonEquivalence}]
If $(\Gamma,\Pi,\Hil)$ is a dual integrable triple, then a Helson map is provided by Proposition \ref{prop:periodization}, while the converse is due to Proposition \ref{prop:propertyii}.
\end{proof}

\paragraph{Integer translations on $L^2(\R)$.}
The map $U_\Psi$ extends to general discrete groups and unitary representations the fiberization mapping (\ref{eq:Fourierperiodization}) (see also \cite[Prop. 3.3]{CabrelliPaternostro10} for general LCA groups). Indeed, when $\Gamma$ is abelian, by the isomorphism between $\vn(\Gamma)$ and $\wh{\Gamma}$ provided by Pontryagin duality we have
$$
\ell_2(\ind,L^2(\vn(\Gamma))) \approx \ell_2(\ind,L^2(\wh{\Gamma})) \approx L^2(\wh{\Gamma},\ell_2(\ind)) .
$$
If $\Gamma = \Z$ and $\Pi$ are integer translations on $L^2(\R)$, i.e. $\Pi(k)\varphi(x) = \varphi(x - k)$, let $\ind = \Gamma^*$, the annihilator of the discrete group $\Gamma$, which coincides in this case with $\Z$, and let $\Psi \subset L^2(\R)$ be the Shannon system
$$
\wh{\psi}_j = \chi_{[j,j+1]}, \,\, j\in\Z.
$$
It is easy to see that this system satisfies (\ref{eq:GramSchmidt}) for $V = L^2(\R)$, and that integer translates of each $\psi_j$ generates an orthonormal system, so that $[\psi_j,\psi_j] = \Id_{\ell_2(\Z)}$. In this case $U_{\psi_j} = S_{\psi_j}$ for all $j\in\Z$, because $s_{[\psi_j,\psi_j]} = \Id_{\ell_2(\Z)}$. Now let $\varphi \in L^2(\R)$, with components $\varphi^{(j)} \in \csp\{\psi_j(\cdot - k)\}_{k \in \Z}$ given by
$$
\varphi^{(j)}(x) = \sum_{k \in \Z} a^j_k \psi_j(x - k) = \sum_{k \in \Z} a^j_k \Pi(k)\psi_j(x).
$$
Then, the map $U_{\Psi}$ reads 
\begin{equation}\label{eq:integerisometry}
U_\Psi[\varphi] = \Big\{U_{\psi_l}[\varphi^{(l)}]\Big\}_{l \in \Z} = \Big\{S_{\psi_l}[\varphi^{(l)}]\Big\}_{l \in \Z} = \Big\{\sum_{k \in \Z} a_k^l \rr(k)^*\Big\}_{l \in \Z},
\end{equation}
where $\{\rr(k)\}_{k \in \Z}$ is the sequence of translation operators in $\ell^2(\Z)$.

On the other hand, by definition, $\wh{\psi}_j (\alpha + l) = \delta_{j,l}$ for all $j, l \in \Z$ and a.e. $\alpha \in \wh{\Gamma} = [0,1)$. Thus, $\wh{\varphi^{(j)}}(\alpha + l) = \sum_{k \in \Z} a^j_k e^{-2\pi i k \alpha}\delta_{j,l}$ for a.e. $\alpha \in [0,1)$, and
$$
\wh{\varphi}(\alpha + l) = \sum_{j\in\Z} \wh{\varphi^{(j)}}(\alpha + l) = \sum_{k \in \Z} a^l_k e^{-2\pi i k \alpha} \quad \textrm{for a.e.} \ \alpha \in [0,1) \ \textnormal{and all} \ l \in \Z.
$$
Therefore, the fiberization mapping (\ref{eq:Fourierperiodization}) reads  $\iso \varphi (\cdot)= \{ \sum_{k \in \Z} a^l_k e^{-2\pi i k \cdot}\}_{l\in\Z}$. 

As a conclusion, we see that the mappings $\iso$ and $U_{\Psi}$ can be obtained one from the another one by applying Pontryagin duality, i.e. by turning the characters $\{e^{2\pi i k\cdot}\}_{k \in \Z}$ of $\Gamma=\Z$ into the integer translations $\{\rr(k)\}_{k \in \Z}$ over $\ell_2(\Z)$.

\paragraph{Measurable group actions on $L^2(\meas,\mu)$ and Zak transform.}
A particular construction of a Helson map, which differs from the fiberization mapping, can be given in terms of the Zak transform whenever the representation $\Pi$ arises from a measurable action of a discrete group on a measure space. This was first considered in \cite{HSWW10} and then in \cite{Zak-SIS} in the abelian setting. For the nonconmmutative case, the Zak transform was taken into consideration  in \cite{BHP14}. For the sake of completeness we include its construction here.
\newpage
Consider a $\sigma$-finite measure space $(\meas,\mu)$, $\Gamma$ a countable discrete group  and let $\sigma: \Gamma  \times \meas \to \meas$ be a quasi $\Gamma$-invariant measurable action of $\Gamma$ on $\meas$. This means that for each $\gamma \in \Gamma$ the map $x \mapsto \sigma_\gamma (x) = \sigma(\gamma,x)$ is $\mu$-measurable, that for all $\gamma, \gamma' \in \Gamma$ and almost all $x \in \meas$ it holds $\sigma_\gamma (\sigma_{\gamma'} (x)) = \sigma_{\gamma \gamma'}(x)$ and $\sigma_\id (x) = x$, and that for each $\gamma \in \Gamma$  the measure $\mu_\gamma$ defined by $\mu_\gamma(E) = \mu(\sigma_\gamma (E))$ is absolutely continuous with respect to $\mu$ with positive Radon-Nikodym derivative. Let us indicate the family of associated Jacobian densities with the measurable function $J_\sigma: \Gamma \times \meas \to \R^+$ given by
$$
d\mu(\sigma_\gamma (x)) = J_\sigma (\gamma, x)\, d\mu (x) .
$$
We can then define a unitary representation $\Pi_\sigma$ of $\Gamma$ on $L^2(\meas,\mu)$ as
\begin{equation}\label{eq:measurablerep}
\Pi_\sigma(\gamma)\varphi(x) = J_\sigma (\gamma^{-1}, x)^\frac12 \varphi(\sigma_{\gamma^{-1}}(x)) . 
\end{equation}
We say that the action $\sigma$ has the \emph{tiling property} if there exists a $\mu$-measurable subset $C \subset \meas$ such that the family
$\{\sigma_\gamma (C)\}_{\gamma \in \Gamma}$ is a disjoint covering of $\mu$-almost all $\meas$, i.e.
$\mu\big(\sigma_{\gamma_1}(C) \cap \sigma_{\gamma_2} (C)\big) = 0$ for $\gamma_1 \neq \gamma_2$ and
$$
\mu\bigg(\meas \setminus \bigcup_{\gamma \in \Gamma}\,\sigma_\gamma (C)\bigg) = 0 .
$$

Following \cite{BHP14}, we call \emph{noncommutative Zak transform} of $\varphi \in L^2(\meas,\mu)$ associated to the action $\sigma$ the measurable field of operators over $\meas$ given by
$$
Z_\sigma[\varphi](x) = \sum_{\gamma \in \Gamma} \Big(\big(\Pi_\sigma(\gamma)\varphi\big)(x)\Big) \rr(\gamma), \quad x \in \meas .
$$

The following result is a slight improvement of \cite[Th. B, i)]{BHP14}, showing that $Z_\sigma$ defines an isometry that is surjective on the whole $L^2((C,\mu),L^2(\vn(\Gamma)))$.
\begin{prop}\label{prop:Zak}
Let $\sigma$ be a quasi-$\Gamma$-invariant action of the countable discrete group $\Gamma$ on the measure space $(\meas,\mu)$, and let $\Pi_\sigma$ be the unitary representation given by  (\ref{eq:measurablerep}) on $L^2(\meas,\mu)$. If $\sigma$ has the tiling property with tiling set $C$, then the Zak transform $Z_\sigma$ defines an isometric isomorphism
$$
Z_\sigma : L^2(\meas,\mu) \to L^2((C,\mu),L^2(\vn(\Gamma)))
$$
satisfying the quasi-periodicity condition
$$
Z_\sigma[\Pi_\sigma(\gamma)\varphi] = Z_\sigma[\varphi] \rr(\gamma)^* \, , \quad  \forall \, \gamma \in \Gamma , \ \forall \, \varphi \in L^2(\meas,\mu) .
$$
\end{prop}

\begin{proof}
The isometry can be proved as in \cite[Th. B]{BHP14}, while quasi-periodicity can be obtained explicitly by
$$
Z_\sigma[\Pi_\sigma(\gamma)\varphi] = \sum_{\gamma'} \Big(\big(\Pi_\sigma(\gamma'\gamma)\varphi\big)(x)\Big) \rr(\gamma') = \sum_{\gamma''} \Big(\big(\Pi_\sigma(\gamma'')\varphi\big)(x)\Big) \rr(\gamma''\gamma^{-1}) .
$$
To prove surjectivity, take $\varPsi \in L^2((C,\mu),L^2(\vn(\Gamma)))$ and for each $\gamma \in \Gamma$ define
\begin{equation}\label{eq:Zakinversion}
\psi(x) = J_\sigma (\gamma^{-1}, \sigma_{\gamma}(x))^{-\frac12} \tau \big( \varPsi(\sigma_{\gamma}(x)) \rr(\gamma)^* \big) \quad \textnormal{a.e.} \ x \in \sigma_{\gamma^{-1}}(C) .
\end{equation}
Such a $\psi$ belongs to $L^2(\meas,\mu)$, since by the tiling property it is measurable and its norm reads
\begin{align*}
\|\psi\|_{L^2(\meas,\mu)}^2 &\! = \!\int_\meas \!\! |\psi(x)|^2 d\mu(x) = \sum_{\gamma \in \Gamma} \int_{\sigma_{\gamma^{-1}}(C)} \!\!\!\!\!\!\!\!\!\!\!\!\!\!\!\!J_\sigma (\gamma^{-1}, \sigma_{\gamma}(x))^{-1} \Big|\tau \big( \varPsi(\sigma_{\gamma}(x)) \rr(\gamma)^* \big)\Big|^2 d\mu(x)\\
&\! = \!\sum_{\gamma \in \Gamma} \int_{C} J_\sigma (\gamma^{-1}, y)^{-1} |\tau (\varPsi(y) \rr(\gamma)^* )|^2 J_\sigma (\gamma^{-1}, y)d\mu(y)
\end{align*}
where the last identity is due to the definition of the Jacobian density, because 
$d\mu(x)=d\mu(\sigma_{\gamma^{-1}}(y))=J_\sigma (\gamma^{-1}, y)d\mu(y)$.
Then, by Plancherel Theorem
$$
\|\psi\|_{L^2(\meas,\mu)}^2 = \int_C \sum_{\gamma \in \Gamma} |\tau (\varPsi(y) \rr(\gamma)^* )|^2  d\mu(y) = \int_C \|\varPsi(y)\|_2^2  d\mu(y)  
$$
so that $\|\psi\|_{L^2(\meas,\mu)}^2 = \|\varPsi\|^2_{L^2((C,\mu),L^2(\vn(\Gamma)))}<+\infty$. By applying the Zak transform to $\psi$ we then have that, for a.e. $x \in C$,
\begin{align*}
Z_\sigma[\psi](x) = \sum_{\gamma \in \Gamma} J_\sigma (\gamma^{-1}, x)^\frac12 \psi(\sigma_{\gamma^{-1}}(x)) \rr(\gamma) = \sum_{\gamma \in \Gamma} \tau ( \varPsi(x) \rr(\gamma)^* \big) \rr(\gamma) = \varPsi(x)
\end{align*}
again by Plancherel Theorem. This proves surjectivity and in particular shows that (\ref{eq:Zakinversion}) provides an explicit inversion formula for $Z_\sigma$.
\end{proof}

\begin{rem}
The Zak transform is actually directly related to the isometry $S_\psi$ introduced in (\ref{eq:isometry}), since for all $F \in \vn(\Gamma)$ it holds
$$
S^{-1}_\psi F (x) = \tau \big( Z_\sigma[\psi](x) F \big) \quad \textnormal{a.e.} \ x \in \meas .
$$
Indeed, denoting with $\{\wh{F}(\gamma)\}_{\gamma \in \Gamma}$ the Fourier coefficients of $F$, by the orthonormality of $\{\rr(\gamma)\}_{\gamma \in \Gamma}$ in $L^2(\vn(\Gamma))$ it holds
$$
\tau \big( Z_\sigma[\psi](x) F \big) = \sum_{\gamma, \gamma' \in \Gamma} \big(\Pi_\sigma(\gamma) \psi\big) (x) \wh{F}(\gamma') \, \tau(\rr(\gamma)\rr(\gamma')^*) = \sum_{\gamma \in \Gamma} \wh{F}(\gamma) \Pi_\sigma(\gamma)\psi(x).
$$
For a relationship between the Zak transform and the global isometry $U_\Psi$ in LCA groups, see \cite[Prop. 6.11]{Zak-SIS}.
\end{rem}

\section{Shift invariance and Hilbert modules}\label{sec:Hilbertmodules}

In this section we will show how the invariance under the action of a dual integrable representation is associated to the notion of Hilbert $L^2(\vn(\Gamma))$-modules.

Throughout this section we will assume that $(\Gamma,\Pi,\Hil)$ is  a dual integrable triple and denote with $\iso : \Hil \to \Kil$ an associated Helson map 

The first fundamental result is the following.

\begin{theo}\label{theo:multiplicatively}
Let $V$ be a closed subspace of $\Hil$, and let $M = \iso(V)$ be the corresponding closed subspace of $\Kil$. Then $V$ is $(\Gamma,\Pi)$-invariant if and only if
\begin{equation}\label{eq:multiplicatively}
M A \subset M \quad \forall \, A \in \vn(\Gamma) .
\end{equation}
\end{theo}
\begin{proof}
Assume that $V$ is $(\Gamma,\Pi)$-invariant, and take $\varPhi \in M$. Then, by Proposition \ref{prop:fundamental}, we have $\varPhi A \in M$ for all $A \in \vn(\Gamma)$ because, calling $\varphi = \iso^{-1}\varPhi \in V$
$$
\varPhi A = \iso[\varphi] A = \iso [\qroj_A \varphi]
$$
where $\qroj_A \varphi \in V$. Conversely, if we take $\varphi \in V$ then for all $\gamma \in \Gamma$ the relations (\ref{eq:Helsonproperty}) and (\ref{eq:multiplicatively}) imply that $\iso(\Pi(\gamma)\varphi) \in M$ and hence $\Pi(\gamma)\varphi \in V$.
\end{proof}

In the abelian setting, a subspace $M$ of $\Kil$ satisfying (\ref{eq:multiplicatively}) is called \emph{multiplicatively invariant} \cite{BownikRoss13}. In this noncommutative setting, however, the invariance is at the level of composition with operators in the von Neumann algebra, hence giving rise to a module structure.

\subsection{\texorpdfstring{$L^2(\vn(\Gamma))$}{L2}-Hilbert modules with outer coefficients}

Given a von Neumann algebra $\ALG$, an $L^p(\ALG)$-Hilbert module is a module over $\ALG$ endowed with an $L^\frac{p}{2}(\ALG)$ valued inner product. Standard references for modules can be \cite{Keating98, KostrikinShafarevich91}, or \cite{BlecherLeMerdy04} for operator modules. The main reference for $L^p$ Hilbert modules over von Neumann algebras is \cite{JungeSherman05}, where they were introduced. Standard reference for Hilbert modules over $C^*$-algebras, which include the case $p=\infty$, can be \cite[Chapt. 15]{Wegge93} or \cite{Lance95, ManuilovTroitsky05}, while a comprehensive bibliography can be found in \cite{FrankBiblio}.

For the sake of completeness, we recall in detail the definition of module, and the definition of inner product introduced in \cite{JungeSherman05}.

\begin{defi}\label{def:modules}
Let $\ALG$ be a complex associative algebra with unity $\Id$. A complex vector space $\E$ is called an \emph{$\ALG$-module} if there exists a map
$$
\vartheta : \E \times \ALG \to \E
$$
satisfying the following relations for all $x, y \in \E$, all $F, G \in \ALG$ and all $\alpha \in \C$:
\begin{itemize}
\item[M1.] $\vartheta(x + y, F) = \vartheta(x,F) + \vartheta(y,F)$\vspace{-4pt}
\item[M2.] $\vartheta(x, F + G) = \vartheta(x,F) + \vartheta(x,G)$\vspace{-4pt}
\item[M3.] $\vartheta(x, FG) = \vartheta(\vartheta(x,F),G)$\vspace{-4pt}
\item[M4.] $\vartheta(x,\Id) = x$\vspace{-4pt}
\item[M5.] $\vartheta(\alpha x, F) = \vartheta(x, \alpha F) = \alpha \vartheta(x,F)$.
\end{itemize}
\end{defi}
\begin{defi}\label{def:Hmodules}
Let $\ALG$ be a von Neumann algebra, let $\E$ be an \emph{$\ALG$-module}, and let $0 < p \leq \infty$.
An $L^\frac{p}{2}(\ALG)$-valued \emph{inner product} over $\E$ is a map $$\ipr{\cdot}{\cdot} : \E \times \E \to L^\frac{p}{2}(\ALG)$$ satisfying the following relations for all $x, y, z \in \E$, all $F \in \ALG$ and all $\alpha, \beta \in \C$:
\begin{itemize}
\item[i.] $\ipr{x}{\alpha y + \beta z} = \ol{\alpha} \ipr{x}{y} + \ol{\beta} \ipr{x}{z}$
\item[ii.] $\ipr{x}{y} = \ipr{y}{x}^*$
\item[iii.] $\ipr{x}{x} \geq 0 \quad \textnormal{and} \quad \ipr{x}{x} = 0 \ \Rightarrow \ x = 0$
\item[iv.] $\ipr{x}{\vartheta(y,F)} = F^*\ipr{x}{y}$ .
\end{itemize}
For $p \geq 2$, let $\|x\|_\E = \|\ipr{x}{x}^\frac12\|_{p}$ be the norm induced by the inner product.
A pair $(\E,\{\cdot,\cdot\})$, where $\E$ is an $\ALG$-module, $\{\cdot,\cdot\}$ is an $L^\frac{p}{2}(\ALG)$-valued inner product over $\E$, and $(\E,\|\cdot\|_\E)$ is a Banach space, is called an $L^p(\ALG)$-Hilbert module. When $p = 2$ it is a Hilbert space with scalar product
\begin{equation}\label{eq:Xscalarprod}
\langle x, y\rangle_\E = \tau(\ipr{x}{y}).
\end{equation}
\end{defi}
We can now prove that such structures are the ones that arise naturally as images under Helson maps of $(\Gamma,\Pi)$-invariant spaces.
\begin{proof}[Proof of Theorem \ref{theo:Hilmodshiftinv}]
By Theorem \ref{theo:multiplicatively}, the action $\vartheta$ of $\vn(\Gamma)$ on $(M,\ipr{\cdot}{\cdot}_\Kil)$ given by operator composition from the right maps $M$ into itself. Moreover, it satisfies axioms $M1$ to $M5$. Then $\Kil$ is an $\vn(\Gamma)$-module, and the images of $(\Gamma,\Pi)$-invariant subspaces of $\Hil$ under a Helson map $\iso$ are submodules.

Points $i.$ to $iii.$ of Definition \ref{def:Hmodules} are properties of the bracket map. The main point to prove is then Property $iv.$ of the inner product, i.e. that for all $F \in \vn(\Gamma)$ and all $\varPhi, \varPsi \in M$ it holds
$$
\ipr{\varPhi}{\varPsi F}_\Kil = F^* \ipr{\varPhi}{\varPsi}_\Kil .
$$
This is a consequence of Proposition \ref{prop:fundamental}. Indeed, for $\varphi = \iso^{-1}[\varPhi]$, $\psi = \iso^{-1}[\varPsi]$
\begin{displaymath}
\ipr{\varPhi}{\varPsi F}_\Kil = [\varphi, \iso^{-1}[\varPsi F]] = [\varphi, \qroj_F\psi] = F^* [\varphi, \psi] = F^*\ipr{\varPhi}{\varPsi}.
\end{displaymath}
Finally by (\ref{eq:Kilinnerproduct}) and Property III of the bracket map we have that
\begin{equation}\label{eq:Kilnormidentity}
\tau(\ipr{\varPsi}{\varPsi}_\Kil) = \tau([\iso^{-1}[\varPsi],\iso^{-1}[\varPsi]) = \|\iso^{-1}[\varPsi]\|_\Hil^2 = \norm{\varPsi}^2 .
\end{equation}
Thus, the norm induced by $\ipr{\cdot}{\cdot}_\Kil$ coincides with the natural one induced by $L^2(\M,L^2(\vn(\Gamma)))$, so $\Kil$ is a Hilbert space with the scalar product (\ref{eq:Xscalarprod}).
\end{proof}
An analogous consequence of Proposition \ref{prop:fundamental} is the following.
\begin{cor}
Let $(\Gamma,\Pi,\Hil)$ be a dual integrable triple, and let $V$ be a closed subspace of $\Hil$. Then $V$ is $(\Gamma,\Pi)$-invariant if and only if $(V,[\cdot,\cdot])$ is a $L^2(\vn(\Gamma))$-Hilbert module with map $\vartheta$ given by
$$
\vartheta(\psi,F) = \qroj_F \psi .
$$
\end{cor}
Thus $\Hil$ together with the bracket map is an $L^2(\vn(\Gamma))$-Hilbert module, that is unitarily isomorphic to $(\Kil, \ipr{\cdot}{\cdot})$ by means of the Helson map $\iso$.

Since the next treatment will deal either with general $L^2(\vn(\Gamma))$-Hilbert modules or with the modular structure of $\Kil$, for simplicity we will abandon the notation $\vartheta$ and simply indicate the action of $\vn(\Gamma)$ as a right action, i.e. with
$$
\vartheta(\varPsi,F) = \varPsi F.
$$

Modules are linear structures that allow much of the constructions associated to vector spaces, by taking linear combinations with coefficients in the associated algebras. In the present situation, the module structure resulting from Theorem \ref{theo:multiplicatively} would suggest to consider general linear spaces (submodules) generated by elements of $\Kil$ with coefficients in $\vn(\Gamma)$. However, we will see that it is particularly useful to consider coefficients belonging to a larger set of possibly unbounded operators. We will call these coefficients \emph{outer coefficients}, in order to distinguish them from the ordinary ones belonging to $\vn(\Gamma)$. We first define a composition $\varPhi F$ of an element $\varPhi \in \E$ with an element $F \in L^2(\vn(\Gamma),\ipr{\varPhi}{\varPhi})$, which can be done by the density of $\vn(\Gamma)$ in $L^2(\vn(\Gamma),\ipr{\varPhi}{\varPhi})$.

\begin{prop}\label{prop:coefficients}
Let $\E$ be an $L^2(\vn(\Gamma))$-Hilbert module, let $\varPhi \in \E$, and let $F \in L^2(\vn(\Gamma),\ipr{\varPhi}{\varPhi})$. Then for all sequences $\{F_n\}_{n\in\N} \subset \vn(\Gamma)$ converging to $F$ in $L^2(\vn(\Gamma),\ipr{\varPhi}{\varPhi})$ there is a unique element $\varPsi_{\varPhi,F} \in \E$ such that $\{\varPhi F_n\}_{n\in\N} \subset \E$ converges to $\varPsi_{\varPhi,F}$ in $\E$, and satisfies
$$
\|\varPsi_{\varPhi,F}\|_\E = \|F\|_{2,\ipr{\varPhi}{\varPhi}} .
$$
We will denote such element with $\varPsi_{\varPhi,F} = \varPhi F$.
\end{prop}

\begin{proof}
Let $\{F_n\}_{n\in\N} \subset \vn(\Gamma)$ be a sequence converging to $F$ in $L^2(\vn(\Gamma),\ipr{\varPhi}{\varPhi})$. Then $\{\varPhi F_n\}_{n\in\N} \subset \E$ is a Cauchy sequence, because
\begin{align*}
\|\varPhi F_n - \varPhi F_m\|^2_\E & = \tau\Big(\ipr{\varPhi (F_n - F_m)}{\varPhi (F_n - F_m)}\Big)\\
& = \tau \left( |(F_n - F_m)^*|^2 \ipr{\varPhi}{\varPhi}\right) = \|F_n - F_m\|_{2,\ipr{\varPhi}{\varPhi}}^2 .
\end{align*}
Therefore, $\{\varPhi F_n\}_{n\in\N}$ has a limit in $\E$.
The same argument proves the identity of the norms and the uniqueness of the element $\varPhi F$.
\end{proof}

The relevance of considering coefficients in this larger class of unbounded operators can be understood with the following argument. For a given single element $\varPhi \in \E$ consider the set
\begin{equation}\label{eq:single}
M_\varPhi = \{\varPsi \in \E \, | \, \exists\,\, F \in \vn(\Gamma) \, : \, \varPsi = \varPhi F\}.
\end{equation}
This set is not closed in the (Banach) topology of $\E$, because one could take a sequence $\{F_n\}_{n\in\N} \subset \vn(\Gamma)$ converging to some $F \in L^2(\vn(\Gamma),\ipr{\varPhi}{\varPhi})$, so that the sequence $\{\varPsi_n\}_{n\in\N} = \{\varPhi F_n\}_{n\in\N}$ converges to $\varPhi F$ in $\E$, as in Proposition \ref{prop:coefficients}. As a consequence, the smallest closed submodule of $\E$ containing $M_\varPhi$ contains also the element $\varPhi F$, which does not belong to $M_\varPhi$. In particular this implies that, in order to generate a closed submodule from a finite family, one needs either to take the closure of the $\vsp_{\vn(\Gamma)}$ or to consider a span associated to coefficients that do not belong to $\vn(\Gamma)$.
This motivates the following definition.
\begin{defi}\label{def:outerspan}
Let $\E$ be an $L^2(\vn(\Gamma))$-Hilbert module and let $\Phi = \{\varPhi_j\}_{j \in \ind} \subset \E$ be a countable sequence. We will call \emph{outer span} of $\Phi$ the space of their finite linear combinations with coefficients in the associated weighted $L^2$ spaces:
\begin{align*}
\osp\,\Phi = \bigg\{ & \varPsi \in \E \, | \, \exists \ \textnormal{finite sequence} \ \{C_j\}_{j \in \ind} \ \textnormal{such that}\\
& C_j \in L^2(\vn(\Gamma),\ipr{\varPhi_j}{\varPhi_j}) \ \textnormal{and} \ \varPsi = \sum_{j \in \ind} \varPhi_jC_j\bigg\}
\end{align*}
where the compositions $\varPhi_jC_j$ are defined in the sense of Proposition \ref{prop:coefficients}.
\end{defi}
Observe that, for any finite sequence $\Phi = \{\varPhi_j\}_{j \in \ind} \subset \E$, one has that
$$
\osp\,\Phi = \sum_{j \in \ind} \ol{M_{\varPhi_j}}^\E = \ol{\sum_{j \in \ind} M_{\varPhi_j}}^\E \ ,
$$
where we have used the notation (\ref{eq:single}), and the sum of the $M_{\varPhi_j}$ is the usual modular span with coefficients in $\vn(\Gamma)$ of the finite sequence $\Phi$, i.e.
$$
\sum_{j \in \ind} M_{\varPhi_j} = \{\varPsi \in \E \, | \, \exists \ \textnormal{finite} \ \{F_j\}_{j \in \ind} \subset \vn(\Gamma) \, : \, \varPsi = \sum_{j \in \ind} \varPhi_j F_j\} = \vsp_{\vn(\Gamma)}\, \Phi .
$$
For general, possibly infinite, countable families, it can be seen that the closure of the outer span coincides with the closure of the usual modular span, so that
\begin{equation}\label{eq:closedmodularspan}
\ol{\osp\ \Phi}^\E = \ol{\vsp_{\vn(\Gamma)}\, \Phi}^\E
\end{equation}
defines a closed submodule of $\E$, that we will denote with $\E_\Phi$.

Not only the modular structure of $\E$ can be extended to outer coefficients, as in Proposition \ref{prop:coefficients}. Also the Hilbert structure, in the sense that Property $iv.$ of Definition \ref{def:Hmodules} can be obtained as in the following proposition.
\begin{prop}\label{prop:outerinner}
Let $\E$ be an $L^2(\vn(\Gamma))$-Hilbert module, let $\varPhi,\varPsi \in \E$ and let $F \in L^2(\vn(\Gamma),\ipr{\varPhi}{\varPhi})$. Then $\{\varPhi, \varPsi\}F \in L^1(\vn(\Gamma))$ and
\begin{equation}\label{eq:outerinner}
\{\varPhi F, \varPsi\} = \{\varPhi, \varPsi\}F
\end{equation}
where the composition $\varPhi F$ is intended in the sense of Proposition \ref{prop:coefficients}.
\end{prop}
\begin{proof}
By \cite[Prop. 3.2]{JungeSherman05}, for $\varPhi, \varPsi \in \E$ there is a $B \in \vn(\Gamma)$ with $\|B\|_\infty \leq 1$ such that\footnote{Observe that the relationship between the inner product $\langle \cdot, \cdot \rangle$ of \cite{JungeSherman05} and the one of Definition \ref{def:Hmodules} is $\langle x, y\rangle = \ipr{y}{x}$.} $\ipr{\varPhi}{\varPsi} = \ipr{\varPsi}{\varPsi}^\frac12 B \ipr{\varPhi}{\varPhi}^\frac12$. Thus, by H\"older's inequality
\begin{align*}
\|\ipr{\varPhi}{\varPsi} F\|_1 & = \|\ipr{\varPsi}{\varPsi}^\frac12 B \ipr{\varPhi}{\varPhi}^\frac12 F\|_1 \leq \|\ipr{\varPsi}{\varPsi}^\frac12 B\|_2 \ \|\ipr{\varPhi}{\varPhi}^\frac12 F\|_2\\
& \leq \|\varPsi\|_\E \ \|F\|_{2,\ipr{\varPhi}{\varPhi}} .
\end{align*}
Hence $\ipr{\varPhi}{\varPsi}F \in L^1(\vn(\Gamma))$. To prove (\ref{eq:outerinner}), let $\{F_n\}_{n\in\N} \subset \vn(\Gamma)$ be a sequence converging to $F$ in $L^2(\vn(\Gamma),\ipr{\varPhi}{\varPhi})$. By the same argument as before and the definition of $L^2(\vn(\Gamma))$-Hilbert module we have that for all $n \in \N$
$$
\|\ipr{\varPhi F_n}{\varPsi} - \ipr{\varPhi}{\varPsi} F\|_1 = \|\ipr{\varPhi}{\varPsi}F_n - \ipr{\varPhi}{\varPsi} F\|_1 \leq \|\varPsi\|_\E \ \|F_n - F\|_{2,\ipr{\varPhi}{\varPhi}}
$$
so that $\{\ipr{\varPhi F_n}{\varPsi}\}_{n \in \N}$ converges to $\{\varPhi, \varPsi\}F$ in $L^1(\vn(\Gamma))$. On the other hand, by Proposition \ref{prop:coefficients} we have that $\{\varPhi F_n\}_{n \in \N}$ converges to $\varPhi F$ in $\E$, while \cite[Prop. 3.2]{JungeSherman05}
implies that for all $\varPsi \in \E$ the map
$$
\begin{array}{rccl}
b_\varPsi : & \E & \to & L^1(\vn(\Gamma))\\
& \varPhi & \mapsto & \ipr{\varPhi}{\varPsi}
\end{array}
$$
is bounded. Thus $\{\ipr{\varPhi F_n}{\varPsi}\}_{n \in \N}$ converges to $\{\varPhi F, \varPsi\}$ in $L^1(\vn(\Gamma))$ and the conclusion follows by uniqueness of the limit.
\end{proof}

The presence of outer coefficients is not exclusive of this setting. Indeed, they appeared in disguise already in \cite[Th. 2.14]{BoorVoreRon93}, that is one of the fundamental results in the theory of shift-invariant spaces. We can extend it to the present setting, and characterize the multiplier as an outer coefficient, as follows.
\begin{prop}\label{prop:BDR}
Let $(\Gamma,\Pi,\Hil)$ be a dual integrable triple with  associated Helson map $\iso$ and let $\varphi, \psi \in \Hil$. Then $\varphi \in \psis{\psi}$ if and only if there exists $F \in L^2(\vn(\Gamma),[\psi,\psi])$ satisfying
$$
\iso[\varphi] = \iso[\psi]F
$$
and in this case one has $[\varphi,\psi] = [\psi,\psi] F$.
\end{prop}

The proof of this proposition is based on an intermediate result that extends Proposition \ref{prop:fundamental} by making use of Proposition \ref{prop:coefficients} and Corollary \ref{cor:correspondence}.
\begin{lem}\label{lem:coefficients}
Let $\psi \in \Hil$ and let $F \in L^2(\vn(\Gamma),[\psi,\psi])$. Then
\begin{equation}\label{eq:outerHelson}
\iso[\psi] F = \iso[\qroj_F \psi]
\end{equation}
where $\qroj_F \psi$ is given by (\ref{eq:correspondence}), so in particular $\iso[\psi]F \in \iso(\psis{\psi})$.
\end{lem}
\begin{proof}
Let $\{F_n\}_{n\in\N} \subset \vn(\Gamma)$ be a sequence converging to $F$ in the $\|\cdot\|_{2,[\psi,\psi]}$ norm. By Proposition \ref{prop:coefficients} we have that $\iso[\psi]F_n$ converges to $\iso[\psi]F$ in $\Kil$.
On the other hand, Proposition \ref{prop:fundamental} ensures that
$$
\iso[\psi] F_n = \iso[\qroj_{F_n} \psi] \quad \forall \, n
$$
while, by Corollary \ref{cor:correspondence} and the fact that $\iso$ is an isometry, we have that the sequence $\{\iso[\qroj_{F_n} \psi] \}_{n\in\N}$ converges to $\iso[\qroj_{F} \psi] F$ in $\Kil$. This shows (\ref{eq:outerHelson}), because
$$
\norm{\iso[\qroj_{F} \psi] - \iso[\qroj_{F_n} \psi]} = \|\qroj_{F} \psi - \qroj_{F_n} \psi\|_\Hil = \|F - F_n\|_{2,[\psi,\psi]} .
$$
Since, by Corollary \ref{cor:correspondence}, $\qroj_F \psi \in \psis{\psi}$, then $\iso[\psi]F$ belongs to $\iso(\psis{\psi})$.
\end{proof}

\begin{proof}[Proof of Proposition \ref{prop:BDR}]
By Corollary \ref{cor:correspondence} we have that $\varphi \in \psis{\psi}$ if and only if there exists $F \in L^2(\vn(\Gamma),[\psi,\psi])$ such that $\varphi = \proj_F \psi$. The main statement is then a consequence of Lemma \ref{lem:coefficients}. This implies in particular that
$$
\ipr{\iso[\varphi]}{\iso[\psi]}_\Kil = \ipr{\iso[\psi]F}{\iso[\psi]}_\Kil
$$
so that the identity $[\varphi,\psi] = [\psi,\psi] F$ can be deduced by Proposition \ref{prop:outerinner} and the definition of $\Kil$ inner product given in (\ref{eq:Kilinnerproduct}).
\end{proof}

Next theorem extends Proposition \ref{prop:BDR} to countable families of generators. In particular, when the family is finite it shows the role of outer coefficients in order to represent $(\Gamma,\Pi)$-invariant subspaces of the original Hilbert space $\Hil$ as submodules of $\Kil$. As such, it is a direct extension of \cite[Th. 1.7]{BoorVoreRon94}.

\begin{theo}\label{theo:coefficients}
Let $(\Gamma,\Pi,\Hil)$ be a dual integrable triple with Helson map $\iso$. Let $\{\varphi_j\}_{j\in \ind}\subset\Hil$ be a countable family, denote with $\Phi = \{\iso[\varphi_j]\}_{j \in \ind}$ and with $E=\{\Pi(\gamma)\varphi_j:\, \gamma\in\Gamma, j\in\ind\}$, and let $V = \ol{\vsp\,E}^\Hil$. Then
$$
\iso[V] = \E_\Phi
$$
and in particular, when the family $\{\varphi_j\}_{j\in \ind}$ is finite, we have
$$
\iso[V] = \osp\big\{\iso[\varphi_j]:\, j\in \ind\big \}.
$$
\end{theo}
\begin{proof}
By Lemma \ref{lem:coefficients} we have that for each $j\in\ind$
$$
\osp\big\{\iso[\varphi_j]\big\} \subset \iso[\psis{\varphi_j}],
$$
which in particular implies that $\E_\Phi = \ol{\osp\big\{\iso[\varphi_j]:\, j\in \ind\big \}}^{\norm{\cdot} }\subset \iso[V]$.

To show the opposite inclusion, let first take $\varPsi \in \iso[\vsp\,E]$.
Then $\varPsi=\iso[\psi]$ with 
$\psi = \sum_{ j\in \ind} \sum_{\gamma \in \Gamma} c_{j,\gamma} \Pi(\gamma)\varphi_j$ and 
$c = \{c_{j,\gamma}\}_{ j\in \ind, \gamma\in\Gamma}$ being a finite sequence. 
For each $j$ the operator
$C_j = \sum_{\gamma \in \Gamma} c_{j,\gamma} \rr(\gamma)^*$ is a trigonometric polynomial, so that it satisfies
$$
\iso[\varphi_j] C_j = \iso[\proj_{C_j} \varphi_j]
$$
by definition of Helson map. Thus,
\begin{displaymath}
\psi = \sum_{ j\in \ind} \proj_{C_j} \varphi_j = \iso^{-1}\sum_{ j\in \ind} \iso[\varphi_j] C_j
\end{displaymath}
hence proving that $\varPsi \in \vsp_{\vn(\Gamma)}\{\iso[\varphi_j]:\, j\in \ind\}$. Finally we have
\begin{displaymath}
\iso[V]=\iso[\ol{\textnormal{span}\,E}^\Hil] \subset \ol{\iso[\textnormal{span}\,E]}^{\norm{\cdot}} \subset \ol{\vsp_{\vn(\Gamma)}\big\{\iso[\varphi_j]:\, j\in \ind\big \}}^{\norm{\cdot} } = \E_\Phi. \qedhere
\end{displaymath}
\end{proof}

\begin{cor}
Let $(\Gamma,\Pi,\Hil)$ be a dual integrable triple with Helson map $\iso$, let $\Kil = \iso[\Hil]$ and let $M$ be a closed submodule of $\Kil$. Then $M$ admits a countable family $\{\varPhi_j\}_{j \in \ind} \subset M$ of generators.
\end{cor}
\begin{proof}
By Theorem \ref{theo:multiplicatively}, all submodules of $\Kil$ that are closed with respect to the topology of $\Kil$ are the image under $\iso$ of a closed $(\Gamma,\Pi)$-invariant subspace of $\Hil$, so this corollary is a consequence of Lemma \ref{lem:generators} and Theorem \ref{theo:coefficients}.
\end{proof}

\subsection{Fundamental modular operators}

In this section we will introduce the modular analogous of the synthesis, analysis and Gram operators associated to a countable family $\Phi = \{\varPhi_j\}_{j \in \ind}$ in an $L^2(\vn(\Gamma))$-Hilbert module $\E$. We will do this with respect to two different classes of spaces of coefficients. At first we will consider the case of outer coefficients, which does not require any assumption on $\Phi$ and allows to characterize the domains and ranges of the operators in a general way that is analogous to that of Section \ref{sec:boundedness}. Then, we will consider an assumption on the family $\Phi$ that is the weakest possible assumption allowing to obtain a modular Gram operator that is densely defined on a space of unweighted coefficients. This approach will follow that of Section \ref{sec:squareintegrability}, and will provide the basis for the characterizations of the modular analogous of Riesz and frame systems that will be developed in the next section.

\subsubsection{Basic definitions on spaces of outer coefficients}

We introduce now the fundamental operators in terms of the natural spaces of outer coefficients studied in the previous section. In order to do so, we first define some classes of spaces of sequences associated to countable families in $\E$.

\begin{defi}
Let $\E$ be an $L^2(\vn(\Gamma))$-Hilbert module and let $\Phi = \{\varPhi_j\}_{j \in \ind} \subset \E$ be a countable family. For $0 \leq p \leq \infty$ we define the weighted summability spaces
\begin{align*}
\ell_p^{2,\Phi}(\ind,\vn(\Gamma)) = \Big\{ & C = \{C_j\}_{j \in \ind} \, \big| \, C_j \in L^2(\vn(\Gamma),\ipr{\varPhi_j}{\varPhi_j}) \ \forall \, j \in \ind \ \ \textnormal{and}\\
& \big\{\|C_j\|_{2,\ipr{\varPhi_j}{\varPhi_j}}\big\}_{j \in \ind} \in \ell_p(\ind)\Big\} .
\end{align*}
For $p \geq 1$ they are Banach spaces with norm
$$
\|C\|_{\ell_p^{2,\Phi}} =
\left\{
\begin{array}{cc}
\displaystyle\Big(\sum_{j \in \ind}\|C_j\|^p_{2,\ipr{\varPhi_j}{\varPhi_j}}\Big)^{\frac{1}{p}} & 1 \leq p < \infty \vspace{4pt}\\
\displaystyle \sup_{j \in \ind}\Big\{\|C_j\|_{2,\ipr{\varPhi_j}{\varPhi_j}}\Big\} & p = \infty
\end{array}
\right.
$$
satisfying the usual inclusions
$$
\ell_p^{2,\Phi}(\ind,\vn(\Gamma)) \subset \ell_{q}^{2,\Phi}(\ind,\vn(\Gamma)) \, , \quad p \leq q .
$$
With $p = 0$ we mean that the sequence is finite, while $p = 2$ corresponds to a direct sum of Hilbert spaces
$$
\ell_2^{2,\Phi}(\ind,\vn(\Gamma)) = \bigoplus_{j \in \ind} L^2(\vn(\Gamma),\ipr{\varPhi_j}{\varPhi_j}) .
$$
\end{defi}

These spaces provide a natural environment to introduce a synthesis operator as an operator turning coefficients into linear combinations.
\begin{defi}
The \emph{modular synthesis operator} associated to a countable family $\Phi = \{\varPhi_j\}_{j \in \ind} \subset \E$ is the operator
$$
\begin{array}{rccl}
\mSy_\Phi : & \ell_0^{2,\Phi}(\ind,\vn(\Gamma)) & \to &  \osp \,\Phi\\
& C & \mapsto & \displaystyle\sum_{j \in \ind} \varPhi_j C_j .
\end{array}
$$
\end{defi}

The modular synthesis operator can be extended to infinite sequences belonging to $\ell_1^{2,\Phi}(\ind,\vn(\Gamma))$, as in the setting of Section \ref{sec:boundedness}. However, since we are using weighted spaces, no further assumption is needed here.
\begin{lem}
The modular synthesis operator $\mSy_\Phi$ associated to a countable family $\Phi = \{\varPhi_j\}_{j \in \ind} \subset \E$ extends to a bounded operator from $\ell_1^{2,\Phi}(\ind,\vn(\Gamma))$ to $\E$, satisfying
$$
\|\mSy_\Phi C\|_\E \leq \|C\|_{\ell_1^{2,\Phi}} .
$$
\end{lem}
\begin{proof}
This is a direct consequence of Proposition \ref{prop:coefficients}. Indeed,
\begin{displaymath}
\|\mSy_\Phi C\|_\E = \|\sum_{j \in \ind} \varPhi_j C_j\|_\E \leq \sum_{j \in \ind} \|\varPhi_j C_j\|_\E = \sum_{j \in \ind} \|C_j\|_{2,\ipr{\varPhi_j}{\varPhi_j}} . \qedhere
\end{displaymath}
\end{proof}

In analogy with Section \ref{sec:boundedness}, one can introduce the dual operator of $\mSy_\Phi$ as the operator $\mAn_\Phi : \E^* \to (\ell_1^{2,\Phi}(\ind,\vn(\Gamma)))^*$ which, when applied to $\varPsi \in \E \approx \E^*$, acts linearly on $C \in \ell_1^{2,\Phi}(\ind,\vn(\Gamma))$ as
\begin{equation}\label{eq:properlydefmodularanalysis}
(\mAn_\Phi \varPsi) C = \langle \mSy_\Phi C, \varPsi\rangle_\E = \tau\Big(\ipr{\sum_{j \in \ind} \varPhi_j C_j}{\varPsi}\Big) = \tau\Big(\sum_{j \in \ind} \ipr{\varPhi_j}{\varPsi}C_j\Big),
\end{equation}
where the last identity is due to Proposition \ref{prop:outerinner}. This leads to the following definitions of modular analysis and modular Gram operators.
\begin{defi}\label{def:modularAG}
We will call \emph{modular analysis operator} $\mAn_\Phi$ associated to a countable family $\Phi = \{\varPhi_j\}_{j \in \ind} \subset \E$ the bounded map
$$
\begin{array}{rccl}
\mAn_\Phi : & \E & \to & (\ell_1^{2,\Phi}(\ind,\vn(\Gamma)))^*\vspace{4pt}\\
& \varPsi & \mapsto & \Big\{\ipr{\varPsi}{\varPhi_j}\Big\}_{j \in \ind}
\end{array}
$$
where we have identified the sequence $\Big\{\ipr{\varPsi}{\varPhi_j}\Big\}_{j \in \ind}$ with the linear functional (\ref{eq:properlydefmodularanalysis}). We will call \emph{modular Gram operator} the composition of modular synthesis and analysis operators
$$
\begin{array}{rccl}
\mGrop_\Phi = \mAn_\Phi \mSy_\Phi : \!\!\!& \ell_1^{2,\Phi}(\ind,\vn(\Gamma)) & \!\!\!\!\to \!\!\!\!& (\ell_1^{2,\Phi}(\ind,\vn(\Gamma)))^*\vspace{4pt}\\
& C & \!\!\!\!\mapsto\!\!\!\! & \Big\{\ipr{\sum_k \varPhi_k C_k}{\varPhi_j}\Big\}_{j \in \ind} \!\!= \Big\{\sum_k \ipr{\varPhi_k}{\varPhi_j}C_k\Big\}_{j \in \ind} .
\end{array}
$$
\end{defi}

As for the usual Hilbert space setting, the modular Gram operator allows to define a quadratic form which characterizes the domain of the modular synthesis operator. The resulting structure in this setting is not only a Hilbert space but an $L^2(\vn(\Gamma))$-Hilbert module.
\begin{lem}\label{lem:ncquadratic}
Let $\E$ be an $L^2(\vn(\Gamma))$-Hilbert module, and let $\Phi = \{\varPhi_j\}_{j \in \ind} \subset \E$ be a countable family. The map
$$
\begin{array}{rccl}
\ipr{\cdot}{\cdot}_{\mGrop_\Phi} : & \ell_1^{2,\Phi}(\ind,\vn(\Gamma)) \times \ell_1^{2,\Phi}(\ind,\vn(\Gamma)) & \to & L^1(\vn(\Gamma))\vspace{4pt}\\
& (C,D) & \mapsto & \displaystyle\sum_{j \in \ind} D_j^* (\mGrop_\Phi C)_j
\end{array}
$$
is an $L^1(\vn(\Gamma))$-valued inner product in the sense of Definition \ref{def:Hmodules}, satisfying
\begin{equation}\label{eq:modularisometry}
\tau(\ipr{C}{C}_{\mGrop_\Phi}) = \|\mSy_\Phi C\|^2_\E .
\end{equation}
\end{lem}
\begin{proof}
The argument is based on Proposition \ref{prop:outerinner}, which implies
\begin{align*}
\ipr{C}{D}_{\mGrop_\Phi} & = \sum_{j \in \ind} D_j^* (\mGrop_\Phi C)_j = \sum_{j,k \in \ind}D_j^* \ipr{\varPhi_k}{\varPhi_j}C_k = \big\{\sum_{k \in \ind} \varPhi_k C_k , \sum_{j \in \ind} \varPhi_j D_j\big\}\\
& = \ipr{\mSy_\Phi C}{\mSy_\Phi D}
\end{align*}
so that all the properties of an $L^1(\vn(\Gamma))$-valued inner product are inherited from those of $\ipr{\cdot}{\cdot}$. The identity (\ref{eq:modularisometry}) is the definition of the $\E$ norm.
\end{proof}

If we denote with $\|\cdot\|_{\mGrop_\Phi}$ the seminorm on $\ell_1^{2,\Phi}(\ind,\vn(\Gamma))$ given by (\ref{eq:modularisometry}), i.e.
\begin{equation}\label{eq:mGropNorm}
\|C\|_{\mGrop_\Phi} = \|\mSy_\Phi C\|_\E
\end{equation}
and we denote with $\mathcal{N}$ its null space 
$$
\mathcal{N} = \{C \in \ell_1^{2,\Phi}(\ind,\vn(\Gamma)) \, | \, \|C\|_{\mGrop_\Phi} = 0\}
$$
we can obtain the modular analogous of Lemma \ref{lem:synthesisdomain}, in the following Theorem.

\begin{theo}
Let $\Phi = \{\varPhi_j\}_{j \in \ind} \subset \E$ be a countable family, and let $\Hil_{\mGrop_\Phi}$ be the Hilbert space
$$
\Hil_{\mGrop_\Phi} = \ol{\ell_1^{2,\Phi}(\ind,\vn(\Gamma))/\mathcal{N}}^{\|\cdot\|_{\mGrop_\Phi}} .
$$
Then $(\Hil_{\mGrop_\Phi},\ipr{\cdot}{\cdot}_{\mGrop_\Phi})$ is an $L^2(\vn(\Gamma))$-Hilbert module, and the modular synthesis operator defines an isomorphism of $L^2(\vn(\Gamma))$-Hilbert modules
\begin{equation}\label{eq:modularsynthesisdomain}
\mSy_\Phi : \Hil_{\mGrop_\Phi} \to \E_\Phi
\end{equation}
where $\E_\Phi$ is defined as in (\ref{eq:closedmodularspan}).
\end{theo}
\begin{proof}
Since $\ell_1^{2,\Phi}(\ind,\vn(\Gamma))$ is a module, also $\Hil_{\mGrop_\Phi}$ is a module. Moreover, by Lemma \ref{lem:ncquadratic} it is endowed with the $L^1(\vn(\Gamma))$-valued inner product $\ipr{\cdot}{\cdot}_{\Grop_\Phi}$, and this induces a scalar product that coincides with the polarization of the norm (\ref{eq:mGropNorm}), which makes it a Hilbert space. So it is an $L^2(\vn(\Gamma))$-Hilbert module.

In order to see that $\mSy_\Phi$ defines a surjective isometry as in (\ref{eq:modularsynthesisdomain}), observe first that, from the identity (\ref{eq:modularisometry}) we obtain that $\mathcal{N} = \Ker(\mSy_\Phi)$, so that (\ref{eq:modularsynthesisdomain}) holds, and we also obtain that $\mSy_\Phi$ is an isometry. The argument to prove surjectivity is the same used in Lemma \ref{lem:synthesisdomain}.
\end{proof}

\subsubsection{Densely defined operators \texorpdfstring{in $\ell_2(\ind,L^2(\vn(\Gamma)))$}{}}

In this and the following sections, for a given Banach space $V$ we will denote with $\ell_p(\ind,V)$ the associated vector-valued $\ell_p$ spaces. A particularly relevant space is $\ell_2(\ind,L^2(\vn(\Gamma)))$, which is actually endowed with a natural $L^2(\vn(\Gamma))$-Hilbert module structure that generalizes the one of standard countably generated Hilbert $C^*$-modules. It is called \emph{principal} in \cite[\S 3]{JungeSherman05}. In this section we will provide the minimal assumption in order for the modular Gram operator to be densely defined in $\ell_2(\ind,L^2(\vn(\Gamma)))$. We resume the basic properties that we will need in the following proposition, referring to \cite{JungeSherman05} for proof and details.
\begin{prop}
The separable Hilbert space
$
\ell_2(\ind,L^2(\vn(\Gamma))) = \displaystyle\bigoplus_{j \in \ind} L^2(\vn(\Gamma))
$
is an $L^2(\vn(\Gamma))$-Hilbert module with inner product
$$
\ipr{C}{D}_2 = \sum_{j \in \ind} D_j^* C_j \ , \quad C, D \in \ell_2(\ind,L^2(\vn(\Gamma))) .
$$
\end{prop}

The condition that we will need in this subsection is the following. We say that a countable family $\Phi = \{\varPhi_j\}_{j \in \ind}$ in an $L^2(\vn(\Gamma))$-Hilbert module $\E$ is \emph{modular square integrable} if
\begin{equation}\label{eq:modularsquareintegrability}
T_j = \sum_{k \in \ind} |\ipr{\varPhi_j}{\varPhi_k}|^2 \in L^1(\vn(\Gamma)) \quad \forall \, j \in \ind .
\end{equation}

\begin{theo}\label{theo:modulardenseGram}
The modular Gram operator $\mGrop_\Phi$ associated to a countable family $\Phi = \{\varPhi_j\}_{j \in \ind} \subset \E$ is densely defined from $\ell_2(\ind,L^2(\vn(\Gamma)))$ to itself, and its domain contains $\ell_0(\ind,\vn(\Gamma))$, if and only if $\Phi$ is modular square integrable.
\end{theo}

The proof strongly relies on the following modular analogous of Lemma \ref{lem:generalsquareintegrability} that proves that having a modular analysis operator that maps $\vsp_{\vn(\Gamma)}\, \Phi$ into $\ell_2(\ind,L^2(\vn(\Gamma)))$ is equivalent to the modular square integrability condition.
\begin{lem}\label{lem:modularanalysisdense}
The modular analysis operator $\mAn_\Phi$ for $\Phi = \{\varPhi_j\}_{j \in \ind} \subset \E$ sends $\vsp_{\vn(\Gamma)}\, \Phi$ to $\ell_2(\ind,L^2(\vn(\Gamma)))$ if and only if condition (\ref{eq:modularsquareintegrability}) holds. 
\end{lem}
\begin{proof}
Let $\Lambda \subset \ind$ be a finite set of indices, let $\varPsi = \sum_{j \in \Lambda} \varPhi_j C_j \in \vsp_{\vn(\Gamma)}\, \Phi$ for some $\{C_j\}_{j \in \Lambda} \in \ell_0(\ind,\vn(\Gamma))$, and let us first assume (\ref{eq:modularsquareintegrability}). Then\footnote{Observe that the operator modulus satisfies the triangular inequality $|A + B| \leq |A| + |B|$ by definition, and $(|A| + |B|)^2 \leq 2(|A|^2 + |B|^2)$ holds because
$$
\begin{array}{r}
(|A| + |B|)^2 = |A|^2 + |B|^2 + |A||B| + |B||A|\\
(|A| - |B|)^2 = |A|^2 + |B|^2 - |A||B| - |B||A|
\end{array}
\Rightarrow \ (|A| + |B|)^2 + (|A| - |B|)^2 = 2(|A|^2 + |B|^2) .
$$
}
\begin{align*}
\|\mAn_\Phi \varPsi&\|^2_{\ell_2(\ind,L^2(\vn(\Gamma)))} = \tau\Big(\sum_{k \in \ind} \Big| \ipr{\sum_{j \in \Lambda} \varPhi_j C_j}{\varPhi_k}\Big|^2\Big) = \tau\Big(\sum_{k \in \ind} \Big| \sum_{j \in \Lambda} \ipr{\varPhi_j}{\varPhi_k}C_j\Big|^2\Big)\\
& \leq 2^{|\Lambda|} \tau\Big(\sum_{k \in \ind} \sum_{j \in \Lambda} | \ipr{\varPhi_j}{\varPhi_k}C_j |^2\Big) = 2^{|\Lambda|} \sum_{k \in \ind} \sum_{j \in \Lambda} \tau(| \ipr{\varPhi_j}{\varPhi_k}C_j|^2)\\
& \leq 2^{|\Lambda|} \sum_{j \in \Lambda} \sum_{k \in \ind} \|C_j\|^2_\infty \|\ipr{\varPhi_j}{\varPhi_k}\|^2_2 = 2^{|\Lambda|} \sum_{j \in \Lambda} \|C_j\|^2_\infty \|T_j\|_1\\
& \leq |\Lambda| 2^{|\Lambda|} \ \max_{j \in \Lambda} \|C_j\|^2_{\infty} \ \max_{j \in \Lambda} \|T_j\|_1 < \infty .
\end{align*}
On the other hand, if $\mAn_\Phi : \vsp_{\vn(\Gamma)}\, \Phi \to \ell_2(\ind,L^2(\vn(\Gamma)))$, then
$$
\|\mAn_\Phi \varPhi_j\|^2_{\ell_2(\ind,L^2(\vn(\Gamma)))} = \tau(T_j) < \infty \quad \forall j \in \ind
$$
which coincides with condition (\ref{eq:modularsquareintegrability}).
\end{proof}

\begin{proof}[Proof of Theorem \ref{theo:modulardenseGram}]
Observe first that $\ell_0(\ind,\vn(\Gamma))$ is dense in $\ell_2(\ind,L^2(\vn(\Gamma)))$, and that the modular synthesis operator of a family $\Phi = \{\varPhi_j\}_{j \in \ind} \subset \E$ sends finite sequences of $\vn(\Gamma)$ coefficients into the $\vn(\Gamma)$-linear span of $\Phi$
\begin{equation}\label{eq:mSydense}
\mSy_\Phi : \ell_0(\ind,\vn(\Gamma)) \to \vsp_{\vn(\Gamma)}\, \Phi .
\end{equation}
Then the proof of the Theorem follows then by Lemma \ref{lem:modularanalysisdense} and the definition of modular Gram operator.
\end{proof}

Finally we observe that, even if $\mSy_\Phi$ is densely defined from $\ell_2(\ind,L^2(\vn(\Gamma)))$ to $\E$ in the sense of (\ref{eq:mSydense}), one can not obtain in this space a modular analogous of the useful classical result given by \cite[Lemma 3.2.1]{Christensen03} without assuming that the weights $\{\ipr{\varPhi_j}{\varPhi_j}\}_{j \in \ind}$ belong to $\vn(\Gamma)$. In the next section we will discuss how the modular analogous of the Riesz and Bessel condition, that are stronger than modular square integrability, allow to replace the weighted spaces of coefficients with $\ell_2(\ind,L^2(\vn(\Gamma)))$. Here we conclude with the following simple observation.
\begin{cor}\label{cor:modularWellDefined}
Let $\Phi = \{\varPhi_j\}_{j \in \ind}\subset \E$ be a modular square integrable countable family, let $\Proj$ be an orthogonal projection of $\ell_2(\ind,L^2(\vn(\Gamma)))$. If
$$
\langle C , \mGrop_\Phi C\rangle_{\ell_2(\ind,L^2(\vn(\Gamma)))} \leq \langle C , \Proj C\rangle_{\ell_2(\ind,L^2(\vn(\Gamma)))} \quad \forall \, C \in \ell_0(\ind,\vn(\Gamma))
$$
then the modular synthesis operator $\mSy_\Phi$ is bounded from $\ell_2(\ind,L^2(\vn(\Gamma)))$ to $\E$.
\end{cor}
\begin{proof}
This is a consequence of the density of $\ell_0(\ind,\vn(\Gamma))$ in $\ell_2(\ind,L^2(\vn(\Gamma)))$ and of identity (\ref{eq:modularisometry}), observing that
\begin{displaymath}
\|\mSy_\Phi C\|^2_\E = \sum_{i, j \in \ind}\tau(C_j^*\ipr{\varPhi_i}{\varPhi_j}C_i) = \langle C , \mGrop_\Phi C \rangle_{\ell_2(\ind,L^2(\vn(\Gamma)))} .\qedhere
\end{displaymath}
\end{proof}

\section{Noncommutative reproducing systems}\label{sec:noncommutativereproducing}

Frames in Hilbert $C^*$-modules were introduced in \cite{FrankLarson2002}, following works of \cite{DaiLarson98}, and constitute the content of several subsequent works (see e.g. \cite{RaeburnThompson03, Jing06, HanJingLarsonMohapatra08, KaftalLarsonZhang09}). These notions were also developed to address the study of multiresolution analysis and of spaces invariant under unitary representations in \cite{PackerRieffel04, Roysland11}. The Hilbert modular structure considered in the present paper is different, because the inner product does not take values in the algebra of coefficients but rather in a larger space of unbounded operators. While several aspects of the $C^*$ theory are preserved in this setting, others require a different treatment mainly due to the appearance of unbounded operators. The main advantage, which is also the leading motivation to develop this framework, is the natural correspondence with $(\Gamma,\Pi)$-invariant spaces, which also gives the possibility to include previous results concerning discrete shifts in locally compact abelian groups, together with the associated notions and techniques of shift-invariant spaces. In particular, we will prove a characterization of reproducing systems in invariant spaces in terms of their analogous modular conditions on the generators.

\subsection{Riesz and frame modular sequences}

Let $(\E,\{\cdot,\cdot\})$ be an $L^2(\vn(\Gamma))$-Hilbert module, and let $\Phi = \{\varPhi_j\}_{j \in \ind} \subset \E$ be a countable family. We introduce the following notion of Riesz sequence of operators in such modules as direct generalization of the usual one.

\begin{defi}
We say that $\Phi = \{\varPhi_j\}_{j \in \ind} \subset \E$ is a \emph{modular Riesz sequence} with Riesz bounds $0 < A \leq B < \infty$ if it satisfies
\begin{equation}\label{eq:ncRiesz}
A \sum_{j \in \ind} |C_j|^2 \leq \bigg\{\sum_{j\in \ind} \varPhi_j C_j , \sum_{j \in \ind} \varPhi_j C_j\bigg\} \leq B \sum_{j \in \ind} |C_j|^2
\end{equation}
for all finite sequence $\{C_j\}_{j \in \ind}$ in $\ell_0(\ind,\vn(\Gamma))$.
\end{defi}
The modular Riesz condition implies the following numerical condition, which shows in particular that the Hilbert space $\ell_2(\ind,L^2(\vn(\Gamma)))$ is the natural space of coefficients for modular Riesz sequences.
\begin{lem}\label{lem:ncRiesz}
If $\Phi = \{\varPhi_j\}_{j \in \ind} \subset \E$ is a modular Riesz sequence with Riesz bounds $A$ and $B$, then the modular synthesis operator $\mSy_\Phi$ extends to a bounded invertible operator from $\ell_2(\ind,L^2(\vn(\Gamma)))$ to $\E_\Phi$ satisfying
\begin{equation}\label{eq:ncRiesznumerical}
A \|C\|^2_{\ell_2(\ind,L^2(\vn(\Gamma)))} \leq \|\mSy_\Phi C\|^2_\E \leq B \|C\|^2_{\ell_2(\ind,L^2(\vn(\Gamma)))}
\end{equation}
for all $C \in \ell_2(\ind,L^2(\vn(\Gamma)))$.
\end{lem}
\begin{proof}
By taking the trace in the modular Riesz condition (\ref{eq:ncRiesz}), and using the definition of norm in the $L^2(\vn(\Gamma))$-Hilbert modules $\E$ and $\ell_2(\ind,L^2(\vn(\Gamma)))$, one gets condition (\ref{eq:ncRiesznumerical}) for all $\{C_j\}_{j \in \ind}$ in $\ell_0(\ind,\vn(\Gamma))$. The conclusion then follows by the density of $\ell_0(\ind,\vn(\Gamma))$ in $\ell_2(\ind,L^2(\vn(\Gamma)))$.
\end{proof}

Next lemma shows that the boundedness of the modular synthesis operator allows to prove that the modular analysis operator is its adjoint.

\begin{lem}\label{lem:Rieszmodadj}
Let $\Phi = \{\varPhi_j\}_{j \in \ind} \subset \E$ be a modular Riesz sequence. Then the adjoint of the modular synthesis operator $\mSy_\Phi$ is the bounded operator from $\E$ to $\ell_2(\ind,L^2(\vn(\Gamma)))$ given by
\begin{equation}\label{eq:ncRieszadj}
\mAn_\Phi \varPsi = \big\{\ipr{\varPsi}{\varPhi_j}\big\}_{j \in \ind} \ , \quad \varPsi \in \E .
\end{equation}
\end{lem}
\begin{proof}
Suppose that $\Phi = \{\varPhi_j\}_{j \in \ind} \subset \E$ is a modular Riesz sequence with Riesz bounds $A$ and $B$. 
The modular Riesz condition implies in particular that
$$
A \Id \leq \ipr{\varPhi_j}{\varPhi_j} \leq B \Id \quad \forall \, j\in\ind
$$
where $\Id$ stands for the identity operator on $\ell_2(\Gamma)$. It can be obtained from (\ref{eq:ncRiesz}) when $\{C_j\}_{j \in \ind}$ consists of only one element, and this element is the identity operator. Then, in particular, $s_{\ipr{\varPhi_j}{\varPhi_j}} = \Id$ so, by Lemma \ref{lem:weightmap}, we have
\begin{equation}\label{eq:ncRieszbrackets}
L^2(\vn(\Gamma),\ipr{\varPhi_j}{\varPhi_j}) \approx L^2(\vn(\Gamma)) \quad \textnormal{for all} \ j \in \ind
\end{equation}
in the sense that the two spaces coincide as sets and their norms are equivalent up to constants $A$ and $B$. Now, by Lemma \ref{lem:ncRiesz} the modular synthesis operator $\mSy_\Phi$ extends to a bounded operator between the Hilbert spaces $\ell_2(\ind,L^2(\vn(\Gamma)))$ and $\E$. Its adjoint operator $\mAn_\Phi$ satisfies
$$
\langle \mAn_\Phi \varPsi, C\rangle_{\ell_2(\ind,L^2(\vn(\Gamma)))} = \langle \varPsi, \mSy_\Phi C\rangle_\E
$$
for all $\varPsi \in \E$ and all $C \in \ell_2(\ind,L^2(\vn(\Gamma)))$. Therefore,
\begin{align*}
\langle \mAn_\Phi \varPsi, C\rangle_{\ell_2(\ind,L^2(\vn(\Gamma)))} & = \langle \varPsi, \sum_{j \in \ind} \varPhi_j C_j\rangle_\E = \tau\Big(\ipr{\varPsi}{\sum_{j \in \ind} \varPhi_j C_j}\Big)\\
& = \tau\Big(\sum_{j \in \ind} \ipr{\varPsi}{\varPhi_jC_j}\Big) = \tau\Big(\sum_{j \in \ind} C_j^*\ipr{\varPsi}{\varPhi_j}\Big)
\end{align*}
where the last identity is due to Proposition \ref{prop:outerinner}, which can be applied due to (\ref{eq:ncRieszbrackets}). Since $\mSy_\Phi$ is bounded, the sequence $\big\{\ipr{\varPsi}{\varPhi_j}\big\}_{j \in \ind}$ belongs to $\ell_2(\ind,L^2(\vn(\Gamma)))$.
\end{proof}

A similar modular generalization of the notion of Bessel sequence is the following.
\begin{defi}
We say that $\Phi = \{\varPhi_j\}_{j \in \ind} \subset \E$ is a \emph{modular Bessel sequence} if there exists $B > 0$ such that
\begin{equation}\label{eq:ncBessel}
\sum_{j \in \ind} |\ipr{\varPsi}{\varPhi_j}|^2 \leq B \ipr{\varPsi}{\varPsi} \quad \forall \, \varPsi \in \E_\Phi, \vspace{-4pt}
\end{equation}
where the convergence of the series is in $L^1(\vn(\Gamma))$.
\end{defi}

\begin{rem}\label{rem:ncBesselextension}
One can not extend the modular Bessel condition from the (outer) span to the whole space relying on the same argument used in the usual Hilbert space case. Indeed, in Hilbert spaces, if $\sum_j |\langle \psi, \phi_j\rangle_\Hil|^2 \leq B \|\psi\|^2$ for all $\psi \in \Hil_\Phi = \csp\{\phi_j\}_j$, then the same holds for any $\varphi \in \Hil$ because $\langle \varphi, \phi_j\rangle_\Hil = \langle \Proj_{\Hil_\Phi}\varphi, \phi_j\rangle_\Hil$, and $\|\Proj_{\Hil_\Phi}\varphi\|^2 \leq \|\varphi\|^2$. But in a Hilbert module the notion of orthogonality provided by the inner product $\ipr{\cdot}{\cdot}$ is weaker than the one provided by the scalar product $\langle \cdot, \cdot \rangle_\E$, and the whole module is not in general a sum of orthogonal complements (as for Hilbert $C^*$-modules, see e.g. \cite{Lance95}). More precisely, if $M \subset \E$ is a closed submodule, define
$$
M^{\bot} = \{\varPsi \in \E \,|\, \langle \varPsi , \varPhi \rangle_\E = 0 \ \forall \, \varPhi \in M\} \, , \ M^{\bot}_{\ipr{}{}} = \{\varPsi \in \E \,|\, \ipr{\varPsi}{\varPhi} = 0 \ \forall \, \varPhi \in M\} .
$$
Then $M^{\bot}_{\ipr{}{}} \subset M^\bot$ since by definition $\ipr{\varPsi}{\varPhi} = 0$ implies $\langle \varPsi , \varPhi \rangle_\E = 0$, but the converse is not true in general because $\tau(F) = 0$ does not imply $F = 0$.
\end{rem}

The modular Bessel condition clearly implies modular square integrability, and is also strictly stronger than the ordinary Bessel condition.
\begin{cor}\label{cor:modBessel}
Let $\Phi = \{\varPhi_j\}_{j \in \ind} \subset \E$ be a modular Bessel sequence. Then it is an ordinary Bessel sequence in the Hilbert space $\E$ with the same constant.
\end{cor}
\begin{proof}
By H\"older inequality and the finiteness of $\tau$, we have
$$
|\langle \varPsi, \varPhi_j\rangle_\E| = |\tau(\ipr{\varPsi}{\varPhi_j})| \leq \tau(|\ipr{\varPsi}{\varPhi_j}|) \leq \tau(|\ipr{\varPsi}{\varPhi_j}|^2)^\frac12 \quad \forall \ \varPsi \in \E .
$$
Then, for all $\varPsi \in \E_\Phi$ we have
$$
\sum_{j \in \ind} |\langle \varPsi, \varPhi_j\rangle_\E|^2 \leq \sum_{j \in \ind} \tau(|\ipr{\varPsi}{\varPhi_j}|^2) \leq B \tau(\ipr{\varPsi}{\varPsi}) = B \|\varPsi\|_\E^2
$$
where the last inequality is obtained by taking the trace in the modular Bessel condition (\ref{eq:ncBessel}). So $\Phi$ is an ordinary Bessel sequence in the Hilbert space $\E_\Phi$, that can be extended to the whole $\E$ by standard arguments, see Remark \ref{rem:ncBesselextension}.
\end{proof}
\newpage

One important property of modular Bessel sequences is that the modular analysis and synthesis operators are bounded between the Hilbert spaces $\E_\Phi$ and $\ell_2(\ind,L^2(\vn(\Gamma)))$, and are the adjoint of one another. 
\begin{theo}\label{theo:ncBessel}
Let $\Phi = \{\varPhi_j\}_{j \in \ind} \subset \E$ be a modular Bessel sequence. Then
\begin{itemize}
\item[i.] the modular analysis operator $\mAn_\Phi$ is bounded from $\E_\Phi$ to $\ell_2(\ind,L^2(\vn(\Gamma)))$
\item[ii.] the modular synthesis operator $\mSy_\Phi$ extends to a bounded operator from $\ell_2(\ind,L^2(\vn(\Gamma)))$ to $\E_\Phi$ that concides with the adjoint of $\mAn_\Phi$.
\end{itemize}
\end{theo}
\begin{proof}
The Bessel condition implies $\ipr{\varPhi_j}{\varPhi_j} \in \vn(\Gamma)$, with $\|\ipr{\varPhi_j}{\varPhi_j}\|_\infty \leq B$, for all $j \in \ind$, because (see also the argument in \cite[Proof of Th. A]{BHP14})
$$
\ipr{\varPhi_{j}}{\varPhi_{j}}^2 \leq \sum_{l \in \ind} |\ipr{\varPhi_{j}}{\varPhi_l}|^2 \leq B \ipr{\varPhi_{j}}{\varPhi_{j}} .
$$
Then, in particular, $L^2(\vn(\Gamma)) \subset L^2(\vn(\Gamma),\ipr{\varPhi_j}{\varPhi_j})$ for all $j \in \ind$, because if $F \in L^2(\vn(\Gamma))$ then by Holder's inequality we have that $\|F\|_{2,\ipr{\varPhi_j}{\varPhi_j}} \leq \sqrt{B} \|F\|_2$. Thus for all $p \geq 0$ we have
\begin{equation}\label{eq:Besselcoefficients}
\ell_p(\ind,L^2(\vn(\Gamma))) \subset \ell_p^{2,\Phi}(\ind,\vn(\Gamma)) .
\end{equation}

That $\mAn_\Phi : \E_\Phi \to \ell_2(\ind,L^2(\vn(\Gamma)))$ is bounded is a direct consequence of the Bessel condition, since, by Fubini's Theorem and the monotonicity of $\tau$
\begin{align*}
\|\mAn_\Phi\varPsi\|^2_{\ell_2(\ind,L^2(\vn(\Gamma)))} & = 
\sum_{j \in \ind} \tau\Big(|\ipr{\varPsi}{\varPhi_j}|^2\Big) = \tau \bigg( \sum_{j \in \ind} |\ipr{\varPsi}{\varPhi_j}|^2\bigg)\\
& \leq B \tau(\ipr{\varPsi}{\varPsi}) = B \|\varPsi\|_\E^2 .
\end{align*}
Let us now denote with $(\mAn_\Phi)^* : \ell_2(\ind,L^2(\vn(\Gamma))) \to \E_\Phi$ the adjoint of $\mAn_\Phi$. By definition, for all $\varPsi \in \E$ and all $C \in \ell_2(\ind,L^2(\vn(\Gamma)))$, we have
\begin{align}\label{eq:intermezzo2}
\langle \varPsi, (\mAn_\Phi)^* C\rangle_\E & = \langle \mAn_\Phi \varPsi, C\rangle_{\ell_2(\ind,L^2(\vn(\Gamma)))} = \sum_{j \in \ind} \tau\bigg( C_j^*\ipr{\varPsi}{\varPhi_j}\bigg)\nonumber\\
& = \tau\bigg(\sum_{j \in \ind} \ipr{\varPsi}{\varPhi_j C_j}\bigg)
\end{align}
where the last identity is due to Fubini's Theorem and Proposition \ref{prop:outerinner}. When $C \in \ell_0(\ind,L^2(\vn(\Gamma)))$, identity (\ref{eq:intermezzo2}) implies
$$
\langle \varPsi, (\mAn_\Phi)^* C\rangle_\E = \tau\bigg(\Big\{\varPsi , \sum_{j \in \ind} \varPhi_j C_j\Big\}\bigg) = \langle \varPsi, \mSy_\Phi C\rangle_\E
$$
so that $(\mAn_\Phi)^*$ coincides with $\mSy_\Phi$ on $\ell_0(\ind,L^2(\vn(\Gamma)))$, where $\mSy_\Phi$ is defined because of (\ref{eq:Besselcoefficients}). Since this is a dense subset of $\ell_2(\ind,L^2(\vn(\Gamma)))$, then $(\mAn_\Phi)^*$ provides a bounded extension of $\mSy_\Phi$, that we will still denote with the same symbol $\mSy_\Phi$.
\end{proof}

We now can prove that the modular Bessel property implies that the modular analysis operator is adjointable as an operator between the Hilbert modules $(\E_\Phi,\ipr{\cdot}{\cdot})$ and $(\ell_2(\ind,L^2(\vn(\Gamma))),\ipr{\cdot}{\cdot}_2)$, and its modular adjoint is the bounded extension of the modular synthesis operator provided by Theorem \ref{theo:ncBessel}.

\begin{lem}\label{lem:modadj}
If $\Phi = \{\varPhi_j\}_{j \in \ind} \subset \E$ is a modular Bessel sequence, then
$$
\ipr{\mAn_\Phi \varPsi}{C}_{2} = \ipr{\varPsi}{\mSy_\Phi C} \quad \forall \, \varPsi \in \E_\Phi , \ \forall \, C \in \ell_2(\ind,L^2(\vn(\Gamma))) .
$$
\end{lem}
\begin{proof}
By (\ref{eq:Besselcoefficients}) and Proposition \ref{prop:outerinner}, for all $\varPsi \in \E_\Phi$ and all $C \in \ell_2(\ind,L^2(\vn(\Gamma)))$ we have that $C_j^*\ipr{\varPsi}{\varPhi_j} = \ipr{\varPsi}{\varPhi_j C_j}$ for all $j$, so
$$
\ipr{\mAn_\Phi \varPsi}{C}_{2} = \sum_{j \in \ind} C_j^*\ipr{\varPsi}{\varPhi_j} = \sum_{j \in \ind} \ipr{\varPsi}{\varPhi_j C_j} .
$$
Let $\{C^{(k)}\}_{k \in \N} \subset \ell_0(\ind,L^2(\vn(\Gamma)))$ be a sequence converging to $C$ in $\ell_2(\ind,L^2(\vn(\Gamma)))$, where each $C^{(k)} = \{C^{(k)}_j\}_{j \in \ind}$ is a finite sequence of $L^2(\vn(\Gamma))$ elements.
Then
$$
\ipr{\mAn_\Phi \varPsi}{C^{(k)}}_{2} = \sum_{j \in \ind} \ipr{\varPsi}{\varPhi_j C^{(k)}_j} = \ipr{\varPsi}{\sum_{j \in \ind} \varPhi_j C^{(k)}_j} = \ipr{\varPsi}{\mSy_{\Phi}C^{(k)}} .
$$
Since the inner products at both sides are, by definition, continuous in the topology of $L^1(\vn(\Gamma))$ with respect to each one of their arguments, and since, by Theorem \ref{theo:ncBessel}, $\{\mSy_{\Phi}C^{(k)}\}_{k \in \N}$ converges to $\mSy_{\Phi}C$, then the claim is proved.
\end{proof}

A crucial property of modular Bessel sequences is that they allow us to use inner products arising from the modular analysis operator as coefficients for linear combinations. Indeed, if $\Phi \subset \E$ is a modular Bessel sequence, then the operator $\mFrame_\Phi = \mSy_\Phi \mAn_\Phi$ is bounded on $\E_\Phi$. It reads
$$
\mFrame_\Phi \varPsi = \sum_{j \in \ind} \varPhi_j \ipr{\varPsi}{\varPhi_j}
$$
and we will call it \emph{modular frame operator}.

We can now provide a notion of modular frames as follows.
\begin{defi}
We say that $\Phi = \{\varPhi_j\}_{j \in \ind} \subset \E$ is a \emph{modular frame sequence} with frame bounds $0 < A \leq B < \infty$ if it satisfies
\begin{equation}\label{eq:ncframe}
A \ipr{\varPsi}{\varPsi} \leq \sum_{j \in \ind} |\ipr{\varPsi}{\varPhi_j}|^2 \leq B \ipr{\varPsi}{\varPsi} \quad \forall \, \varPsi \in \E_\Phi,
\end{equation}
where the convergence of the series in the middle term is in  $L^1(\vn(\Gamma))$.
\end{defi}

Similarly to what we have done with Riesz systems, by applying the trace to (\ref{eq:ncframe}) one can obtain a numerical necessary condition for modular frames.
\begin{lem}\label{lem:ncframe}
Let $\Phi = \{\varPhi_j\}_{j \in \ind} \subset \E$ be a modular frame sequence. Then
\begin{equation}\label{eq:ncframenumerical}
A \|\varPsi\|^2_\E \leq \langle \varPsi, \mFrame_\Phi \varPsi\rangle_\E \leq B \|\varPsi\|^2_\E \quad \forall \, \varPsi \in \E_\Phi .
\end{equation}
\end{lem}
\begin{proof}
By applying the trace to the inequalities (\ref{eq:ncframe}), one obtains
$$
A \|\varPsi\|^2_\E \leq \|\mAn_\Phi \varPsi\|^2_{\ell_2(\ind,L^2(\vn(\Gamma)))} \leq B \|\varPsi\|^2_\E
$$
so that (\ref{eq:ncframenumerical}) can be obtained by Theorem  \ref{theo:ncBessel} observing that
\begin{displaymath}
\langle \mAn_\Phi \varPsi, \mAn_\Phi \varPsi\rangle_{\ell_2(\ind,L^2(\vn(\Gamma)))} = \langle \varPsi, \mFrame_\Phi \varPsi\rangle_\E . \qedhere
\end{displaymath}
\end{proof}

\begin{rem}
We note explicitly that Lemma \ref{lem:ncframe} does not imply that a modular frame sequence $\Phi$ is an ordinary frame of the Hilbert space $(\E_\Phi,\langle\cdot,\cdot\rangle_\E)$.
Indeed, this would give us the possibility to express any $\varPsi \in \E_\Phi$ as an ordinary linear combination of $\Phi$ with scalar coefficients, which is not true even for a family $\Phi$ consisting of only one element. In particular, one has the inequality
$$
\Frame_\Phi \leq \mFrame_\Phi \quad \textnormal{on} \ \E_\Phi .
$$
Indeed, by Minkowski and Holder's inequalities and the finiteness of $\tau$ we obtain\vspace{-4pt}
$$
\begin{array}{rcl}
\langle \varPsi, \Frame_\Phi \varPsi\rangle_\E & = & \displaystyle\sum_{j \in \ind} |\tau(\ipr{\varPsi}{\varPhi_j})|^2 \leq \tau\bigg(\Big(\sum_{j \in \ind} |\ipr{\varPsi}{\varPhi_j}|^2\Big)^\frac12\bigg)^2 = \|\ipr{\varPsi}{\mFrame_\Phi \varPsi}^\frac12\|^2_1\vspace{2pt}\\
& \leq & \displaystyle\|\ipr{\varPsi}{\mFrame_\Phi \varPsi}^\frac12\|^2_2 = \tau(\ipr{\varPsi}{\mFrame_\Phi \varPsi}) = \langle \varPsi, \mFrame_\Phi \varPsi\rangle_\E \quad \forall \, \varPsi \in \E_\Phi .\vspace{-4pt}
\end{array}
$$
On the other hand, the fact that $\ol{\vsp_\C \ \Phi}^\E$ is in general strictly smaller than $\E_\Phi$, so that modular frames are not necessarily ordinary frames, implies that the reversed inequality does not hold in general.
\end{rem}

\subsection{Reproducing properties in \texorpdfstring{$L^2(\vn(\Gamma))$}{L2}-Hilbert modules}

In this subsection we will obtain the reproducing properties of modular Riesz and frame sequences. They are based on a general argument that allows to see that the numerical conditions (\ref{eq:ncRiesznumerical}) and (\ref{eq:ncframenumerical}) are not only necessary, but also sufficient for having respectively modular Riesz and modular frame sequences. This relies on a notion of invariance on the introduced modular structure.

Let us define the following action of $\vn(\Gamma)$ on $L^1(\vn(\Gamma))$
\begin{equation}\label{eq:action}
\begin{array}{rccl}
\mathcal{A} : & \vn(\Gamma) \times L^1(\vn(\Gamma)) & \to & L^1(\vn(\Gamma))\\
& (p, F) & \mapsto & \mathcal{A}_p(F) = p^* F p .
\end{array}
\end{equation}
The map $\mathcal{A}$ satisfies $\mathcal{A}_{p_1}\mathcal{A}_{p_2} = \mathcal{A}_{p_2 p_1}$ for all $p_1, p_2 \in \vn(\Gamma)$ and, for any fixed $p \in \vn(\Gamma)$, the operator $\mathcal{A}_p$ is continuous in $L^1(\vn(\Gamma))$.

\begin{defi}
By $\vn(\Gamma)$-\emph{conjugacy set} we mean a set of selfadjoint operators $\Upsilon \subset L^1(\vn(\Gamma))$ that is invariant under the action (\ref{eq:action}), i.e.
$$
\mathcal{A}_p \Upsilon \subset \Upsilon \quad \forall \, p \in \vn(\Gamma) .
$$
\end{defi}

The property that we will need about these conjugacy sets is that positivity of all their elements can be established in terms of the positivity of the traces, which is equivalent to \cite[Lemma 2.3]{BHP14} when applied to $\vn(\Gamma)$-modules.

\begin{lem}\label{lem:invariance}
Let $\Upsilon$ be an $\vn(\Gamma)$-conjugacy set. Then the following are equivalent
\begin{itemize}
\item[i.] $\tau(F) \geq 0$ for all $F \in \Upsilon$
\item[ii.] $F \geq 0$ for all $F \in \Upsilon$.
\end{itemize}
\end{lem}
\begin{proof}
We only need to prove that $i.$ implies $ii.$ In order to show this, let
$$
p = \chi_{(-\infty,0)}(F)
$$
be the spectral projection of an $F \in \Upsilon$ on the negative real axis. Assume by contradiction that $ii.$ does not hold, so that $p \neq 0$. Then
$$
\langle F p u, p u\rangle_{\ell_2(\Gamma)} < 0 \quad \forall \, u \in \ell_2(\Gamma).
$$
This means that $p^*Fp < 0$, hence implying that $\tau(p^*Fp) < 0$, but this contradicts $i.$ because $p \in \vn(\Gamma)$ by \cite[Prop. 5.3.4]{Sunder97}, so $p^*Fp \in \Upsilon$.
\end{proof}

\begin{theo}\label{theo:ncnumerical}
Let $\E$ be an $L^2(\vn(\Gamma))$-Hilbert module.
\begin{itemize}
\item[i.] A countable family $\Phi = \{\varPhi_j\}_{j \in \ind}\subset \E$ is a modular Riesz sequence with Riesz bounds $0 < A \leq B < \infty$ if and only if for all $C \in \ell_0(\ind,\vn(\Gamma))$
\begin{equation}\label{eq:ncRiesznumerical2}
A \sum_{j \in \ind} \tau(|C_j|^2) \leq \tau(\ipr{\mSy_{\Phi}C}{\mSy_{\Phi}C}) \leq B \sum_{j \in \ind} \tau(|C_j|^2)
\end{equation}
where $\mSy_{\Phi}$ is the modular synthesis operator associated to $\Phi$.
\item[ii.] A countable family $\Phi = \{\varPhi_j\}_{j \in \ind}\subset \E$ is a modular frame sequence with frame bounds $0 < A \leq B < \infty$ if and only if for all $\varPsi \in \E_\Phi$ it holds
\begin{equation}\label{eq:ncframenumerical2}
A \tau(\ipr{\varPsi}{\varPsi}) \leq \sum_{j \in \ind} \tau(|\ipr{\varPsi}{\varPhi_j}|^2) \leq B \tau(\ipr{\varPsi}{\varPsi}) .
\end{equation}
\end{itemize}
\end{theo}

\begin{proof}
To prove $i.$, observe first that if $\Phi$ is a modular Riesz sequence then (\ref{eq:ncRiesznumerical2}) holds by Lemma \ref{lem:ncRiesz}. Assume now (\ref{eq:ncRiesznumerical2}). Then, by Lemma \ref{lem:invariance}, the proof is concluded if we can show that for any $\alpha,\beta \in \R$ the set
$$
\Upsilon_1 = \Big\{F \in L^1(\vn(\Gamma)) \, | \, F = \alpha\ipr{\mSy_{\Phi}C}{\mSy_{\Phi}C} + \beta \sum_{j \in \ind} |C_j|^2 \, , \ C \in \ell_0(\ind,\vn(\Gamma))\Big\}
$$
is an $\vn(\Gamma)$-conjugacy set. Since $\Upsilon_1$ is a space of selfadjoint $L^1(\vn(\Gamma))$ operators, we only need to check the $\mathcal{A}$-invariance, which holds true by the modular invariance of the inner product $\ipr{\cdot}{\cdot}$. Indeed, if $C \in \ell_0(\ind,\vn(\Gamma))$, we have
$$
\mathcal{A}_p \Big(\sum_{j \in \ind} |C_j|^2\Big) = \sum_{j \in \ind} |C_j p|^2 \ , \ \ \mathcal{A}_p \left(\ipr{\mSy_{\Phi}C}{\mSy_{\Phi}C}\right) = \ipr{\mSy_{\Phi}Cp}{\mSy_{\Phi}Cp}
$$
so $\Upsilon_1$ is $\mathcal{A}_p$ invariant because $Cp \in \ell_0(\ind,\vn(\Gamma))$.

To prove $ii.$, by Lemma \ref{lem:ncframe} if $\Phi$ is a modular frame sequence then (\ref{eq:ncframenumerical2}) holds. Assume now (\ref{eq:ncframenumerical2}). The right inequality reads equivalently
$$
\Big\| \sum_{j \in \ind} |\ipr{\varPsi}{\varPhi_j}|^2\Big\|_1 \leq B \|\varPsi\|_\E^2
$$
so the sum $\sum_{j \in \ind} |\ipr{\varPsi}{\varPhi_j}|^2$ converges in $L^1(\vn(\Gamma))$. For any $\alpha,\beta \in \R$ the set
$$
\Upsilon_2 = \Big\{F \in L^1(\vn(\Gamma)) \, | \, F = \alpha\sum_{j \in \ind} |\ipr{\varPsi}{\varPhi_j}|^2 + \beta \ipr{\varPsi}{\varPsi} \, , \ \varPsi \in \E_\Phi\Big\}
$$
is then a well defined set of selfadjoint $L^1(\vn(\Gamma))$ operators. By Lemma \ref{lem:invariance}, the proof is then finished if we show that it is $\mathcal{A}$-invariant. Again, this holds true by the modular invariance of the inner product $\ipr{\cdot}{\cdot}$ and the continuity of $\mathcal{A}_p$, because
\begin{align*}
\mathcal{A}_p \Big(\sum_{j \in \ind} |\ipr{\varPsi}{\varPhi_j}|^2 \Big) & = \sum_{j \in \ind} \mathcal{A}_p(|\ipr{\varPsi}{\varPhi_j}|^2) = \sum_{j \in \ind} p^*\ipr{\varPhi_j}{\varPsi}\ipr{\varPsi}{\varPhi_j}p\\
& = \sum_{j \in \ind} \ipr{\varPhi_j}{\varPsi p}\ipr{\varPsi p}{\varPhi_j} = \sum_{j \in \ind} |\ipr{\varPsi p}{\varPhi_j}|^2
\end{align*}
and $\varPsi p \in \E_\Phi$ because $\E_\Phi$ is an $\vn(\Gamma)$-module.
\end{proof}

An important consequence of this theorem is the following characterizations of modular Riesz and frame sequences in terms of the modular Gram operator.
\begin{theo}\label{theo:ncGramianCharacterization}
Let $\Phi = \{\varPhi_j\}_{j \in \ind}\subset \E$ be a modular square integrable countable family, and let $\mGrop_\Phi : \ell_0(\ind,\vn(\Gamma)) \to \ell_2(\ind,L^2(\vn(\Gamma)))$ be the associated modular Gram operator (see Definition \ref{def:modularAG} and Theorem \ref{theo:modulardenseGram}). Then the sequence $\Phi$ is
\begin{itemize}
\item[i.] a modular Riesz sequence with Riesz bounds $0 \!<\! A\! \leq \!B \!< \infty$ if and only if
\begin{equation}\label{eq:ncRieszGram}
A \Id_{\ell_2(\ind,L^2(\vn(\Gamma)))} \leq \mGrop_\Phi \leq B \Id_{\ell_2(\ind,L^2(\vn(\Gamma)))}
\end{equation}
where $\Id_{\ell_2(\ind,L^2(\vn(\Gamma)))}$ stands for the identity operator on $\ell_2(\ind,L^2(\vn(\Gamma)))$.
\item[ii.] a modular frame sequence with frame bounds $0 \!<\! A\! \leq \!B \!< \infty$ if and only if
\begin{equation}\label{eq:ncframeGram}
A \Proj_{V_\Phi} \leq \mGrop_\Phi \leq B \Proj_{V_\Phi}
\end{equation}
where $\Proj_{V_\Phi}$ stands for the orthogonal projection in $\ell_2(\ind,L^2(\vn(\Gamma)))$ onto the submodule $V_\Phi = \ol{\Ran(\mAn_\Phi)} = \Ker(\mSy_\Phi)^\bot$.
\end{itemize}
\end{theo}

\begin{proof}
To prove $i.$ assume first that $\Phi$ is a modular Riesz sequence. Then, using Lemma \ref{lem:Rieszmodadj}, the inequalities in (\ref{eq:ncRiesznumerical2}) read equivalently
\begin{equation}\label{eq:ncRieszGramequiv}
A \|C\|^2_{\ell_2(\ind,L^2(\vn(\Gamma)))} \leq \langle C , \mGrop_{\Phi}C\rangle_{\ell_2(\ind,L^2(\vn(\Gamma)))} \leq B \|C\|^2_{\ell_2(\ind,L^2(\vn(\Gamma)))}
\end{equation}
which proves (\ref{eq:ncRieszGram}). Conversely, if we assume (\ref{eq:ncRieszGram}), then by Lemma \ref{cor:modularWellDefined} the modular synthesis operator $\mSy_\Phi$ is bounded on $\ell_2(\ind,L^2(\vn(\Gamma)))$, and $\mAn_\Phi$ is its bounded adjoint. Thus condition (\ref{eq:ncRieszGram}), which can be written equivalently as (\ref{eq:ncRieszGramequiv}), implies (\ref{eq:ncRiesznumerical2}), and hence $\Phi$ is a modular Riesz sequence.

The proof of $ii.$ can be carried on similarly, with the aid of Lemma \ref{lem:duallemma}. Assume first that $\Phi$ is a modular frame sequence. Then point $ii.$ of Theorem \ref{theo:ncnumerical} together with point $ii.$ of Theorem \ref{theo:ncBessel} implies that
\begin{equation}\label{eq:ncframeOp}
A \Id_{\E_\Phi} \leq \mFrame_\Phi \leq B \Id_{\E_\Phi}
\end{equation}
where $\Id_{\E_\Phi}$ is the identity operator on $\E_\Phi$. This, together with Lemma \ref{lem:duallemma}, imply (\ref{eq:ncframeGram}). Conversely, assume (\ref{eq:ncframeGram}). Then the modular Gram, synthesis and analysis operators are bounded, and this together with Lemma \ref{lem:duallemma} imply (\ref{eq:ncframeOp}). Then, by Theorem \ref{theo:ncnumerical}, $\Phi$ is a modular frame sequence.
\end{proof}

An immediate consequence is the following.
\begin{cor}\label{cor:ncRieszisframe}
Let $\Phi = \{\varPhi_j\}_{j \in \ind} \subset \E$ be a modular Riesz sequence. Then it is a modular frame sequence with the same frame bounds.
\end{cor}

On the other hand, we can also deduce the following modular variant of a classical result, contained e.g. in \cite[Th. 6.1.1]{Christensen03}.

\begin{prop}\label{prop:ncframeRiesz}
Let $\Phi = \{\varPhi_j\}_{j \in \ind} \subset \E$ be a modular frame sequence. Then $\Phi$ is a modular Riesz sequence if and only if $\mSy_\Phi$ is injective.
\end{prop}
\begin{proof}
By Lemma \ref{lem:ncRiesz}, if $\Phi$ is a modular Riesz sequence then $\mSy_\Phi$ is injective (without requiring the modular frame condition, which is actually deduced from Corollary \ref{cor:ncRieszisframe}). To prove the converse, consider a modular frame $\Phi$. By Theorem \ref{theo:ncGramianCharacterization} the modular Gram operator $\mGrop_\Phi$ satisfies (\ref{eq:ncframeGram}). If $\mSy_\Phi$ is injective, then $\Ker(\mSy_\Phi) = 0$, so the projection $\Proj_{V_\Phi}$ coincides with the identity operator, hence providing the modular Riesz condition (\ref{eq:ncRieszGram}).
\end{proof}

\begin{rem}
Riesz sequences were introduced in the modular setting of \cite{FrankLarson2002} as frame sequences satisfying a linear independence condition that is in general weaker than the one involved in the present definition of modular Riesz sequence. Indeed, the notion of linear independence considered in \cite{FrankLarson2002} takes into account an issue of zero divisors that on one side is natural when dealing with operators, but on the other side does not appear in the present setting. Indeed, the injectivity of the modular synthesis operator means that modular Riesz sequences are linearly independent families in $\ell_2(\ind,L^2(\vn(\Gamma)))$, in the sense that $\mSy_\Phi C = 0$ implies $C = 0$ for all $C \in \ell_2(\ind,L^2(\vn(\Gamma)))$.
\end{rem}

We can now address the issue of reproducing formulas. The starting point is the following direct consequence of condition (\ref{eq:ncframeOp}).
\begin{cor}\label{cor:modtight}
A countable family $\Phi = \{\varPhi_j\}_{j \in \ind} \subset \E$ is a modular tight frame sequence, i.e.\! a modular frame sequence with frame bounds $A \!= \!B$, if and only if
$$
\mFrame_\Phi \varPsi = A \varPsi \quad \forall \, \varPsi \in \E_\Phi .
$$
\end{cor}

Proceeding as in \cite[\S 8.1]{HW96}, we obtain general reproducing formulas for frame and Riesz sequences by means of a modular notion of canonical dual frame.

\begin{prop}
Let $\Phi = \{\varPhi_j\}_{j \in \ind} \subset \E$ be a modular frame with frame bounds $0 < A \leq B < \infty$, and let us denote with $\mathring\Phi = \{\mathring\varPhi_j = \mFrame_\Phi^{-1}\varPhi_j\}_{j \in \ind}$. Then
\begin{itemize}
\item[i.] $\E_\Phi = \E_{\mathring\Phi}$
\item[ii.] the system $\mathring\Phi$ is a modular frame with frame bounds $0 < \frac{1}{B} \leq \frac{1}{A} < \infty$, that we will call the \emph{canonical modular dual} of $\Phi$.
\end{itemize}
\end{prop}
\begin{proof}
To prove $i.$, observe first that since $\mFrame_\Phi$ is bounded and invertible on $\E_\Phi$, then $\mathring\Phi_j \in \E_\Phi$ for all $j$, so that $\E_{\mathring\Phi} \subset \E_\Phi$. On the other hand, starting from (\ref{eq:closedmodularspan}), and using that $\mFrame_\Phi^{-1}$ is $\vn(\Gamma)$-linear, we have
\begin{align*}
\E_\Phi & = \mFrame_\Phi^{-1} \E_\Phi = \mFrame_\Phi^{-1} (\ol{\vsp_{\vn(\Gamma)}\ \Phi}^\E) \subset \ol{\mFrame_\Phi^{-1}(\vsp_{\vn(\Gamma)}\, \Phi)}^\E\\
& = \ol{\vsp_{\vn(\Gamma)}\, \mFrame_\Phi^{-1}\Phi}^\E = \ol{\vsp_{\vn(\Gamma)}\, \mathring\Phi}^\E = \E_{\mathring\Phi} .
\end{align*}
To prove $ii.$, since $\mFrame_\Phi$ is bounded and invertible on $\E_\Phi$, its inverse $\mFrame_\Phi^{-1}$ satisfies
$$
\frac{1}{B} \Id_{\E_\Phi} \leq \mFrame_\Phi^{-1} \leq \frac{1}{A} \Id_{\E_\Phi} .
$$
Now it suffices to show that
\begin{equation}\label{eq:HWframes}
\sum_{j \in \ind} |\ipr{\varPsi}{\mathring\varPhi_j}|^2 = \ipr{\mFrame_\Phi^{-1}\varPsi}{\varPsi} \quad \forall \, \varPsi \in \E_\Phi \, ,
\end{equation}
because, if this holds, by applying the trace at both sides we have
$$
\sum_{j \in \ind} \tau(|\ipr{\varPsi}{\mathring\varPhi_j}|^2) = \langle \mFrame_\Phi^{-1}\varPsi, \varPsi\rangle_{\E} \quad \forall \, \varPsi \in \E_\Phi
$$
and hence
$$
\frac{1}{B} \|\varPsi\|_\E^2 \leq \sum_{j \in \ind} \tau(|\ipr{\varPsi}{\mathring\varPhi_j}|^2) \leq \frac{1}{A} \|\varPsi\|_\E^2 \quad \forall \, \varPsi \in \E_\Phi \, ,
$$
that, by Theorem \ref{theo:ncnumerical}, is equivalent to the modular frame condition.
\newpage
In order to prove (\ref{eq:HWframes}) we first observe that, since $\mFrame_\Phi$ is modular selfadjoint by Lemma \ref{lem:modadj}, then also $\mFrame_\Phi^{-1}$ is modular selfadjoint, in the sense that $\ipr{\varPsi_1}{\mFrame_\Phi^{-1}\varPsi_2} = \ipr{\mFrame_\Phi^{-1}\varPsi_1}{\varPsi_2}$ for all $\varPsi_1, \varPsi_2 \in \E_\Phi$, because
$$
\ipr{\varPsi_1}{\mFrame_\Phi^{-1}\varPsi_2} = \ipr{\mFrame_\Phi\mFrame_\Phi^{-1}\varPsi_1}{\mFrame_\Phi^{-1}\varPsi_2} = \ipr{\mFrame_\Phi^{-1}\varPsi_1}{\mFrame_\Phi\mFrame_\Phi^{-1}\varPsi_2} = \ipr{\mFrame_\Phi^{-1}\varPsi_1}{\varPsi_2}.
$$
This implies in particular that
$$
(\mAn_{\mathring\Phi}\varPsi)_j = \ipr{\varPsi}{\mathring\varPhi_j} = \ipr{\varPsi}{\mFrame_\Phi^{-1}\varPhi_j} = \ipr{\mFrame_\Phi^{-1}\varPsi}{\varPhi_j} = (\mAn_\Phi \mFrame_\Phi^{-1}\varPsi)_j .
$$
Now, by making use of Lemma \ref{lem:modadj}, we have
\begin{align*}
\sum_{j\in\ind} |\ipr{\varPsi}{\mathring\varPhi_j}|^2 & = \ipr{\mAn_{\mathring\Phi}\varPsi}{\mAn_{\mathring\Phi}\varPsi}_2 = \ipr{\mAn_\Phi \mFrame_\Phi^{-1}\varPsi}{\mAn_\Phi \mFrame_\Phi^{-1}\varPsi}_2\\
& = \ipr{\mFrame_\Phi^{-1}\varPsi}{\mSy_\Phi \mAn_\Phi \mFrame_\Phi^{-1}\varPsi} = \ipr{\mFrame_\Phi^{-1}\varPsi}{\varPsi}
\end{align*}
which proves (\ref{eq:HWframes}).
\end{proof}

The notion of canonical dual allows to obtain a general reproducing formula for modular frame sequences (hence also for modular Riesz sequences).
\begin{theo}\label{theo:reproducing}
Let $\Phi = \{\varPhi_j\}_{j \in \ind} \subset \E$ be a modular frame sequence, and let $\mathring\Phi = \{\mathring\varPhi_j\}_{j \in \ind}$ be its canonical modular dual. Then
$$
\varPsi = \sum_{j \in \ind} \varPhi_j \ipr{\varPsi}{\mathring\varPhi_j} = \sum_{j \in \ind} \mathring\varPhi_j \ipr{\varPsi}{\varPhi_j} \quad \forall \, \varPsi \in \E_\Phi .
$$
\end{theo}
\begin{proof}
Let us define $\ipr{\cdot}{\cdot}_\sharp : \E_\Phi \times \E_\Phi \to L^1(\vn(\Gamma))$ as $\ipr{\varPsi_1}{\varPsi_2}_\sharp = \ipr{\mFrame_\Phi^{-1}\varPsi_1}{\varPsi_2}$.
Such a map provides another inner product on $\E_\Phi$, in terms of which (\ref{eq:HWframes}) reads
$$
\sum_{j \in \ind} |\ipr{\varPsi}{\varPhi_j}_\sharp|^2 = \ipr{\varPsi}{\varPsi}_\sharp
$$
so that both $\Phi$ and $\mathring\Phi$ are modular tight frames with constant 1 with respect to it. The conclusion now follows from Corollary \ref{cor:modtight}.
\end{proof}

\subsection{Dimensionality reduction in \texorpdfstring{$(\Gamma,\Pi)$}{shift}-invariant spaces}

This section is devoted to prove a characterization of Riesz and frame sequences of $(\Gamma,\Pi)$-invariant systems in terms of reduced conditions, in operator spaces, on the Helson map of their generators. This characterization provides a generalization to nonabelian groups and general unitary representations of the results of \cite{RonShen95, Bownik00, CabrelliPaternostro10}.

The proof of Theorem \ref{theo:main} is based on the following result, which basically relies on Plancherel Theorem between $\ell_2(\Gamma)$ and $L^2(\vn(\Gamma))$.
\begin{prop}\label{prop:equivalence}
Under the same assumptions of Theorem \ref{theo:main}, the system $E$ is
\begin{itemize}
\item[i.] a Riesz sequence with Riesz bounds $0 < A \leq B < \infty$ if and only if
$$
\hspace{-10pt} A \tau\Big(\sum_{j \in \ind} |C_j|^2\Big) \leq \tau\Big(\ipr{\mSy_{\Phi}C}{\mSy_{\Phi}C}_\Kil\Big) \leq B \tau\Big(\sum_{j \in \ind} |C_j|^2\Big) \quad \forall \ C \in \ell_0(\ind,\vn(\Gamma))
$$
\item[ii.] a frame sequence with frame bounds $0 < A \leq B < \infty$ if and only if
$$
A \tau([f,f]) \leq \sum_{j \in \ind} \tau(|[f, \phi_j]|^2) \leq B \tau([f,f]) \quad \forall \ f \in \csp E .
$$
\end{itemize}
\end{prop}
\newpage
\begin{proof}
To prove $i.$, let $C \in \ell_0(\ind,\vn(\Gamma))$ be given by $C_j = \sum_{\gamma \in \Gamma} c_j(\gamma) \rr(\gamma)^*$ for a given sequence $\{c_j(\gamma)\}_{\myatop{\gamma \in \Gamma}{j \in \ind}} \in \ell_0(\ind,\ell_2(\Gamma))$. By Plancherel Theorem we have that
$$
\sum_{\myatop{\gamma \in \Gamma}{j \in \ind}} |c_j(\gamma)|^2 = \sum_{j \in \ind} \tau(|C_j|^2)
$$
while
\begin{align*}
\bigg\|\sum_{\myatop{\gamma \in \Gamma}{j \in \ind}} c_j(\gamma) \Pi(\gamma)\phi_j\bigg\|_\Hil^2 & = \normbig{\iso\Big[\sum_{{\myatop{\gamma \in \Gamma}{j \in \ind}}} c_j(\gamma) \Pi(\gamma)\phi_j\Big]}^2 = \normbig{\sum_{j \in \ind} \iso[\phi_j]C_j}^2 .
\end{align*}
Now the conclusion follows because, by (\ref{eq:Kilnormidentity}), we have
$$
\norm{F}^2 = \tau(\ipr{F}{F}_\Kil) .
$$
To prove $ii.$, we only need to observe that, by the definition of bracket map and Plancherel Theorem in $L^2(\vn(\Gamma))$, for all $f \in \csp\{E\}$ we have
\begin{displaymath}
\sum_{\gamma \in \Gamma} |\langle f, \Pi(\gamma)\phi\rangle_\Hil|^2 = \sum_{\gamma \in \Gamma} |\tau(\rr(\gamma)[f,\phi])|^2 = \tau \big(|[f,\phi]|^2\big) . \qedhere
\end{displaymath}
\end{proof}

\begin{proof}[Proof of Theorem \ref{theo:main}]
It suffices to combine Proposition \ref{prop:equivalence} with Theorem \ref{theo:ncnumerical}, recalling the definition of the inner product in $\Kil$ given by (\ref{eq:Kilinnerproduct}).
\end{proof}

\begin{rem}
Observe in particular that, by (\ref{eq:Kilinnerproduct}) and (\ref{eq:Helsonbracket}), we can explicitly write the inner product of $\Kil$ as
$$
\ipr{\varPhi}{\varPsi}_\Kil = \int_\M \varPsi(x)^* \varPhi(x) d\nu(x) \quad \forall \ \varPhi, \varPsi \in \Kil .
$$
\end{rem}

The frame operator of a $(\Gamma,\Pi)$-invariant system can actually be written in terms of the modular frame operator associated to its generators and the Helson map, as follows.
\begin{cor}\label{cor:dualGramian}
Let $\{\phi_j\}_{j \in \ind} \subset \Hil$ be a countable family and let $(\Gamma,\Pi,\Hil)$ be a dual integrable triple. Suppose that $E = \{\Pi(\gamma)\phi_j\}_{{\myatop{\gamma \in \Gamma}{j \in \ind}}} \subset \Hil$ is a Bessel sequence. Then, calling $\Phi = \{\iso[\phi_j]\}_{j \in \ind} \subset \Kil$, we have
$$
\mFrame_{\Phi} = \iso \Frame_{E} \iso^{-1} .
$$
\end{cor}
\begin{proof}
The frame operator of $E$ on $\Hil$ is the bounded operator
$$
\Frame_{E} \psi = \sum_{{\myatop{\gamma \in \Gamma}{j \in \ind}}} \langle \psi, \Pi(\gamma)\phi_j\rangle_\Hil \Pi(\gamma)\phi_j
$$
so its composition with $\iso$ is also bounded, and reads
$$
\iso \Frame_{E} \psi = \sum_{{\myatop{\gamma \in \Gamma}{j \in \ind}}} \langle \psi, \Pi(\gamma)\phi_j\rangle_\Hil \iso[\Pi(\gamma)\phi_j] = \sum_{j \in \ind} \iso[\phi_j] \sum_{\gamma \in \Gamma} \langle \psi, \Pi(\gamma)\phi_j\rangle_\Hil \rr(\gamma)^* .
$$
The Bessel condition implies that $\{\langle \psi, \Pi(\gamma)\phi_j\rangle_\Hil\}_{\gamma \in \Gamma} \in \ell_2(\Gamma)$, so
\begin{displaymath}
\iso \Frame_{E} \psi = \sum_{j \in \ind} \iso[\phi_j] [\psi,\phi_j] = \sum_{j \in \ind} \iso[\phi_j] \ipr{\iso[\psi]}{\iso[\phi_j]} = \mFrame_{\Phi}\iso[\psi]. \qedhere
\end{displaymath}
\end{proof}

\begin{rem}\label{rem:dualGramian}
When considering integer translations, the same argument used to relate the fiberization mapping (\ref{eq:Fourierperiodization}) with the Helson map (\ref{eq:integerisometry}) allows to see that the modular frame operator described by Corollary \ref{cor:dualGramian} is related to the so-called \emph{dual Gramian} operator of \cite{BoorVoreRon93, BoorVoreRon94, RonShen95, Bownik00} by the Pontryagin isomorphism between the character group and the von Neumann algebra.
Analogously, the modular Gram operator over the $L^2(\vn(\Gamma))$-Hilbert module $\Kil$ obtained from the Helson map (\ref{eq:integerisometry}) coincides, under Pontryagin isomorphism, with the so-called \emph{Gramian} matrix.
This shows in particular that Theorem \ref{theo:ncGramianCharacterization}, together with Theorem \ref{theo:main}, provide a direct generalization of results such as \cite[Th. 3.2.3, Th. 3.5.3]{RonShen95} \cite[Th. 2.5]{Bownik00}, \cite[Prop. 4.9]{CabrelliPaternostro10}.
\end{rem}

\vspace{20pt}
\noindent
\textbf{Davide Barbieri}\\
Universidad Aut\'onoma de Madrid, 28049 Madrid, Spain\\
\href{mailto:davide.barbieri@uam.es}{\tt davide.barbieri@uam.es}

\

\noindent
\textbf{Eugenio Hern\'andez}\\
Universidad Aut\'onoma de Madrid, 28049 Madrid, Spain\\
\href{mailto:eugenio.hernandez@uam.es}{\tt eugenio.hernandez@uam.es}

\

\noindent
\textbf{Victoria Paternostro}\\
Universidad de Buenos Aires and 
IMAS-CONICET, Consejo  Nacional de Investigaciones Cient\'ificas y T\'ecnicas, 1428 Buenos Aires,  Argentina\\
\href{mailto:vpater@dm.uba.ar}{\tt vpater@dm.uba.ar}


\begin{thebibliography}{99}

\bibitem{AldroubiCabrelliMolter08} A. Aldroubi, C. Cabrelli, U. Molter, \emph{Optimal non-linear models for sparsity and sampling}.
J. Fourier Anal. Appl. 14:793-812 (2008)

\bibitem{AldroubiGrochenig01} A. Aldroubi, K. Grochenig, \emph{Nonuniform Sampling and Reconstruction in Shift-Invariant Spaces}. SIAM Rev. 43:585–620 (2001)

\bibitem{Aronszajn50} N. Aronszajn, \emph{Theory of reproducing kernels}. Trans. Amer. Math. Soc. 68:337-404 (1950).

\bibitem{BHM14} D. Barbieri, E. Hern\'andez, A. Mayeli, \emph{Bracket map for the Heisenberg group and the characterization of cyclic subspaces}. Appl. Comput. Harmon. Anal. 37:218-234 (2014).

\bibitem{BHP14} D. Barbieri, E. Hern\'andez, J. Parcet, \emph{Riesz and frame systems generated by unitary actions of discrete groups}. To appear on Appl. Comput. Harmon. Anal. (2014).

\bibitem{Zak-SIS} D. Barbieri, E. Hern\'andez, V. Paternostro, \emph{The Zak transform and the structure of spaces invariant by the action of an LCA group}. To appear on J. Funct. Anal. (2015).

\bibitem{BenedettoLi93} J. J. Benedetto, S. Li, \emph{Multiresolution analysis frames with applications}. ``ICASSP'93'', Minneapolis, III:304-307  (1993).

\bibitem{BenedettoLi98} J. J. Benedetto, S. Li, \emph{The theory of multiresolution analysis frames and applications to filter banks}. Appl. Comput. Harmon. Anal. 5:389-427 (1998).

\bibitem{BlecherLeMerdy04} P. Blecher, C. Le Merdy, \emph{Operator algebras and their modules. An operator space approach.} Oxford University Press 2004.

\bibitem{KostrikinShafarevich91} L. A. Bokhut', I. V. L'vov, V. K. Kharchenko, \emph{Noncommutative rings}. In A. I. Kostrikin, I. R. Shafarevich (eds.), \emph{Algebra II}. Springer 1991.

\bibitem{BoorVoreRon93} C. de Boor, R. A. DeVore, A. Ron, \emph{Approximation from shift invariant subspaces of $L^2(\R^d)$}. Trans. Amer. Math. Soc. 341:787-806 (1994).

\bibitem{BoorVoreRon94} C. de Boor, R. A. DeVore, A. Ron, \emph{The structure of finitely generalized shift-invariant spaces in $L^2(\R^d)$}. J. Funct. Anal. 119:37-78 (1994).

\bibitem{Bownik00} M. Bownik, \emph{The structure of shift-invariant subspaces of $L^2(R^n)$}. J. Funct. Anal. 177 (2):282-309 (2000).

\bibitem{BownikRoss13} M. Bownik, K. A. Ross, \emph{The structure of translation-invariant spaces on locally compact abelian groups}. To appear on J. Fourier Anal. Appl. (2015).

\bibitem{CabrelliPaternostro10} C. Cabrelli, V. Paternostro, \emph{Shift-invariant spaces on LCA groups}. J. Funct. Anal. 258:2034-2059 (2010).

\bibitem{CabrelliPaternostro11} C. Cabrelli, V. Paternostro, \emph{Shift-modulation invariant spaces on LCA groups}. Studia Math. 211:1-19 (2011).

\bibitem{CaiDongOsherShen10} J.-F. Cai, B. Dong, S. Osher, Z. Shen, \emph{Image restoration: Total variation, wavelet frames, and beyond}. J. Amer. Math. Soc. 25:1033-1089 (2012).

\bibitem{CasazzaLammers99} P. G. Casazza, M. C. Lammers, \emph{Bracket products for Weyl-Heisenberg frames}. In \emph{Advances in Gabor analysis}, H. G. Feichtinger and T. Strohmer (eds.), Chapt. 4, Springer 2003.

\bibitem{Christensen03} O. Christensen, \emph{An introduction to frames and Riesz bases}. Birkh\"auser 2003.

\bibitem{ChristensenEldar05} O. Christensen, Y. C. Eldar, \emph{Generalized Shift-Invariant Systems and Frames for Subspaces}. J. Fourier Anal. Appl. 11:299-313 (2005).

\bibitem{Connes94} A. Connes, \emph{Noncommutative geometry}. Academic Press 1994.

\bibitem{Conway90} J. B. Conway, \emph{A course in functional analysis}. Springer $2^{\textnormal{nd}}$ Edition 1990.

\bibitem{Conway00} J. B. Conway, \emph{A course in operator theory}. AMS 2000.

\bibitem{CurreyMayeliOussa14} B. Currey, A. Mayeli, V. Oussa, \emph{Characterization of Shift-Invariant Spaces on a Class of Nilpotent Lie Groups with Applications}. J. Fourier Anal. Appl. 20:384-340 (2014).

\bibitem{DaiLarson98} X. Dai, D. R. Larson, \emph{Wandering vectors for unitary systems and orthogonal wavelets}. Memoirs AMS 1998.

\bibitem{Daubechies92} I. Daubechies, \emph{Ten lectures on wavelets}. SIAM 1992.

\bibitem{DaubechiesHanRonShen03} I. Daubechies, B. Han, A. Ron, Z. Shen, \emph{Framelets: MRA-based constructions of wavelet frames}. Appl. Comp. Harmon. Anal. 14:1-46 (2003).

\bibitem{DevoreLorentz} R. A. DeVore, G. G. Lorentz, \emph{Constructive approximation}. Springer 1993.


\bibitem{DutkayHanLarson09} D. Dutkay, D. Han, D. R. Larson, \emph{A duality principle for groups}. J. Funct. Anal. 257:1133-1143 (2009).

\bibitem{EnockSchwartz92} M. Enock, J. M. Schwartz, \emph{Kac algebras and duality of locally compact groups}. Springer 1992.

\bibitem{Falcone00} T. Falcone, \emph{$L^2$-von Neumann modules, their relative tensor products and the spatial derivative}. Illinois J. Math. 44:407-437 (2000).

\bibitem{Folland95} G. B. Folland, \emph{A course in abstract harmonic analysis}. CRC Press 1995.

\bibitem{FrankBiblio} M. Frank, \emph{Hilbert $C^*$-modules and related subjects - a guided reference overview}. Available at \href{http://www.imn.htwk-leipzig.de/~mfrank/hilmod.html}{www.imn.htwk-leipzig.de/$\sim$mfrank/hilmod.html} 

\bibitem{FrankLarson2002} M. Frank, D. R. Larson, \emph{Frames in Hilbert $C^*$-modules and $C^*$-algebras}. J. Operator Theory 48:273-314 (2002).

\bibitem{FuhrXian14} H. F\"uhr, J. Xian, \emph{Relevant sampling in finitely generated shift-invariant spaces}. Preprint 2014. Available at \href{http://arxiv.org/abs/1410.4666}{arxiv.org/abs/1410.4666}

\bibitem{GarciaMedinaVillalon08} A. G. Garc\'ia, M. A. Hern\'andez-Medina, G. P\'erez-Villal\'on, \emph{Generalized sampling in shift-invariant spaces with multiple stable generators}. J. Math. Anal. Appl. 337:69-84 (2008)

\bibitem{HanJingLarsonMohapatra08} D. Han, W. Jing, D. R. Larson, R. N. Mohapatra, \emph{Riesz bases and their dual modular frames in Hilbert $C^∗$-modules}. J. Math. Anal. Appl. 343:246-256 (2008).

\bibitem{Helson64} H. Helson, \emph{Lectures on invariant subspaces}. Academic Press 1964.

\bibitem{Helson86} H. Helson, \emph{The spectral theorem}. Springer 1986.

\bibitem{HLWW02} E. Hern\'andez, D. Labate, G. Weiss, E. Wilson, \emph{A unified characterization of reproducing systems generated by a finite family, II}. J. Geom. Anal. 12:615-662 (2002).

\bibitem{HSWW10} E. Hern\'andez, H. \v{S}iki\'{c}, G. Weiss, E. Wilson, \emph{Cyclic subspaces for unitary representations of LCA groups; generalized Zak transform}. Colloq. Math. 118:313-332 (2010).

\bibitem{HW96} E. Hern\'andez, G. Weiss, \emph{A first course on wavelets}. CRC Press 1996.

\bibitem{Iverson14} J. W. Iverson, \emph{Subspaces of $L^2(G)$ invariant under translations by an abelian subgroup}. J. Funct. Anal. 269:865-913 (2015).

\bibitem{JakobsenLemvig2014} M. S. Jakobsen, J. Lemvig, \emph{Reproducing formulas for generalized translation invariant systems on locally compact abelian groups}. To appear on Trans. Amer. Math. Soc. (2015)

\bibitem{Janssen82} A. J. E. M. Janssen, \emph{Bargmann transform, Zak transform, and coherent states}. J. Math. Phys. 23:720-731 (1982).

\bibitem{Jing06} W. Jing, \emph{Frames in $C^*$-modules}. PhD Thesis (2006).

\bibitem{JungeSherman05} M. Junge, D. Sherman, \emph{Noncommutative $L^p$ modules}. J. Operator Theory 53:3-34 (2005).

\bibitem{KadisonRingrose83} R. V. Kadison, J. R. Ringrose, \emph{Fundamentals of the theory of operator algebras}, Vol.1 and Vol. 2. Academic Press 1983.

\bibitem{KaftalLarsonZhang09} V. Kaftal, D. R. Larson, S. Zhang, \emph{Operator-valued frames}. Trans. Amer. Math. Soc. 361:6349-6385 (2009). 

\bibitem{KamyabiRaisi08} R. A. Kamyabi Gol, R. Raisi Tousi, \emph{The structure of shift invariant spaces on a locally compact abelian group}. J. Math. Anal. Appl. 340:219-225 (2008).

\bibitem{Kaplansky53} I. Kaplansky, \emph{Modules over operator algebras}. Amer. J. Math. 75:839-858 (1953)

\bibitem{Keating98} M. E. Keating, \emph{A first course in module theory}. Imperial College Press 1998.

\bibitem{Lance95} E. C. Lance, \emph{Hilbert $C^*$-modules. A toolkit for operator algebraists}. Cambridge University Press 1995.

\bibitem{Larson97} D. R. Larson, \emph{Von Neumann algebras and wavelets}. In \emph{Operator algebras and applications}, A. Katavolos (ed.), Chapt. 9, Springer 1997.

\bibitem{Luef11} F. Luef, \emph{Projections in noncommutative tori and Gabor frames}. Proc. Amer. Math. Soc. 139:571-582 (2011).

\bibitem{Mallat12} S. Mallat, \emph{Group invariant scattering}. Comm. Pure Appl. Math. 65:1331-1398 (2012).

\bibitem{ManuilovTroitsky05} V. M. Manuilov, E. V. Troitsky, \emph{Hilbert $C^*$-modules}. AMS 2005.

\bibitem{Meyer93} Y. Meyer, \emph{Wavelets and Operators}. Cambridge University Press, 1992.

\bibitem{NashedSun10} M. Z. Nashed, Q. Sun, \emph{Sampling and reconstruction of signals in a reproducing kernel subspace of $L^p$}. J. Funct. Anal. 258:2422-2452 (2010)

\bibitem{Nelson74} E. Nelson, \emph{Notes on Non-commutative integration}. J. Funct. Anal. 15:103-116 (1974).

\bibitem{PackerRieffel04} J. A. Packer, M. A. Rieffel, \emph{Projective multi-resolution analyses for $L^2(\R^2)$}. J. Fourier Anal. Appl. 10:439-464 (2004)

\bibitem{Paschke73} W. L. Paschke, \emph{Inner product modules over $B^*$-algebras}. Trans. Amer. Math. Soc. 182:443-468 (1973).

\bibitem{PisierXu03} G. Pisier, Q. Xu, \emph{Non-Commutative $L^p$ spaces}. In \emph{Handbook of the geometry of Banach spaces}, Vol. 2, W. B. Johnson and J. Lindenstrauss (eds.), Chapt. 34, Elsevier 2003.

\bibitem{Popa86} S. Popa, \emph{Correspondences}. INCREST Preprint 56/1986. Available at \href{http://www.math.ucla.edu/~popa/popa-correspondences.pdf}{www.math.ucla.edu/$\sim$popa/popa-correspondences.pdf}

\bibitem{RaeburnThompson03} I. Raeburn, S. J. Thompson, \emph{Countably generated Hilbert modules, the Kasparov Stabilisation Theorem, and frames in Hilbert modules}. Proc. Amer. Math. Soc. 131:1557-1564 (2003).

\bibitem{Rieffel74} M. A. Rieffel, \emph{Induced representations of $C^*$-algebras}. Adv. Math. 13:176-257 (1974).

\bibitem{RonShen95} A. Ron, Z. Shen, \emph{Frames and stable bases for shift-invariant subspaces of $L^2(R^d)$}. Canad. J. Math. 47:1051-1094 (1995).

\bibitem{RonShen05} A. Ron, Z. Shen, \emph{Generalized shift-invariant systems}. Constr. Approx. 22:1-45 (2005).

\bibitem{Roysland11} K. R\o{}ysland, \emph{Frames generated by actions of countable discrete groups}. Trans. Amer. Math. Soc. 363:95-108 (2011).

\bibitem{Saitoh97} S. Saitoh, \emph{Integral transforms, reproducing kernels and their applications}. CRC Press 1997.

\bibitem{Saliani14} S. Saliani, \emph{Linear independence of translates implies linear independence of affine Parseval frames on LCA groups}. Preprint 2014. Available at \href{http://arxiv.org/abs/1411.1252}{arxiv.org/abs/1411.1252}

\bibitem{Sunder97} V. S. Sunder, \emph{Functional analysis. Spectral theory}. Birkh\"auser 1997.

\bibitem{Takesaki03} M. Takesaki, \emph{Theory of operator algebras}, Vol. 1 and Vol. 2. Springer 2003.

\bibitem{Terp81} M. Terp, \emph{$L^p$-spaces associated with von Neumann Algebras}. Notes, Copenhagen University 1981.

\bibitem{Weil64} A. Weil, \emph{Sur certaines groupes d'operateurs unitaires}. Acta Math. 111:143-211 (1964).

\bibitem{Wegge93} N. E. Wegge-Olsen, \emph{K-Theory and $C^*$-algebras}. Oxford University Press 1993.

\bibitem{Wood04} P. J. Wood, \emph{Wavelets and Hilbert Modules}. J. Fourier Anal. Appl. 10:573-598 (2004).

\bibitem{Zak67} J. Zak, \emph{Finite translations in solid state physics}. Phys. Rev. Lett. 19:1385-1387 (1967).

\end{thebibliography}
\end{document}